\documentclass[oneside, 11pt,a4paper]{amsart}
               
\usepackage[italian,american]{babel}
\usepackage[dvipsnames,table]{xcolor}
\usepackage[utf8]{inputenc}
\usepackage[T1]{fontenc}
\usepackage{amsmath,amssymb,amsthm}
\usepackage{mathrsfs}
\usepackage[none]{hyphenat}
\usepackage{varioref}
\usepackage{setspace}
\usepackage{pinlabel}
\usepackage{chngpage}
\usepackage{calc}
\usepackage{listings}
\usepackage{float}    
\usepackage{comment}
\usepackage{pgf}
\usepackage{microtype}
\usepackage{tikz-cd}
\usepackage{mathtools}
\usetikzlibrary{shapes.geometric}
\usepackage{graphicx}
\usepackage{subcaption}
\usepackage{multicol}

\usepackage{relsize}
\usepackage{lipsum}
\usepackage{faktor}
\usepackage{enumitem}
\usepackage{hyperref}
\usepackage{stmaryrd}
\newtheorem{prop}{\textsc{Proposition}}[section]
\newtheorem{cor}[prop]{\textsc{Corollary}}
\newtheorem{thm}[prop]{\textsc{Theorem}}
\newtheorem{conj}[prop]{\textsc{Conjecture}}
\newtheorem{quest}[prop]{\textsc{Question}}

\newtheorem{lemma}[prop]{\textsc{Lemma}}
\theoremstyle{definition}
\newtheorem{defn}[prop]{\textsc{Definition}}
\newtheorem{remark}[prop]{\textsc{Remark}}

\newcommand{\F}{\mathbb{F}}
\newcommand{\G}{\mathbb{G}}

\newcommand{\Z}{\mathbb{Z}}
\newcommand{\Q}{\mathbb{Q}}
\newcommand{\R}{\mathbb{R}}

\newcommand{\si}{strongly invertible }
\newcommand{\sil}{strongly invertible Legendrian }

\newcommand{\X}{\mathbb{X}}
\newcommand{\Oo}{\mathbb{O}}
\newcommand{\xx}{\mathbf{x}}
\newcommand{\yy}{\mathbf{y}}
\newcommand{\zz}{\mathbf{z}}
\newcommand{\aalpha}{\textbf{\boldmath$\alpha$}}
\newcommand{\bbeta}{\textbf{\boldmath$\beta$}}
\newcommand{\ggamma}{\textbf{\boldmath$\gamma$}}
\newcommand{\hc}{H \mathcal{C}}
\newcommand{\gc}{GC^-(\mathbb{G})}
\newcommand{\gh}{GH^-(\mathbb{G})}
\newcommand{\gcc}{GC^-(\mathbb{G}')}
\newcommand{\ghh}{GH^-(\mathbb{G}')}
\newcommand{\St}{\mathbf{S}(\G)}
\newcommand{\Stt}{\mathbf{S}(\G')}
\newcommand{\cono}{\text{Cone}_\G(Id+\rho)}
\newcommand{\ccono}{\text{Cone}_{\G'}(Id+\rho)}
\newcommand{\hcono}{\widehat{\text{Cone}}_\G(Id+\rho)}
\newcommand{\ii}{\mathfrak{i}}

\let\origlabelitemi\labelitemi
\renewcommand{\labelitemi}{\scalebox{0.8}{\origlabelitemi}}
\setlist[enumerate,1]{label=\roman*.}
\setlist[enumerate,2]{label=\alph*)}
\setlist[enumerate,3]{label=\arabic*.}
\address{University of Pisa, Department of Mathematics, Largo Bruno Pontecorvo 5, 56127, Pisa, Italy}
\email{\href{mailto:giovanni.framba@phd.unipi.it}{giovanni.framba@phd.unipi.it}}
\title[Equivariant Grid Homology]{Equivariant Grid Homology for Strongly Invertible Knots}
\author[Giovanni Framba]{Giovanni Framba}
\begin{document}
	

\begin{abstract}
We study an equivariant version of grid homology for strongly invertible knots. We consider symmetric grid diagrams encoding the strong inversion, prove that every strongly invertible knot admits such a representation, and establish the corresponding analogue of Cromwell's theorem. A symmetric grid induces an involution on the associated grid complex. We study the mapping cone of the sum of the identity and this involution, proving that its homology is an invariant of strongly invertible knots. Finally, we define equivariant analogues of the tau invariant and of the Legendrian grid invariants, and show that the former provide lower bounds for equivariant unknotting numbers and for the genus of simple equivariant cobordisms.
\end{abstract}
\maketitle
\tableofcontents

\setlength{\parindent}{0pt}
\setlength{\parskip}{5pt}

\section{Introduction and First Definitions}
Symmetry has long been a central theme in knot theory, dating back to the earliest knot tables compiled by Tait \cite{tait}. Over time, the study of symmetries has evolved into a rich and specialized area, with particular attention devoted to knots admitting nontrivial involutions. Among these, \emph{strongly invertible knots} play a prominent role. 

A systematic study of such knots was initiated by Sakuma \cite{sakuma}, who introduced fundamental notions such as the equivariant connected sum and the equivariant concordance group. Since then, strongly invertible knots have attracted renewed attention, with significant developments appearing in recent years \cite{boyle2024equivariant,sano2024involutive,collarilisca,hirasawa2022invariant,boyle2023classification,dai_mallick_stoffregen,hirasawa2022invariant,lipshitz2022khovanov,lamm2022symmetric,lobb2021refinement,snape}, also in the direction of equivariant concordance \cite{alfieri2021strongly,boyle2021equivariant,dai_mallick_stoffregen,dipriframba2,dipriframba1,di2023equivariant,miller2023strongly}.
The present work is situated within this context and aims to develop new tools for the study of \si knots.

Here, we develop an equivariant version of grid homology for strongly invertible knots. The construction is based on a set of grid diagrams, called \emph{symmetric}, encoding the strong inversion, together with an involution on the associated grid complex. The resulting mapping cone complex gives rise to new invariants of strongly invertible knots.
In recent years, homological knot invariants have undergone substantial refinement in the presence of symmetries. 
Equivariant reformulations of Khovanov homology have in fact carved out a place in the literature \cite{sano2024involutive,lipshitz2022khovanov,lobb2021refinement,khovanov2025symmetries}.
The investigation of equivariant knots through knot Floer homology was initiated in \cite{dai_mallick_stoffregen} by Dai, Mallick, and Stoffregen, who introduced new equivariant concordance invariants and lower bounds for the equivariant slice genus, and has more recently been pursued in \cite{xiao2026real,parikh2024localization}. More generally, the interaction between strong inversions and knot invariants has been explored extensively in the recent literature \cite{hendricks2025note,Mallick2022,miller2023strongly,boyle2021equivariant,
 	hirasawa2022invariant,lamm2022symmetric,snape}.

First, we note that \si knots can be represented by grids that reflect their symmetry properties, we call such grids \emph{symmetric grids}. 
Theorem~\ref{thm:cromwellSI} provides a version of Cromwell's theorem for \si knots. More precisely, it provides a list of moves on grids, called \emph{symmetric grid moves}, under which symmetric grids represent equivariantly isotopic \si knots.
We induce an involution $\rho$ on the grid complex associated to a symmetric grid. Theorem~\ref{prop:key} proves that $\rho$ is natural with respect to the symmetric grid moves.

\vspace{3.5mm}
\noindent{\bfseries Theorem \ref{prop:key}}
\textit{Let $\G$ and $\G'$ be two symmetric grids together with the two involutions $\rho:\gc \rightarrow \gc$ and $\rho': GC^-(\G') \rightarrow GC^-(\G')$. Suppose that $\G$ and $\G'$ are connected by a \si grid move $F$ and let $F:\gc \rightarrow \gcc$ be the morphism induced between the complexes.
Then the compositions $F\circ \rho$ and $\rho' \circ F$ are chain-homotopic.}
\vspace{3.5mm}

Motivated by recent equivariant constructions (e.g. \cite{lobb2021refinement,sano2024involutive,dai_mallick_stoffregen}), in Section~\ref{sec:cono} we study the mapping cone $\cono$.  
Theorem~\ref{thm:invariance} is the main result of the section.

\vspace{3.5mm}
\noindent{\bfseries Theorem \ref{thm:invariance}}
\textit{Let $K$ be a \si knot and $\G$ be a grid diagram such that $L(\G) = K$. The homology of $\text{Cone}_{\G}(Id+\rho)$ only depends on the \si type of $K$.}
\vspace{3.5mm}

It is hence possible to set $\hc(K) := H\text{Cone}_{\G}(Id+\rho)$.
We then derive simpler knot invariants from $\hc(K)$, as the equivariant analogous to the Legendrian grid invariants ($d_e^\pm(K)$) and to the Ozsv\'ath-Szab\'o $\tau$ invariant ($\tau_i(K)$, $i=0,1$). These equivariant versions recover specific properties. For example Proposition~\ref{prop:d_U_molt} proves that $d_e^+$ and $d_e^-$ are sensible to equivariant stabilization of \si Legendrian knots, while Theorem~\ref{rmk:unknot_bound} shows that the equivariant Ozsv\'ath-Szab\'o $\tau$ invariant provide lower bounds for the equivariant unknotting numbers defined in \cite{boyle2024equivariant}.

\vspace{3.5mm}
\noindent{\bfseries Proposition \ref{prop:d_U_molt}}
\textit{
	Let $\G$, $\G^+$ and $\G^-$ be grid diagrams whose associated \si Legendrian knots are $K$ and its equivariant stabilization $K^+$ and $K^-$. Assume that both stabilizations are NE- or SW-type along the axis. Then:
	\[ d_e^+(\G^-)=d_e^+(\G) \quad\quad d_e^-(\G^-)=d_e^-(\G)+1, \]
	and
	\[ d_e^+(\G^+)=d_e^+(\G)+1 \quad\quad d_e^-(\G^+)=d_e^-(\G). \]}
\vspace{0mm}

\noindent{\bfseries Theorem \ref{rmk:unknot_bound}} \textit{Let $K$ be a \si knot.
	Assume to unknot $K$ performing only \emph{Type A} crossing changes, then $|\tau_i(K)| \leq 2\widetilde{u}_A(K)$, $i=0,1$. Performing only \emph{Type B} crossing changes, one gets $|\tau_i(K)| \leq \widetilde{u}_B(K)$, $i=0,1$. 
	In general: $|\tau_i(K)| \leq 2\widetilde{u}(K)$, $i=0,1$.}
\vspace{3.5mm}
 
The explicit computation of these invariants is an interesting direction for future investigation. Motivated by this, we propose a few open questions.
 \begin{quest}
	Since $\tau_0(K)\leq\tau(K)\leq\tau_1(K)$ (see Remark~\ref{rmk:tauineq}), find examples of \si knots such that the inequalities are strict.
\end{quest}
\begin{quest}
	Find examples of different \si structure, on the same topological knot, distinguished by $\tau_0$ or $\tau_1$.
\end{quest}
\begin{quest}
	Could $\tau_0$ and $\tau_1$ be used to answer \cite[Problem E]{dioguz} question on the existence of equivariant $\Q$-concordance invariants using Floer homology?
\end{quest}
 \begin{quest}
	Find a symmetric grid $\G$ such that:
	\begin{itemize}
		\item $L(\G)$ as a Legendrian knot is unequivariantly positively (resp. negatively) stabilized.
		\item $d_e^+(\G)$ (resp. $d_e^-(\G)$) is zero.
	\end{itemize}
\end{quest}
 
Finally, we study equivariant cobordisms of \si knots. We establish an equivariant normal form for simple equivariant cobordisms. Using this, Theorem~\ref{thm:genus_bound} shows that the equivariant version of the $\tau$-invariant provides lower bounds for the slice genus of a simple equivariant cobordism.

\vspace{3.5mm}
\noindent{\bfseries Theorem \ref{thm:genus_bound}} \textit{Let $K$ and $K'$ be two \si knots connected by a genus $g$ simple equivariant cobordism $S\subseteq [0,1]\times S^3$. Then the equivariant tau invariants $\tau_0$ and $\tau_1$ both provide a lower bound to the genus of the cobordism:
	\[ \max \left\{|\tau_0(K)-\tau_0(K')|\,,\, |\tau_1(K)-\tau_1(K')| \right\} \leq g. \]}
\vspace{0.5mm}

We conclude by proposing the following conjecture.

\vspace{3.5mm}
\noindent{\bfseries Conjecture \ref{congettura}} \textit{Let $K_1$ and $K_2$ be two \si knots connected by a genus $g$ simple equivariant cobordism.
	Then there exists an $\F[U]$-linear $(-2g,-g)$-homogeneous homomorphism:
	\[ \phi: \hc(K_1) \rightarrow \hc(K_2), \]
	Moreover, $\phi$ is induced by a chain map and its image is not contained in the torsion submodule.}
\vspace{3.5mm}

Our approach to the conjecture consists in first proving the statement for elementary equivariant cobordisms. Proofs for some cases are presented, and the section concludes with a brief discussion of the remaining cases.

\subsection{Strongly invertible knots}
This work focuses on a particular family of knots distinguished by a specific symmetry property.
\begin{defn}
	A \emph{strongly invertible knot} is a pair $(K,\rho)$, where $K:S^1 \hookrightarrow S^3$ is a knot and $\rho:S^3 \rightarrow S^3$ is an orientation-preserving smooth involution such that $\rho(K)=K$ and $\rho$ reverses the orientation on $K$. We will call $\rho$ the \emph{strong involution}.
\end{defn}
Note that, thanks to the resolution of the Smith conjecture \cite{smith}, we can define a strongly invertible knot as a pair $(K,\rho)$ where $K$ is a knot and $\rho:S^3\rightarrow S^3$ is an involution such that $Fix(\rho)$ is an unknotted circle meeting $K$ in exactly two points (in the case of a link, this is valid for each component). Therefore, a strongly invertible knot always admits a \emph{transvergent} diagram, that is, a diagram symmetric via the $\pi$-rotation along an axis lying in the same plane as the diagram. See Figure~\ref{img:Es_SI}.
\begin{figure}[h]
	\centering
	\includegraphics[width=0.4\linewidth]{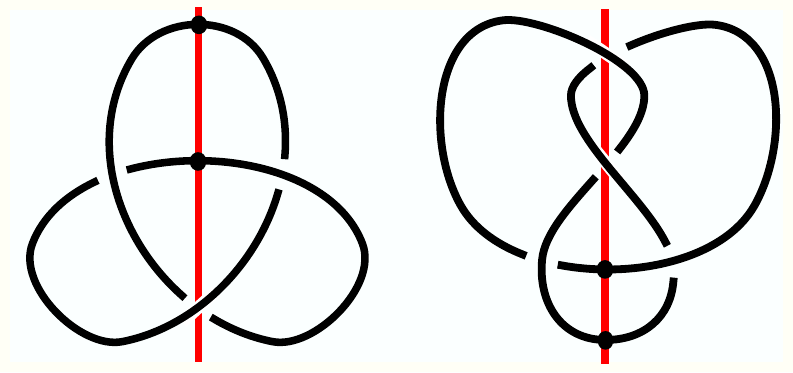}
	\caption{The strongly invertible knots $(3_1,\rho)$ and $(4_1,\rho)$. In both cases, the strong inversion is the $\pi$-rotation along the vertical red axis.}
	\label{img:Es_SI}
\end{figure}

We say that two strongly invertible knots are \emph{equivariantly isotopic}, or equivalently, if and only if they are isotopic via a family of strongly invertible knots. In other words, two strongly invertible knots $(K_0,\rho_0),(K_1,\rho_1) \subset S^3$ are equal if and only if there exists a continuous family $(K_t,\rho_t)$, $t\in[0,1]$ of strongly invertible knots interpolates between $K_0$ and $K_1$.

\section*{Structure}
The paper is organized as follows. Section~\ref{prelim} introduces the basic definitions of \si Legendrian knots and grid homology. This section may be skipped by readers already familiar with the topics. Section~\ref{chp:eq_grid} introduces symmetric grid diagrams for strongly invertible knots, proves that every strongly invertible knot admits such a representation, and establishes the corresponding analogue of Cromwell's theorem. It then constructs the involution on the grid complex and studies its naturality with respect to symmetric grid moves. Section~\ref{sec:cono} introduces the equivariant version of grid homology for strongly invertible knots through the mapping cone of the involution, proves that its homology is an invariant of strongly invertible knots, and derives invariants, including equivariant analogues of the Ozsváth--Szabó $\tau$ invariant and the Legendrian grid invariants. Section~\ref{sec:cobbound} investigates cobordisms of \si knots, establishes an equivariant normal form for simple equivariant cobordisms, derives simple genus bounds from the equivariant $\tau$ invariant, and concludes with conjectural functorial properties of the theory.

\section*{Acknowledgments}
I thank my advisor, Paolo Lisca, for his kindness and generosity with his time. I would like to thank Marco Marengon for his attentive listening.
I am indebted to Kristen Hendricks for her interest in and patience with my work.
I am grateful to Filippo Bianchi, Alice Merz, and Diego Santoro for sharing the perspective that comes with greater experience in the mathematical journey.
Finally, I would like to thank Alessio Di Prisa for his enthusiasm and for his constant willingness to help.

\section{Preliminaries}\label{prelim}
We present the preliminaries needed throughout the paper. 
Section~\ref{app:leg} defines strongly invertible Legendrian knots, following \cite{collarilisca}. 
Section~\ref{chp:grid} recalls the construction of grid homology. The primary reference is \cite{ozsvath2015grid}. 
Readers already familiar with the subjects may skip to Section~\ref{chp:eq_grid}.

\subsection{Strongly Invertible Legendrian Knots}\label{app:leg}
Let $\xi_{st} \subset T\R^3$ be the \emph{standard contact structure} on $\mathbb{R}^3$.
\begin{defn}[\cite{collarilisca}]
	Let $\rho:\R^3\rightarrow \R^3$ be the involution: $\rho(x,y,z) = (x,-y,-z)$. 
	A \emph{strongly invertible Legendrian link} is a link that is simultaneously strongly invertible with respect to $\rho$ and Legendrian with respect to $\xi_{st}$. Two strongly invertible Legendrian links are equivalent if they are connected by a family of strongly invertible Legendrian links.
\end{defn}
\begin{defn}[\cite{collarilisca}]
	A front $\mathcal{F}$ of a \sil link is a \emph{transvergent front} if the reflection $(x,z)\rightarrow (x,-z)$ fixes $\mathcal{F}$ setwise.
\end{defn}
See Figure~\ref{img:sifront} for easy examples of transvergent fronts. In \cite[Theorem 1.3]{collarilisca}, the authors provide the analogue of Reidemeister's theorem in the \sil setting. 

\begin{figure}
	\centering
	\includegraphics[width=0.6\linewidth]{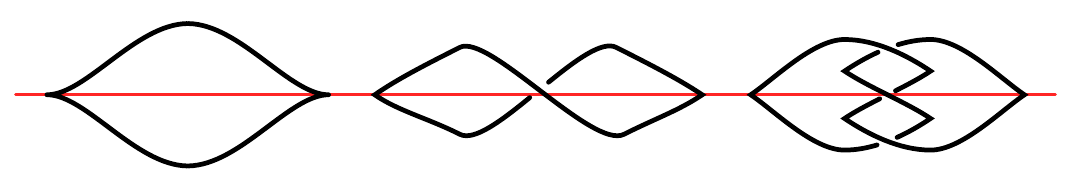}
	\caption{Examples of transvergent fronts. The red axis is the $x$-axis.}
	\label{img:sifront}
\end{figure}

\begin{remark}\label{rem:genericità}
	By the $CX$ and $CR$ moves in \cite[Theorem~1.3]{collarilisca}, we can always assume that the leftmost (resp. rightmost) point of the transvergent front is a left (resp. right) cusp, as in the examples of Figure~\ref{img:sifront}. 
\end{remark}

\begin{figure}[t]
	\centering
	\begin{subfigure}[c]{0.98\textwidth}
		\centering
		\includegraphics[width=0.6\linewidth]{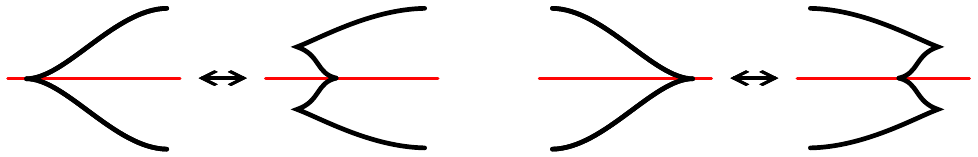}
		\caption{$T$-stabilization on a transvergent front.}
	\end{subfigure}
	\centering
	\begin{subfigure}[c]{0.98\textwidth}
		\centering
		\includegraphics[width=0.6\linewidth]{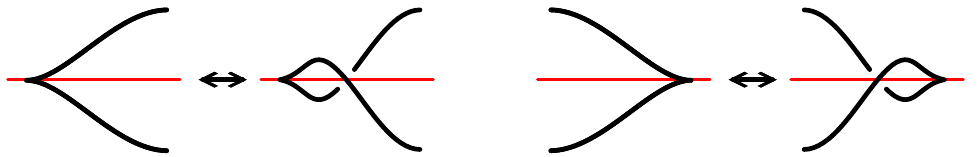}
		\caption{$S$-stabilization on a transvergent front.}
	\end{subfigure}
	\caption{}
	\label{img:TSstab}
\end{figure}
\begin{remark}
	An oriented strongly invertible knot is, in particular, equivalent to the knot with the opposite orientation. For strongly invertible Legendrian knots, this equivalence holds only when the rotation is zero. Inverting the orientation changes the sign of the rotation, thereby altering the Legendrian type. Consequently, the analysis will focus on the unoriented case, where the classical invariants are the Thurston-Bennequin invariant and the absolute value of the rotation.
\end{remark}
In \cite[Definition 4.1]{collarilisca}, the authors define two stabilizations on a \sil knot, namely the \emph{T-stabilization} and the \emph{S-stabilization} (see Figure~\ref{img:TSstab}). The latter is a fundamental notion in establishing the following theorem, which is the \sil analogue of Fuchs and Tabachnikov's result.
\begin{thm}\cite[Theorem 1.4]{collarilisca}\label{thm:SIL_FuchsTabachnikov}
	Let $\mathcal{L}$ and $\mathcal{L}'$ be \sil links which are equivalent as \si links. Then, after sufficiently many $S$-stabilizations, $\mathcal{L}$ and $\mathcal{L}'$ become equivalent \sil links.
\end{thm}

\subsection{Introduction to Grid Homology}\label{chp:grid}
We briefly review the construction of grid homology, following the exposition in \cite{ozsvath2015grid}.

\begin{defn}
	A \emph{planar grid diagram} $\G$ is an $n\times n$ grid on the plane; that is, a square made of $n$ rows and $n$ columns of small squares. Furthermore, $n$ of these small squares are marked with an $X$, and $n$ of them are marked with an $O$. These markings are distributed in such a way that:
	\begin{itemize}
		\item Each row has a single square marked with an $X$ and a single square marked with an $O$. 
		\item Each column has a single square marked with an $X$ and a single square marked with an $O$. 
		\item No square is marked both with an $X$ and with an $O$.
	\end{itemize}
	The number $n$ is called the \emph{grid number} of $\G$, or \emph{size} of $\G$.
\end{defn}
\begin{remark}	\label{rmk:extended}
	Note that the forthcoming construction of the grid homology does not require the third condition (see \cite[Section 8.4]{ozsvath2015grid}). A square marked with both an $O$-marking and an $X$-marking represents an unknotted component, and a grid admitting such squares is called \emph{extended grid}. 
\end{remark}
By convention, rows are enumerated from bottom to top and columns from left to right.
The set of squares marked with an $X$ is denoted by $\X$. Analogously, $\Oo$ is the set of squares marked with an $O$. 
\begin{figure}
	\centering
	\includegraphics[width=0.95\linewidth]{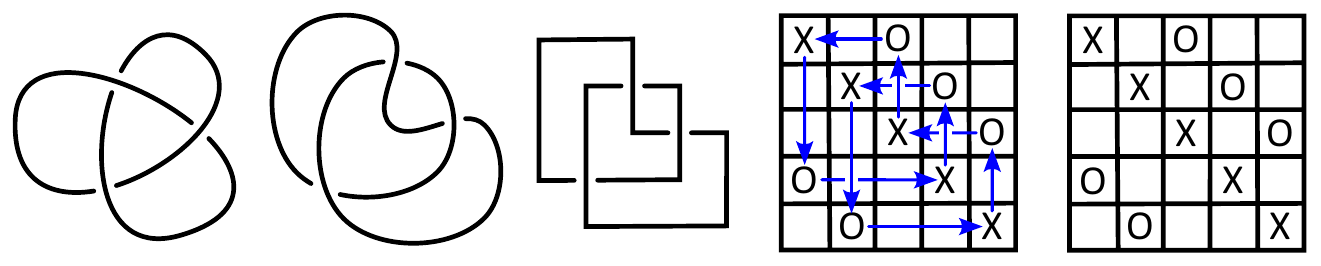}
	\caption{Diagram to grid and vice versa.}
	\label{img:gridtodiag}
\end{figure}
\begin{remark}
	Observe that it is possible to recover a knot (or link) $L(\G)$ from any planar grid diagram $\G$ by drawing oriented segments connecting the $O$ to the $X$ for each row and connecting the $X$ to the $O$ in each column. The convention for the crossings is that the vertical strand is always overcrossing. We say that $L(\G)$ is the \emph{underlying link} of $\G$. See Figure~\ref{img:gridtodiag}.
	Similarly, every link is represented by a planar grid diagram, therefore for every link $L$ there exists a planar grid diagram $\G$ such that $L(\G)$ is equivalent to $L$. In fact, it is sufficient to consider a PL approximation of $L$ such that the projection only admits horizontal and vertical segments and to fix the horizontal over-crossings by planar isotopies. When necessary, place horizontal and vertical segments in general position, namely, so that no collinear segments occur. Then $\G$ is found by marking the turns suitably with $X$'s and $O$'s, as in Figure~\ref{img:gridtodiag}.
\end{remark}

The question arises of when two planar grid diagrams represent the same link. 
To address this question,it is necessary to define two sets of moves on planar grid diagrams.
\begin{defn}
	For each column of a grid diagram, the heights of the $X$ and of the $O$ markings determine a closed interval on $\mathbb{R}$. 
	Fix a planar grid diagram $\G$ and suppose that the intervals associated with two consecutive columns are either disjoint or one is contained in the interior of the other. Then, by switching these two columns, we obtain a new planar grid $\G'$ (see Figure~\ref{img:commut}) and we say that $\G$ and $\G'$ are connected by a column commutation. By the same procedure as rows, it is possible to define row commutations. A column or row commutation is called a \emph{commutation}. 
\end{defn}

\begin{defn}
	Let $\G$ be an $n\times n$ planar grid diagram. A \emph{stabilization} of $\G$ is an $(n+1)\times(n+1)$ planar grid diagram $\G'$ obtained by replacing a row and a column of $\G$ with two adjacent rows and two adjacent columns marked in any of the ways depicted in Figure~\ref{img:stab}. Vice versa, we say that $\G$ is a \emph{destabilization} of $\G'$. 
\end{defn}
\begin{figure}
	\centering
	\includegraphics[width=0.3\linewidth]{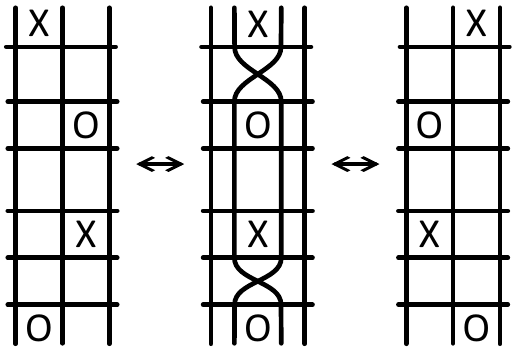}
	\caption{Example of a column commutation, with the intermediate step of a grid containing the vertical lines of both $\G$ and $\G'$.}
	\label{img:commut}
\end{figure}
\begin{figure}
	\centering
	\includegraphics[width=0.9\linewidth]{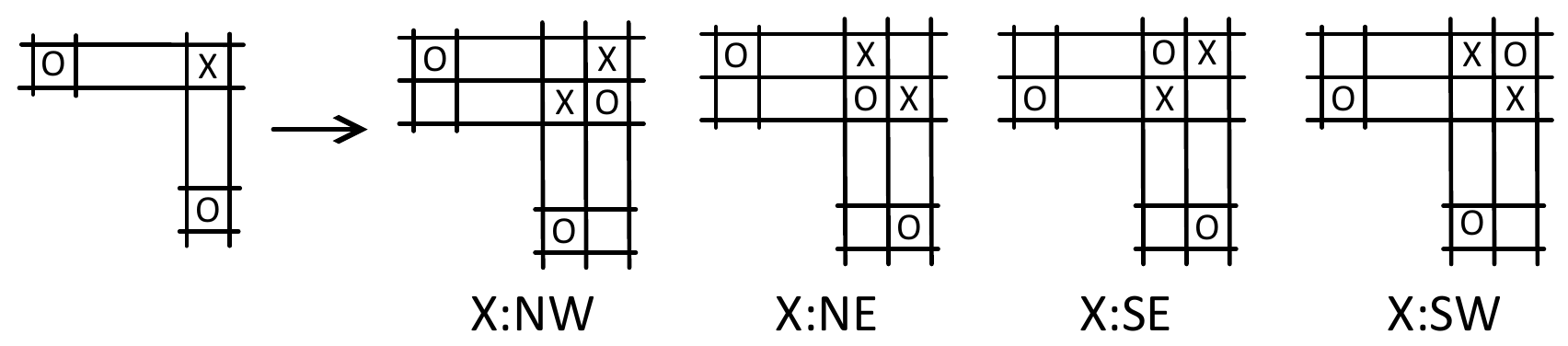}
	\caption{Stabilizations at an $X$-marking. For the stabilizations at an $O$-marking, exchange all the $X$'s and $O$'s markings in the pictures.}
	\label{img:stab}
\end{figure}
We will collectively refer to commutations and (de)stabilizations as \emph{grid moves}.
It is possible to define knot invariants using grid diagrams in light of the following result.
\begin{thm}[Cromwell,\cite{CROMWELL199537}]\label{thm:cromwell}
	Two planar grid diagrams represent equivalent links if and only if there is a finite sequence of grid moves that transforms one into the other.
\end{thm} 

Transfer planar grid diagrams to the torus, by identifying the top boundary segment with the bottom one and the left boundary segment with the right one.
The quotient inherits an orientation from the plane. Furthermore, vertical and horizontal lines become circles. Call $\aalpha=\{\alpha_i\}_{i=1}^n$ the horizontal circles and $\bbeta=\{\beta_i\}_{i=1}^n$ the vertical circles.
The resulting diagram is called a \emph{toroidal grid diagram}.
A size $n$ toroidal grid diagram admits $n^2$ realizations as a planar grid diagram, depending on which $\alpha$ and $\beta$ circles we cut along to obtain a square. The links represented by such realizations are all equivalent (\cite[Section 3.2]{ozsvath2015grid}).

\subsubsection{Grid homology}
The discussion specializes to the case of knots.  
For most definitions and applications, the field of two elements $\F= \Z/2\Z$ is used.
The generators of the grid complex are now defined.

\begin{defn}
	A \emph{grid state} for a grid number $n$ grid diagram $\G$ is an $n$-tuple $\xx=\{x_1,...,x_n\}$ of points in the torus, called \emph{intersection points}, such that each horizontal circle and each vertical circle contains exactly one point of $\xx$. The set of all grid states for $\G$ is denoted $\St$.
\end{defn}

\begin{figure}
	\centering
	\includegraphics[width=0.25\linewidth]{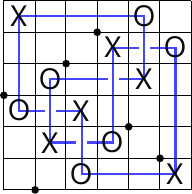}
	\caption{A grid state on a grid representing the figure eight knot.}
	\label{fig:fig8}
\end{figure}

States on a grid diagram are represented using a planar realization of the grid (see Figure~\ref{fig:fig8}).
Due to the identifications of sides, no intersection points are indicated on the rightmost and top edges.

Observe that the $\aalpha$ and $\bbeta$ circles divide the torus into $n^2$ oriented squares $S_1,\dots,S_{n^2}$. A formal linear combination of the closures of these squares, $\psi= \sum a_i S_i$ with integer $a_i$'s, has a boundary $\partial \psi$ that is a formal linear combination of segments in $\aalpha \cup \bbeta$. We call $\partial_\alpha \psi = \partial \psi \cap \aalpha$ and $\partial_\bbeta \psi = \partial \psi \cap \bbeta$. 
\begin{defn}\cite[Definition 4.6.4]{ozsvath2015grid}
	Fix $\xx,\yy \in \St$. A \emph{domain} $\psi$ from $\xx$ to $\yy$ is a formal linear combination of the closure of the squares in $\G\setminus (\aalpha \cup \bbeta)$, with the property that $\partial(\partial_\alpha \psi) = \yy - \xx$ and hence $\partial(\partial_\alpha \psi) = \yy - \xx$. In this equation, the two sides represent a formal linear combination of points. Denote the set of domains from $\xx$ to $\yy$ by $\pi(\xx,\yy)$.
\end{defn}

Fix two states $\xx,\yy \in \St$, and an embedded rectangle $r$ in the torus such that $\partial r \subset \aalpha \cup \bbeta$. Suppose that $\xx$ and $\yy$ intersect in exactly $n-2$ points and that the remaining four points of $\xx\cup\yy$ are the corners of $r$. Call $\partial_\alpha r = \partial r\cap\aalpha$ and similarly for $\partial_\beta r$. Then the rectangle $r$ goes from $\xx$ to $\yy$, hence $r\in \pi(\xx,\yy)$, if:
\[ \partial(\partial_\alpha r) = \yy - \xx \qquad \text{and} \qquad  \partial(\partial_\beta r) = \xx - \yy. \]
\begin{remark}
	Denote the set of rectangles from $\xx$ to $\yy$ by $\text{Rect}_\G(\xx,\yy)$. Note that it is either empty or consists of two distinct elements; in the latter case, the same goes for $\text{Rect}_\G(\yy,\xx)$.
\end{remark}
In particular, we will be interested only in the following rectangles, called \emph{empty rectangles}:
\[ \text{Rect}^\circ_\G(\xx,\yy) = \left\{ r \in \text{Rect}(\xx,\yy) \;|\; \xx\cap \text{Int}(r) = \yy\cap \text{Int}(r) = \emptyset \right\}. \]
In some cases, the reference to the grid will be omitted from the notation.


Given a rectangle $r$, for all $i=1,...,n$ define the \emph{multiplicity} $O_i(r)$ of $r$ at $O_i$, as $1$ or $0$ depending on whether $r$ contains $O_i$ or not. 
Finally, fix formal variables $V_1,...,V_n$ and consider the ring $\mathcal{R}=\F[V_1,...,V_n]$.
\begin{defn}
	The \emph{(unblocked) grid complex} $\gc$ is the free module over $\mathcal{R}$ generated by $\St$, equipped with the $\mathcal{R}$-module endomorphism defined for any $\xx\in \St$ as:
	\[ \partial_\X^-\xx = \sum_{\yy\in\St} 
	\sum_{\left\{r\in \text{Rect}^\circ(\xx,\yy) \;|\; r\cap\X = \emptyset\right\} }
	V_1^{O_1(r)}\cdots V_n^{O_n(r)}\cdot \yy. \]
\end{defn}
There exist two grading functions defined on states (see \cite[Section 4.3]{ozsvath2015grid}), which induce two gradings on the grid complex: the \emph{Maslov grading} and the \emph{Alexander grading}, denoted by $d$ and $s$, respectively. We hence have that $\gc = \bigoplus_{d,s\in\Z}GC_d^-(\G,s)$ is a bigraded $\mathcal{R}$-module and $\partial_\X^-$ is a homogeneous $\mathcal{R}$-module homomorphism of degree $(-1,0)$, i.e. sends $GC_d^-(\G,s)$ to $GC_{d-1}^-(\G,s)$, such that $\partial_\X^-\circ \partial_\X^- = 0$.

The following is a key property.
\begin{lemma}\cite[Lemma 4.6.9]{ozsvath2015grid}\label{eU}
	For any pair of integers $i,j\in\{1,...,n\}$, multiplication by $V_i$ is chain homotopic to multiplication by $V_j$, when thought of as homogeneous maps from $\gc$ to itself of degree $(-2,1)$.
\end{lemma}
Note that the proof of Lemma~\ref{eU} uses the fact that we specialized to the knot case. In the general link case, the thesis holds only for multiplications associated with $O$-markings that lie in the same connected component. 

The grid homology module that we introduced is a knot invariant and can be used to extract other invariants, in the following sense.
\begin{thm}\cite[Theorem 4.6.19]{ozsvath2015grid}\label{thm:gh}
	The homology $\gh$ (thought of as a bigraded $\F[U]$-module) depend on the grid $\G$ only through its underlying (unoriented) knot. 
\end{thm}

\subsubsection{Grid representations for Legendrian knots}\label{sec:grid4leg}
The procedure illustrated in Figure~\ref{img:gridtoleg} demonstrate that any grid diagram $\G$ represents an oriented Legendrian knot (or link) $L(\G)$, where the $\pi/2$ rotation is used to avoid vertical tangencies. The $NE$ and $SW$ corners on the grid become cusps.
The crossing changes arise from the Legendrian condition imposed on crossings.

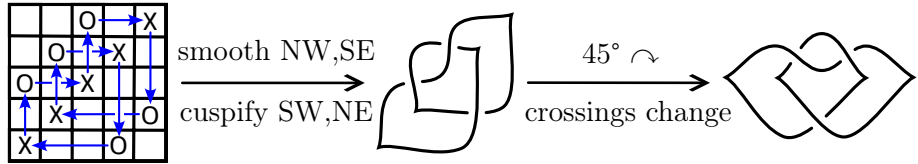
\begin{figure}[ht]
	\centering
	\begin{tikzpicture}
		\draw (0,0) node[above right]{\includegraphics[width=0.97\linewidth]{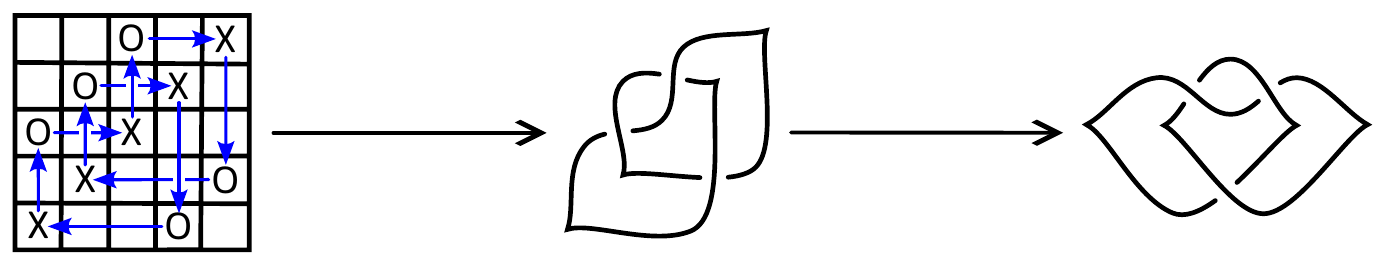}};
		\draw (3.8,1.7) node (a) [scale=1] {smooth NW,SE};
		\draw (3.8,0.9) node (a) [scale=1] {cuspify SW,NE};
		\draw (8.35,1.7) node (a) [scale=1] {45° $\curvearrowright$};
		\draw (8.45,0.9) node (a) [scale=1] {crossings change};
	\end{tikzpicture}
	\caption{Associating a Legendrian knot with a planar grid.}
	\label{img:gridtoleg}
\end{figure}

Let $\overline{\G}$ denote a planar realization of a toroidal grid $\G$, and let $L(\G)$ be the Legendrian knot type determined by the above procedure. The knot $L(\G)$ is referred to as the \emph{Legendrian knot associated to} $\G$. This knot type is well-defined, as it does not depend on the choice of $\overline{\G}$.

With additional effort compared to the classical setting (see \cite{ozsvath2015grid}), it can be shown that for any Legendrian knot type $\mathcal{L}$, there exists a grid diagram $\G$ such that $\mathcal{L}$ is the Legendrian knot associated to $\G$. Therefore, it is necessary to determine when two grids represent the same knot. The following result is the analog of Theorem~\ref{thm:cromwell} in the Legendrian case.
\begin{thm}{\cite[Theorem 12.2.6]{ozsvath2015grid}}\label{thm:cromwellleg}
	Two toroidal grid diagrams represent the same Legendrian knot type if and only if they can be connected by a finite sequence of commutations and (de)stabilizations of type X:NW, O:NW, X:SE, and O:SE on the torus. 
\end{thm}
The (de)stabilizations listed in Theorem~\ref{thm:cromwellleg} are referred to as \emph{Legendrian (de)stabilizations}.


\section{Equivariant Grid Homology}\label{chp:eq_grid}
This section shows how a strong inversion can be represented on
a grid diagram and how it induces an involution on the grid chain
complex, leading to a symmetry-refined version of grid homology.
Following \cite{lobb2021refinement,dai_mallick_stoffregen,sano2024involutive}, in Section~\ref{sec:cono} we study the mapping cone complex of the morphism given by the sum of the identity map and of the involution, showing in Theorem\ref{thm:invariance} that its homology is a knot invariant. 
Building on this new invariant, Section~\ref{sec:altri_inv} extracts simpler invariants, providing an equivariant analogue of the tau invariant \cite{KFfour-ball} and of the canonical cycles \cite{legendriangridinvaria}. 

Given a strongly invertible knot $(K,\rho)$, it is straightforward to observe that not all the grid diagrams $\G$ representing $K$ are suitable for the formulation of an equivariant grid theory.
Only grids which reflect the symmetry property of $(K,\rho)$ are considered in this framework. 
The theory is first developed for \sil knots (\cite{collarilisca}) and subsequently generalized to \si knots.

 \subsection{Equivariant grid representations}\label{skiphere}
 Initially, it is necessary to determine which grid diagrams represent \sil knots.  
 \begin{defn}\label{def:sym_grid}
 	We call \emph{symmetric grid diagram} any grid diagram $\G$ which satisfies the following two conditions:
 	\begin{itemize}
 		\item $\G$ is symmetric with respect to the SW to NE axis. 
 		\item Among the squares lying along the SW to NE axis, exactly two are marked, either with an $O$-marking or with an $X$-marking.
 	\end{itemize}
 \end{defn}
 Later on, we will work with grids representing symmetric links and not only \si knots. In that context, the second condition will be dropped.
 Given a symmetric grid diagram, the procedure described in Section~\ref{sec:grid4leg} and illustrated in Figure~\ref{img:gridtoleg}, yields a \sil knot. The resulting front is fixed setwise by the reflection $(x,z) \rightarrow (x,-z)$ and intersects the symmetry axis in exactly two points, hence is a transvergent front. 
 
 Notice that, these grids were first used to represent transvergent diagrams of \si knots by Yonghan Xiao in \cite{xiao2026real}. For an analogous discussion in the case of intravergent diagrams, see \cite{parikh2024localization}.
 
 Conversely, the following result demonstrates that every \sil knot admits a representation as a symmetric grid diagram.

 \begin{prop}\label{prop:frontetogrid}
 	Let $K\in\R$ be a \sil knot. There exists a symmetric grid diagram $\G$ such that $K$ is the \sil knot associated to $\G$.
 	\begin{proof}
 		Consider a transvergent front $\mathcal{F}$ for $K$. The rough idea is to modify this front by equivariant planar isotopies to obtain a PL-diagram $D$ in general position (that is, with no colinear segments). As we see in Figure~\ref{img:gridtodiag} steps, we can obtain a grid $\G$ from such a diagram $D$. We want $\G$ to be symmetric.
 		Since we only performed equivariant planar isotopies, we will get that $L(\G)$ is equivalent to $K$ as a \sil knot.
 		
 		Fix a transvergent front $\mathcal{F}$ for $K$, with the assumptions of Remark~\ref{rem:genericità}, apply a $\pi/4$ counterclockwise rotation on it and switch every crossing. Recall the Legendrian knot $L(\G)$ represented by $\G$ is the mirror of the topological knot associated with $\G$.
 		
 		Consider a square $S$ such that $\mathcal{F}\subseteq S$, and $\mathcal{F}$ is symmetric with respect to $d$, the bottom-left to top-right diagonal of $S$. Up to equivariant perturbations, namely equivariant planar isotopies, we can suppose that there exists a subdivision of $S$ in \emph{frames} $\{S_i\}_{i=1}^m$, see Figure~\ref{img:frame}, such that:
 		\begin{itemize}
 			\item $S_1$ contains a $SW$ cusp and $S_m$ contains a $NE$ cusp;
 			\item for each $i=2,\dots,n-1$, $S_i$ either contains at most one cusp or crossing on $d$, or it contains at most one couple of symmetric cusps or symmetric crossings.
 		\end{itemize}
 		 \begin{figure}
 			\centering
 			\begin{subfigure}[c]{0.2\textwidth}
 				\centering
 				\includegraphics[width=\textwidth]{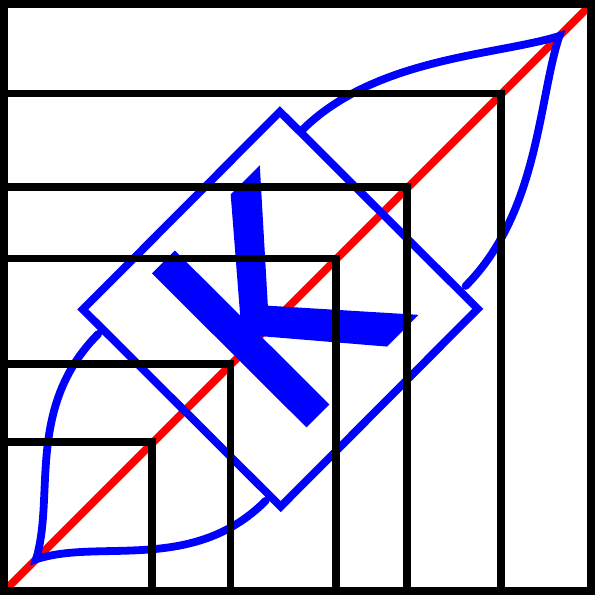}
 				\caption{Frames on a square.}
 				\label{img:frame}
 			\end{subfigure}
 			\hfill
 			\begin{subfigure}[c]{0.7\textwidth}
 				\centering
 				\includegraphics[width=\textwidth]{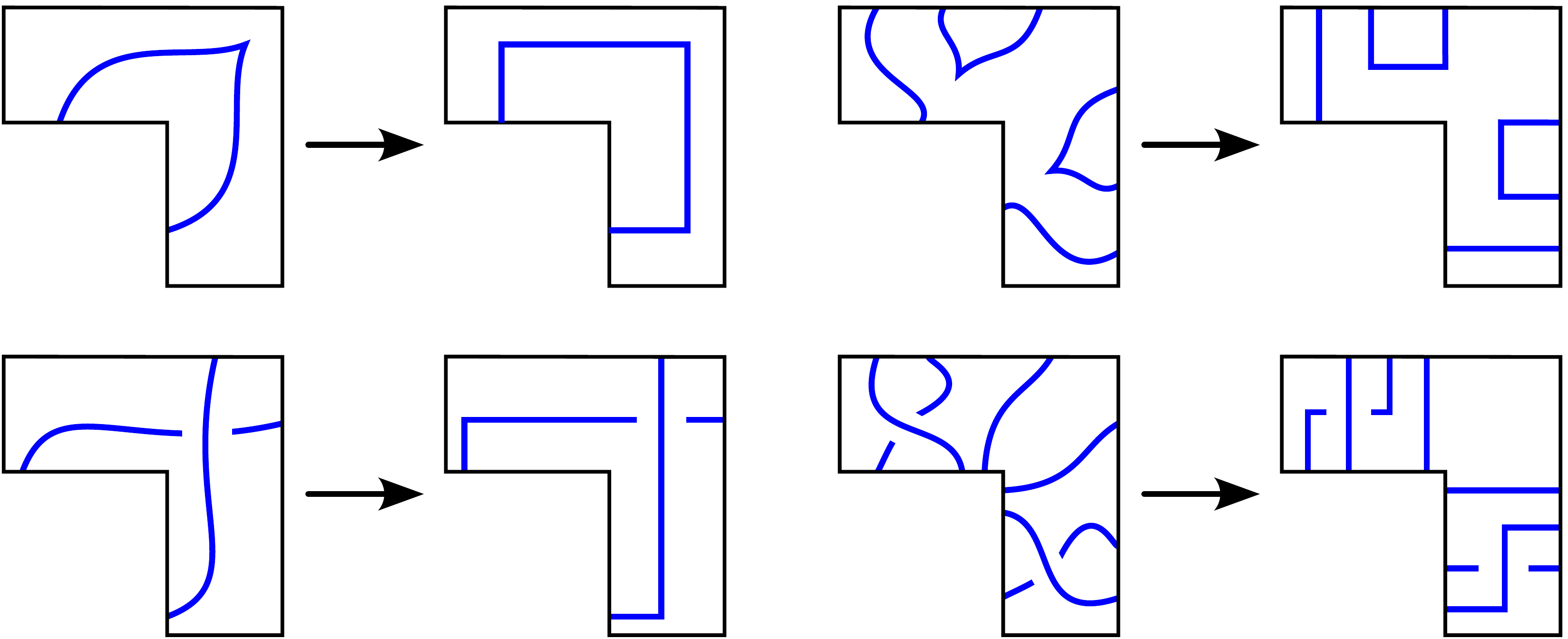}
 				\caption{Approximating by a grid diagram.}
 				\label{img:approx}
 			\end{subfigure}
 			\hfill
 			\caption{}
 		\end{figure}
 		Now we work on each frame separately. If $S_i$ does not contain cusps or crossings, replace the strands in $S_i$ above $d$ by vertical segments and the one below $d$ by horizontal ones, symmetrically. Otherwise, replace the strands with an equivariant planar isotopic picture that above (resp. below) $d$ contains only one horizontal (resp. vertical) arc and all other vertical (resp. horizontal) segments. See Figure~\ref{img:approx}.
 		
 		Now, starting from $i=2$, fit the pieces together, stretching $S_i$ so that the endpoints of the strands that lie on $\partial S_i \cap \partial S_{i-1}$ match with the ones coming from $S_{i-1}$. Up to small perturbations, the resulting PL diagram will be in general position. Of course, these perturbations are equivariant planar isotopies and do not change the \sil type of the diagram.
 		
 		At last, assume the bottom-left corner of $S\subseteq \mathbb{R}^2$ is the origin. Let $n$ be the number of horizontal strands and enumerate them from bottom to top. Since the diagram is symmetric, we also have $n$ vertical strands, enumerated from left to right. For each $i=1,\dots,n$, simultaneously move, by equivariant planar isotopies, the $i$-th vertical strand to one whose $x$-coordinate is $i-1/2$ and the $i$-th horizontal strand to one whose $y$-coordinate is $i-1/2$.
 		
 		Performing the steps on the last three pictures in Figure~\ref{img:gridtodiag}, we obtain the desired symmetric grid $\G$.
 	\end{proof}
 \end{prop}
 
The first essential result is a \sil version of Cromwell's theorem (Theroem~\ref{thm:cromwell}). The appropriate grid moves are defined below.
 \begin{defn}\label{orbo}
 	An \emph{equivariant commutation} (resp. \emph{(de)stabilization}) is a pair of commutations (resp. (de)stabilizations) symmetric with respect to the SW to NE axis. We call a \emph{\sil grid move}, or \emph{equivariant Legendrian grid move}, any composition of grid moves from the set of equivariant commutations, equivariant Legendrian (de)stabilizations, moves in Figure~\ref{img:silgm} (described below), their $\pi$-rotated, and the same moves with switched $O$ and $X$-markings.
 	
 	The move $CX$ is obtained via two stabilizations and a sequence of commutations. The first step is a $X:NW$ stabilization. An $X:SE$ stabilizations on the top-right $X$-marking follows. This pair of moves create a non-symmetric untwisted kink. Via a suitable sequence of commutation moves, one makes the kink symmetric and make it move across the crossing.
 	The move $XX$ is obtained via a sequence of $3$ vertical commutations (bringing the rightmost column to the leftmost position) and a sequence of $3$ horizontal commutations (bringing the top row to the bottom).
 	The move $CC$ is obtained via one pair of vertical commutations (bringing the leftmost column to the rightmost position) and one pair of horizontal commutations (bringing the bottom row to the top). 
 	The move $CR$ is obtained through two stabilizations and a commutation. First, perform an $X:NE$ stabilization. Then an $X:SE$ stabilization on the bottom-left $X$-marking, hence creating a kink. Finally, a row commutation to bring the upper horizontal strand of the kink to the top.
    Notice that some of these moves, for example $XX$ and $CC$, alter the position of some markings outside the depicted square. Such markings move accordingly under the performed commutation moves.
    
 \end{defn}
   The list of moves from Definition~\ref{orbo} is not minimal, it is fixed in analogy with \cite{collarilisca}. 
 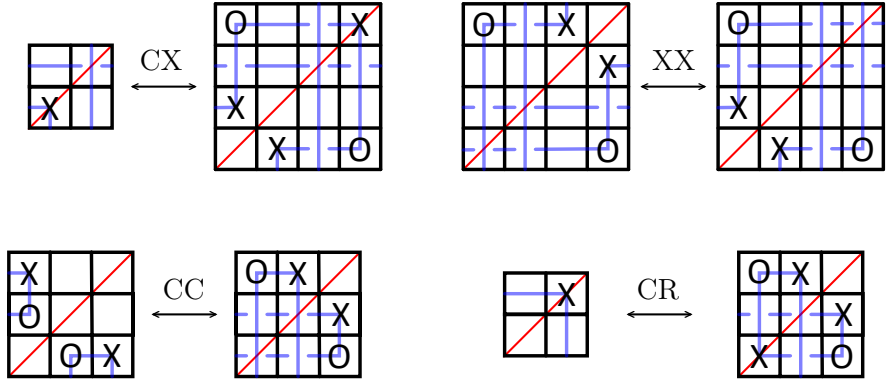
\begin{figure}[ht]
 	\centering
 	\begin{tikzpicture}
 		\draw (0,0) node[above right]{\includegraphics[width=0.95\linewidth]{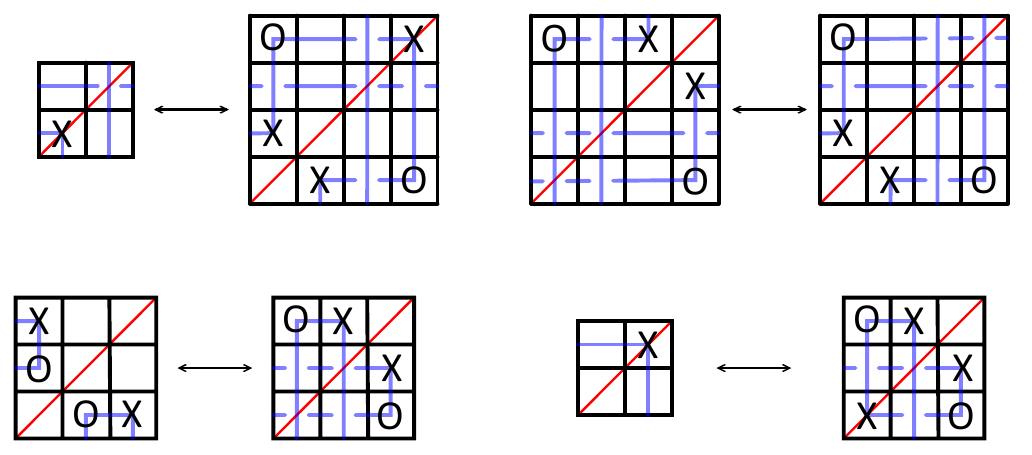}};
 		\draw (2.33,4.5) node (a) [scale=1] {CX};
 		\draw (2.63,1.48) node (a) [scale=1] {CC};
 		\draw (9.11,4.5) node (a) [scale=1] {XX};
 		\draw (8.9,1.48) node (a) [scale=1] {CR};
 	\end{tikzpicture}
 	\caption{Strongly invertible Legendrian grid moves. In red, the SW to NE axis, in translucent blue, the knot portion.}
 	\label{img:silgm}
 \end{figure}
 
 \begin{thm}\label{thm:cromwellmio}
 	Let $\G$ and $\G'$ be two symmetric grids, and call $K=L(\G)$ and $K'=L(\G')$.
 	Then $K$ and $K'$ are \sil isotopic if and only if $\G$ and $\G'$ can be connected by a finite sequence of \sil grid moves.
 	\begin{proof}
 		The strategy from the proof of \cite[Theorem~B.4.15]{ozsvath2015grid} is adapted, specifically by defining two maps:
 		\[
 		\begin{array}{c}
 			\faktor{\{ \text{\sil  knots}\}}{\text{\footnotesize \sil isotopy}} \\[8pt]
 			f\Big\downarrow \;\;\Big\uparrow g \\[8pt]
 			\faktor{\{\text{symmetric grid diagrams}\}}{\text{\footnotesize  \sil  grid moves}}
 		\end{array}
 		\]
 		
 		such that $f\circ g $ and $g\circ f$ are both the identity maps of the respective quotient spaces.
 		We show that the procedures we described to associate a \sil knot to a symmetric grid and vice versa induce two such maps.
 		
 		The construction of the map $g$ is established first. 
 		Call $\mathcal{F}$ the transvergent front obtained from $\G$. To show that $g$ is well defined on the quotient spaces, it suffices to show that when we apply any \sil grid move to $\G$, the associated front $\mathcal{F}$ undergoes a move from the list in \cite[Theorem~1.3]{collarilisca}, that is, the \sil version of the Reidemeister theorem.
 		
 		If $\G$ and $\G'$ are connected by an equivariant commutation, then $\mathcal{F}$ and $\mathcal{F}'$ are connected by an equivariant planar isotopy, an equivariant Legendrian Reidemeister 2 move, or an equivariant Reidemeister 3 move. 
 		In the case of an equivariant Legendrian (de)stabilization, $\mathcal{F}$ undergoes either an equivariant planar isotopy or an equivariant Legendrian Reidemeister 1 move.
 		The moves from Figure~\ref{img:silgm} correspond exactly to the ones in \cite[Figure~2]{collarilisca}.
 		
 		The observations are analogues for $f$. By \cite[Lemma~B.4.16]{ozsvath2015grid} we know that, if $\mathcal{F}$ and $\mathcal{F'}$ are two transvergent fronts that differ by planar isotopies, then the associated grids $\G$ and $\G'$ differ by a sequence of commutations. If the isotopy is equivariant, then so is the sequence of commutations. It follows that we can examine the remaining cases up to equivariant planar isotopy. As we apply the moves from \cite[Figure~2]{collarilisca} to the fronts, the grids undergo the corresponding moves shown in Figure~\ref{img:silgm}. Equivariant Reidemeister 2 and 3 moves can be realized by equivariant commutation moves, and equivariant Legendrian (de)-stabilizations can realize equivariant Reidemeister 1 moves.
 		
 		It is straightforward to observe that the composition $g\circ f$ is the identity. From the proof of Proposition~\ref{prop:frontetogrid}, it follows that for any \sil knot $K$, the \sil type of the knot represented by $f(K)$ coincides with that of $K$. It remains to show that, for any symmetric grid $\G$, there exists a finite sequence of \sil grid moves connecting $\G$ to $f(g(\G))$. 
 		
 		Call $d$ the SW to NE axis and suppose that $\G$ is such that:
 		\begin{itemize}
 			\item a horizontal segment above $d$ with a SW corner is a local minimum, and one with a NE corner is a local maximum;
 			\item each horizontal segment above $d$ contains at most one crossing;
 			\item a horizontal segment above $d$ is a local maximum or a local minimum if and only if it contains no crossings.
 		\end{itemize}
 		By symmetry, all the vertical segments below $d$ share the same properties. This is a \sil analogue to \cite[Definition~B.4.17]{ozsvath2015grid}.
 		Note that, in this case, it is obvious that $f(g(\G))$ is equivalent to $\G$. If we apply the construction from Proposition~\ref{prop:frontetogrid} to $g(\G)$, we are simply performing the steps that associate a transvergent front to a symmetric grid in the opposite order.    
 		
 		Finally, every symmetric grid $\G$ is connected by \sil grid moves to a symmetric grid that satisfies the properties listed above. This is achieved by applying the steps in the proof of \cite[Lemma~B.4.18]{ozsvath2015grid} to the portion of $\G$ above $d$, while performing symmetric modifications on the portion below $d$.
 	\end{proof}
 \end{thm}
 
 Let $\G$ be a symmetric grid diagram. Since the moves from Figure~\ref{img:silgm} can be expressed as a finite sequence of grid moves, $\gh$ is a \sil invariant of the \sil knot associated to $\G$. More generally, if $K$ is a strongly invertible knot, \cite[Proposition 2.1]{collarilisca} ensures that it can be assumed to be in Legendrian position. Consequently, Proposition~\ref{prop:frontetogrid} allows symmetric grids to represent strongly invertible knots. Theorem~\ref{thm:cromwellmio} can be adapted to this setting by incorporating and additional grid move called \text{S}-stabilization, as defined in \cite[Definition 4.1]{collarilisca}. 
 
 \begin{figure}
 	\centering
 	\includegraphics[width=0.25\linewidth]{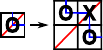}
 	\caption{An instance of an S-stabilization.}
 	\label{fig:s-stab}
 \end{figure}
 
 \begin{defn}
 	We call \emph{\text{S}-stabilization} a grid stabilization on the SW to NE axis as in Figure~\ref{fig:s-stab}, its $\pi$-rotated, or the same moves with switched $O$ and $X$-markings. 
 \end{defn}
 Note that \text{S}-stabilizations are $O:SW$, $X:SW$, $O:NE$, and $X:NE$ stabilization along the SW to NE axis.
 
 \begin{thm}\label{thm:cromwellSI}
 	Let $K$ and $K'$ be two strongly invertible knots that we assume to be in Legendrian position. Let $\G$ and $\G'$ be two symmetric grids such that $K=L(\G)$ and $K'=L(\G')$.
 	Then $K$ and $K'$ are equivariantly isotopic if and only if $\G$ and $\G'$ can be connected by a finite sequence of \sil grid moves and \text{S}-stabilizations.
 	\begin{proof}
 		The listed moves preserve the strongly invertible type of the represented knots. Vice versa, assume that $K$ and $K'$ are equivariantly isotopic. If the \sil type of $K$ and $K'$ is the same, Theorem~\ref{thm:cromwellmio} assures us that we find a sequence of \sil grid moves connecting $\G$ and $\G'$. Otherwise, \cite[Theorem 1.4]{collarilisca} states that after sufficiently many S-stabilizations, the \sil types of $K$ and $K'$ coincide. While the fronts of $K$ and $K'$ undergo these stabilizations, the grids $\G$ and $\G'$ undergo the corresponding grid stabilizations of the non-Legendrian type. 
 		After stabilizing, we end up in the previous case.
 	\end{proof} 
 \end{thm}
 The set of \emph{\si grid moves}, or \emph{equivariant grid moves}, is defined as the union of: \sil grid moves, S-(de)stabilizations, and equivariant (de)stabilizations of the non-Legendrian types (that is: X:SW, O:SW, X:NE, O:NE).
 To show that the latter moves preserve the \si type of the associated knot, one proceeds as in the Legendrian case.
 
 Notice that a version of Theorem~\ref{thm:cromwellSI} with a different set of moves was first stated in \cite[Proposition 3.3]{xiao2026real}.
 
 \subsection{Construction of the involution}
 Let $(K,\rho)$ be a \si knot, assumed to be in Legendrian position. There exists a symmetric grid $\G_K$ representing $K$. The goal is to define a bigraded chain map $\rho_\G:GC^-(\G_K) \rightarrow GC^-(\G_K)$ that, in a suitable sense, is induced by the strong involution $\rho:S^3\rightarrow S^3$.
 In this construction, the fixed point circle of the strong involution $\rho$ corresponds to the SW to NE axis on the grid.
 
 Let $\G$ be any planar grid diagram (therefore, not necessarily symmetric). Let $\rho \G$ be the grid obtained from $\G$ by reflection along the SW to NE axis. We define $\rho:\St \rightarrow \mathbf{S}(\rho \G)$ as the map that sends a state $\xx\in\St$ to its reflection along the SW to NE axis $\rho(\xx)\in\mathbf{S}(\rho \G)$. It is straightforward to observe that the map $\rho$ is $(0,0)$-homogeneous, namely it preserves the bigrading. 
 
 \begin{figure}
 	\centering
 	\begin{tikzpicture}
 		\draw (0,0) node[above right]{\includegraphics[width=0.42\linewidth]{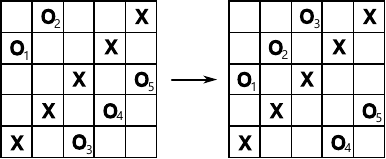}};
 		\draw (2.9,1.48) node (a) [scale=1] {$\rho$};
 		\draw (7.2,2.6) node (a) [scale=1] {$\rho(1) = 4$};
 		\draw (7.2,2) node (a) [scale=1] {$ \rho(2) = 5$};
 		\draw (7.2,1.4) node (a) [scale=1] {$ \rho(3) = 1$};
 		\draw (7.2,0.8) node (a) [scale=1] {$\rho(4) = 2 $};
 		\draw (7.2,0.2) node (a) [scale=1] {$\rho(5) = 3$};
 	\end{tikzpicture}
 	\caption{The action of $\rho$ on a non symmetric grid and description of the permutation $\rho$.}
 	\label{img:ex_permro}
 \end{figure}
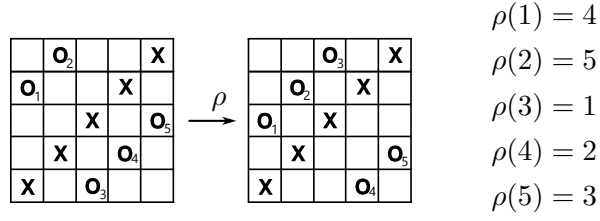
 
 Care must be taken when extending $\rho$ to the entire complex $\gc$. Extending it as a $\F[V_1,\dots,V_n]$-module map does not yield a chain map. Specifically, let $\{O_i\}_{i=1}^n$ and $\{O'_i\}_{i=1}^n$ denote the $\Oo$-sets of $\G$ and $\rho\G$, respectively. By convention, both $\Oo$-sets are enumerated left-to-right, so $\rho$ induces a permutation of $n$ elements (see Figure~\ref{img:ex_permro}), namely: \[O'_{\rho(i)} = \rho(O_i).\] There is a correspondence between empty rectangles $r$ from $\xx$ to $\yy$ and empty rectangles $\rho(r)$ from $\rho(\xx)$ to $\rho(\yy)$, but whenever $O_i(r)=1$, only $O'_{\rho(i)}(\rho(r))=1$ is guaranteed, and the value of $O'_i(\rho(r))$ is not determined. 
 
 What we do to fix this issue is to consider $\gc$ as an infinite-dimensional $\F$-vector space generated by the set \[\{ V_1^{k_1}\cdots V_n^{k_n}\xx \;|\; \xx\in\St,\; k_i\in\Z\;\forall i=1,\dots,n \}.\] For any $i=1,\dots,n$, we set $\rho(V_i)=V_{\rho(i)}$. We know how $\rho$ behaves on the elements of the base we fixed:
 \[ \rho(V_1^{k_1}\cdots V_n^{k_n}\xx) = V_{\rho(n)}^{k_1}\cdots V_{\rho(n)}^{k_n}\rho(\xx), \]
 so we can extend it linearly to obtain an $\F$-linear map $\rho:\gc\rightarrow GC^-(\rho\G)$. It is straightforward to observe that $\rho$ is an isomorphism of $\F$-vector spaces.  
 \begin{prop}\label{prop:rhochain}
 	The above $\F$-linear map $\rho$ is a chain map.
 	\begin{proof}
 		We want to show that $\rho\circ\partial_\X^- = \partial_\X^-\circ\rho$, where we consider the morphisms as linear maps between $\F$-vector spaces. We are using the same symbol for the differential of both complexes $\gc$ and $GC^-(\rho\G)$ to avoid an unnecessarily heavy notation. 
 		
 		Let $\xx \in \St$. Then:
 		\[ \rho\circ\partial_\X^-(\xx) = \sum_{y\in\St}\sum_{\left\{r\in \text{Rect}^\circ_\G(\xx,\yy) \;|\; r\cap\X = \emptyset\right\}} V_{\rho(1)}^{O_1(r)}\cdots V_{\rho(n)}^{O_n(r)}\cdot \rho(\yy).  \]
 		While:
 		\begin{align*}
 			\partial_\X^-\circ\rho(\xx) &=  \sum_{y\in\mathbf{S}(\rho \G)}\sum_{\left\{s\in \text{Rect}^\circ_{\rho\G}(\rho(\xx),\yy) \;|\; r\cap\X = \emptyset\right\}} V_1^{O_1(s)}\cdots V_n^{O_n(s)}\cdot \yy \\
 			= \sum_{y\in\St}&\sum_{\left\{\rho(r)\in \text{Rect}^\circ_{\rho\G}(\rho(\xx),\rho(\yy)) \;|\; r\cap\X = \emptyset\right\}} V_1^{O_1(\rho(r))}\cdots V_n^{O_n(\rho(r))}\cdot \rho(\yy) \\  
 			&= \sum_{y\in\St}\sum_{\left\{r\in \text{Rect}^\circ_{\G}(\xx,\yy) \;|\; r\cap\X = \emptyset\right\}} V_1^{O_{\rho(1)}(r)}\cdots V_n^{O_{\rho(n)}(r)}\cdot \rho(\yy) \\
 			&= \sum_{y\in\St}\sum_{\left\{r\in \text{Rect}^\circ_\G(\xx,\yy) \;|\; r\cap\X = \emptyset\right\}} V_{\rho(1)}^{O_1(r)}\cdots V_{\rho(n)}^{O_n(r)}\cdot \rho(\yy).
 		\end{align*}
 	\end{proof}
 \end{prop}
 
 \begin{remark}
 	Adopting the same notation for $\rho:\gc \rightarrow GC^-(\rho\G)$ and $\rho: GC^-(\rho\G) \rightarrow \gc$, it is immediate to note that, independently of which of the two compositions we are considering, $\rho^2=Id$. Observe how, in the case of a symmetric grid $\G$, we defined an $\F$-linear involution:
 	\[ \rho:\gc \rightarrow \gc. \]
 \end{remark}
 
 \begin{remark}
 	Notice that, in view of \cite[Lemma 2.10]{hendricks2025note}, which allows one to shift the markings so that they become fixed points, our involution $\rho$ appears to agree, in the multipointed setting, with the involution $\iota_K\tau_K$ studied in \cite{dai_mallick_stoffregen}. The precise relation between our forthcoming cone construction $\hc$ (see Section~\ref{sec:cono}) and the corresponding construction in \cite{dai_mallick_stoffregen} remains to be investigated.
 \end{remark}
 
 The following proposition is central to the definition of equivariant invariants of \si knots using $\rho$.
 
 \begin{thm}\label{prop:key}
 	Let $\G$ and $\G'$ be two symmetric grids together with the two involutions $\rho:\gc \rightarrow \gc$ and $\rho': GC^-(\G') \rightarrow GC^-(\G')$. Suppose that $\G$ and $\G'$ are connected by a \si grid move $F$ and let $F:\gc \rightarrow \gcc$ be the morphism induced between the complexes.
 	Then the compositions $F\circ \rho$ and $\rho' \circ F$ are chain-homotopic.
 	\begin{proof}
 		Given a non-equivariant grid move between two non-necessarily symmetric grids $c:GC^-(\G_1) \rightarrow GC^-(G_2)$, consider the composition:
 		\[ \overline{c} = \rho_2 \circ c \circ \rho_1: GC^-(\rho \G_1) \rightarrow GC^-(\rho G_2). \]
 		One can easily check that $\overline{c}$ is still a grid move. In particular, it is the grid move symmetric to $c$ under the reflection along the SW to NE axis.
 		
 		As any \si grid move, $F$ is a composition of grid moves, namely $F=c_k\circ\dots\circ c_1$.
 		It follows that $\rho'\circ F \circ \rho$ is also a composition of grid moves:
 		\begin{align*}
 			\rho'\circ F \circ \rho &= \rho'\circ c_k\circ\dots\circ c_1 \circ \rho \\&=  \rho'\circ c_k\circ \rho_k \circ \rho_k \circ c_{k-1} \circ \dots \circ c_2\circ \rho_2 \circ \rho_2 \circ c_1 \circ \rho \\&= \overline{c}_k \circ \dots \circ \overline{c}_1.
 		\end{align*} 
 		
 		By \cite[Section 5]{sarkar2015moving} we know that, up to chain homotopy, the maps induced on the grid complex by grid moves are the natural maps defined in \cite[Theorem 2.33]{jua_thu_zem}. Since the natural maps define a transitive system, the same holds for the composition of grid moves.
 		More precisely, \cite{sarkar2015moving} and \cite{jua_thu_zem} state that the induced maps are natural at the level of chain homotopy classes of chain maps, not only at the level of homology (See also \cite[Proposition 6.1]{hendricks2017involutive}).
 		The claim follows, as $F$ and $\rho'\circ F \circ\rho$ are both natural maps, unique up to chain homotopy; hence, $\rho'\circ F $ and $ F \circ\rho$ are chain homotopic.

 	\end{proof}
 \end{thm}
 
 This proposition establishes the first invariant in this context.
 
 \begin{defn}
 	Let $\G$ be a symmetric grid and call $\rho_*:\gh\rightarrow\gh$ the map induced in homology by $\rho:\gc \rightarrow \gc$. We call the pair $(\gh,\rho_*)$ the \emph{equivariant grid homology} module of $\G$. 
 \end{defn}
 
 \begin{cor}
 	Let $\G$ and $\G'$ be two symmetric grids such that $L(\G)$ and $L(\G')$ are equivalent as \si knots. There exists a isomorphism  $\psi:\gh \rightarrow \ghh$ such that $\rho' = \psi \circ \rho \circ \psi^{-1}$.
 	\begin{proof}
 		We only have to prove the thesis when $\G$ and $\G'$ are connected by a \si grid move $F$. 
 		We already know that $F$ induces an isomorphism $F_*:\gh \rightarrow \ghh$ and, by Theorem~\ref{prop:key}, $\rho'\circ F_* = F_* \circ \rho$. Hence $\psi=F_*$.
 	\end{proof}
 \end{cor}

 \section{The homology of the Cone complex}\label{sec:cono}
 Following \cite{dai_mallick_stoffregen,sano2024involutive}, consider the chain map:
 \[ Id+\rho : \gc \rightarrow \gc. \]
 The \emph{mapping cone} $\cono$ is the bigraded complex:
 \[  \cono_{d,s} = \left( GC^-_{d-1}(\G,s)\oplus GC^-_{d}(\G,s), \partial =\begin{pmatrix}
 	-\partial_\X^- & 0 \\ (Id+\rho) & \partial_\X^-
 \end{pmatrix} \right). \]
 Using a minus sign does not make any difference since we work on the field with two elements, but the definition holds in more general settings. Since $\rho$ is not a module map, so far, $\cono$ is an $\F$-vector space. However, we can endow it with a module structure. To do so, we will use the following well-known result.
 \begin{lemma}{\cite[Lemma A.3.8]{ozsvath2015grid}}\label{lem:induced_map}
 	Let $C, C', E, E'$ be four bigraded complexes, and suppose that there are chain maps fitting into the square:
 	\[\begin{tikzcd}
 		C \arrow{r}{f} \arrow[swap]{d}{\phi} & 
 		C' \arrow{d}{\phi'} \\ E \arrow{r}{g} & 
 		E'.
 	\end{tikzcd}
 	\]
 	that commutes up to homotopy; i.e. the map $\phi'\circ f$ is chain homotopic to the $g\circ \phi$.
 	Suppose moreover that $\phi$ and $\phi'$ are bigraded maps, $f$ and $g$ are homogeneous of bidegree $(m,t)$, and the homotopies are compatible with these gradings.
 	Then, there is an induced bigraded chain map $\Phi:\text{Cone}(f) \rightarrow \text{Cone}(g)$ that fits into the following commutative diagram of short exact sequences:
 	\[\begin{tikzcd}
 		0 \arrow{r}{} & C \arrow{r}{i} \arrow[swap]{d}{\phi} & \text{Cone}(f) \arrow{r}{p} \arrow[swap]{d}{\Phi}& 
 		C' \arrow{r}{} \arrow{d}{\phi'} & 0  \\ 
 		0 \arrow{r}{} & E \arrow{r}{j} & \text{Cone}(g) \arrow{r}{q} & E' \arrow{r}{} & 0.
 	\end{tikzcd}
 	\]
 	If $\phi$ and $\phi'$ are quasi-isomorphisms, so is $\Phi$.
 \end{lemma} 
 The proof of Lemma~\ref{lem:induced_map} shows how, if $h:C \rightarrow E'$ is a homotopy between $\phi'\circ f $ and $g \circ \phi$, then:
 \[ \Phi := \begin{pmatrix}
 	f & 0 \\ h & g
 \end{pmatrix}. \]
 
 Let $\G$ be a symmetric grid, fix any index $i=1,\dots,n$, and observe that the square:
 \begin{equation}\label{eq:square}
 	\begin{tikzcd}
 		\gc \arrow{r}{Id+\rho} \arrow[swap]{d}{V_i} & 
 		\gc \arrow{d}{V_i} \\ \gc \arrow{r}{Id+\rho} & 
 		\gc.
 	\end{tikzcd}
 \end{equation}
 commutes up to homotopy, since a straightforward computation shows that:
 \begin{align*}
 	(Id+\rho)\circ V_i + V_i \circ (Id + \rho) &= \rho \circ V_i + V_i \circ \rho = \rho \circ V_i + \rho \circ V_{\sigma(i)} \\&= \partial \circ\rho\circ h_{i\sigma(i)} +  \rho\circ h_{i\sigma(i)}\partial,
 \end{align*} 
 where $h_{i\sigma(i)}$ is a chain homotopy from $V_i$ to $V_{\sigma(i)}$.
 Let us briefly recall the definition of the chain homotopy (\cite[Lemma 4.6.9]{ozsvath2015grid}). 
 Two variables $V_i$ and $V_j$ are \emph{consecutive} via $X_{i}$ if there exists an $X$-marking, which we will call $X_{i}\in\X$, in the same row as $O_i$ and in the same column as $O_j$. Note that every variable is consecutive to exactly two other variables. Suppose that $V_i$ and $V_j$ are consecutive, then a homotopy from the multiplication by $V_i$ to the multiplication by $V_j$ is given by:
 \begin{equation}\label{eq:omotopia}
 	h_{ij}(\xx) = \sum_{\yy\in\St}\sum_{\left\{r\in \text{Rect}^\circ(\xx,\yy) \;|\; Int(r)\cap\X = X_i\right\}}V_1^{O_1(r)}\cdots V_n^{O_n(r)}\yy    
 \end{equation}
 where $\xx\in\St$. We call such a homotopy an \emph{elementary homotopy}.
 In the case where $V_i$ and $V_j$ are not consecutive, there exists a sequence of consecutive variables $V_i=V_{n_1},\dots,V_{n_k}=V_j$ (by a connection argument that we are allowed to use because the grid $\G$ represents a knot). A chain homotopy $h_{ij}$ from $V_i$ to $V_j$ is given by the sum of elementary homotopies:
 \[ h_{ij} = \sum_{c=1}^{k-1} h_{n_cn_{c+1}}. \] 
 \begin{figure}
 	\centering
 	\includegraphics[width=0.3\linewidth]{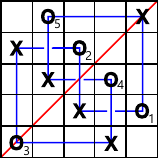}
 	\caption{Example of auxiliary enumeration.}
 	\label{img:aux_enum}
 \end{figure}
 \begin{remark}\label{rmk:2homot}
 	It follows immediately that, given two indices $i,j=1,\dots,n$, we have at least two possible chain homotopies from $V_i$ to $V_j$. Each one is given by one of the two possible ways to reach $O_j$ from $O_i$ through a path of consecutive variables that visits each variable at most once. 
 	
 	We now fix a notation to address such homotopies.
 	We start by fixing an auxiliary enumeration on the $O$-markings, based on the relation of consecutiveness. For each $i=1,\dots,n$, $O_i$ and $O_{n-1}$ are consecutive, and so are $O_n$ and $O_1$.
 	We only need to enumerate one of the $O$-markings; the enumerations of the others will follow.
 	Given a strongly invertible grid $\G$, consider the second marking on the axis, from NW to SE. If it is an $O$-marking, we set it as $O_n$. Otherwise, travel through the knot above the axis and set as $O_n$ the first $O$-marking met in this way (see Figure~\ref{img:aux_enum}).
 	For two distinct indices, wlog $i< j$, we will call:
 	\[ h_{ij} = \sum_{k=i}^{j-1} h_{k,k+1} \]
 	and $\widetilde{h}_{ij}$ the homotopy found adding all the complementary elementary homotopies.  
 	Since we are working with $\F_2$ coefficients, the homotopies are symmetric in the indices, therefore for any $i<j$ we can set:
 	\[ h_{ji} := h_{ij} \;\text{ and }\; \widetilde{h}_{ji} := \widetilde{h}_{ij}. \]
 	
 	It remains to show the same index case.
 	For $i=1,\dots,n-1$, we will call $h_{ii}$ the homotopy given by the zero map, and $\widetilde{h}_{ii}$ the homotopy given by the sum of all the $n$ elementary homotopies. The same holds for $h_{nn}$ and $\widetilde{h}_{nn}$ if $O_n$ is not on the NE to SW axis. When $O_n$ is on the axis, we make the opposite choice. Namely, $\widetilde{h}_{nn}$ will be the zero map, and $h_{nn}$ will be the sum of all the $n$ elementary homotopies.
 	This choice is needed to obtain chain homotopic multiplications on $\cono$, and will be clearer further on.
 \end{remark} 
 Now, through the maps $\{h_{ij},\widetilde{h}_{ij}\}$ and the square in Equation~\ref{eq:square}, we define two sets of multiplications on $\cono$.
 \begin{defn}
 	Let $\G$ be a symmetric grid and enumerate the $O$-markings according to the consecutiveness relation, as in Remark~\ref{rmk:2homot}. 
 	Let $i=1,\dots,n$ such that $O_i\neq O_{\sigma(i)}$. 
 	We call \emph{multiplication by} $\mathbf{V}_i$ and by $\widetilde{\mathbf{V}}_i$ respectively the maps
 	\[ \mathbf{V}_i,\widetilde{\mathbf{V}}_i:\cono \rightarrow\cono \]
 	defined as:
 	\[ \mathbf{V}_i = \begin{pmatrix}
 		V_i & 0 \\ \rho h_{i\sigma(i)} & V_i
 	\end{pmatrix} \;\text{ and }\;
 	\widetilde{\mathbf{V}}_i = \begin{pmatrix}
 		V_i & 0 \\ \rho \widetilde{h}_{i\sigma(i)} & V_i
 	\end{pmatrix}.\]
 \end{defn}
 \begin{remark}
 	Let $\G$ be a symmetric grid. For any index $i=1,\dots,n$, both multiplication by $\mathbf{V}_i$ and multiplication by $\widetilde{\mathbf{V}}_i$ are homogeneous maps of bidegree $(-2,-1)$.
 \end{remark}
 
 Either through the $\{\mathbf{V_i}\}_i$ of $\{\widetilde{\mathbf{V}}_i\}_i$ multiplications, we can endowe $\cono$ with a $\F[V_1,\dots,V_n]$-module structure. 
 
 The following analogue of \cite[Lemma 4.6.9]{ozsvath2015grid} holds. 
 
 \begin{lemma}\label{lem:multipl}
 	Let $\G$ be a symmetric grid. For any pair of integers $i,j \in \{1,\dots, n\}$, multiplication by $\mathbf{V}_i$ is chain homotopic to multiplication by $\mathbf{V}_j$, when thought of as homogeneous maps from $\cono$ to itself of degree $(-2,-1)$. The same holds for the $\widetilde{\mathbf{V}}_i$ multiplications.
 	\begin{proof}
 		With respect to the enumeration defined in Remark~\ref{rmk:2homot}, for every $i,j=1,\dots,n$ we set:
 		\[ H_{ij} = \begin{pmatrix} h_{ij} & 0 \\ 0 & h_{ij} \end{pmatrix}
 		:\text{Cone}_\G(Id+\rho) \rightarrow \text{Cone}_\G(Id+\rho). \]
 		Assume that $O_n$ does not lie on the NW to SE axis. Then
 		\[ \mathbf{V}_i + \mathbf{V}_j = \begin{pmatrix} V_{i}+V_j & 0 \\ \rho h_{i\rho(i)} + \rho h_{j\rho(j)} & V_{i}+V_j \end{pmatrix}, \]
 		and
 		\[ \partial H_{ij} + H_{ij}\partial = \begin{pmatrix} \partial h_{ij} + h_{ij}\partial & 0 \\ h_{ij}+\rho h_{ij} + h_{ij}+ h_{ij}\rho & \partial h_{ij} + h_{ij}\partial \end{pmatrix}. \]
 		A straightforward computation shows that $\rho\circ h_{ij}\circ\rho = h_{\rho(i)\rho(j)}$, hence:
 		\[ \partial H_{ij} + H_{ij}\partial = \begin{pmatrix} \partial h_{ij} + h_{ij}\partial & 0 \\ \rho h_{ij} + \rho h_{\rho(i)\rho(j)} & \partial h_{ij} + h_{ij}\partial \end{pmatrix}. \]
 		It only remains to show that:
 		\[ h_{ij} + h_{\rho(i)\rho(j)} + h_{i\rho(i)} + h_{j\rho(j)} = 0.\]
 		Independently from the order in which $O_i,$ $O_j,$ $O_{\rho(i)}$ and $O_{\rho(j)}$ appear in the auxiliary enumeration we set on the $O$-markings, by writing the four chain homotopies as the sum of elementary homotopies we find that each summand appears exactly two times, hence the sum is zero.
 		
 		If $O_n$ lies on the axis, then for any other index $i$ when we try to show that $\mathbf{V}_i + \mathbf{V_n}$ is equal to $\partial H_{in} + H_{in}\partial$ again the only non-trivial step is checking the bottom-left entry of the matrices. We must verify that:
 		\begin{equation}\label{eq:2sarkar}
 			\rho h_{i\rho(i)} + \rho h_{n\rho(n)} = \rho h_{in} + h_{in}\rho. 
 		\end{equation}
 		We have that:
 		\[ h_{n\rho(n)} = h_{nn} = h_{n1} + \sum_{k=1}^{n-1} h_{k,k+1} \,\text{ and }\, h_{in} \rho = \rho\left(h_{n1} + \sum_{k=1}^{\rho(i)-1}h_{k,k+1}\right). \]
 		So in Equation~\ref{eq:2sarkar} each of the $n$ elementary homotopies appears an even amount of times, hence everything simplifies. 
 		
 		For the $\widetilde{\mathbf{V}}_i$ multiplications, the cases and computations are analogous.
 	\end{proof}
 \end{lemma}
 
 We now know that the $\F[V_1,\dots,n]$-module structure we endowed $\cono$ with (either through $\mathbf{V}_i$ or through $\widetilde{\mathbf{V}}_i$ multiplications) induces a $\F[U]$-module structure on the homology $H(\cono)$.
 If instead of choosing one of the two multiplications sets, we consider $\cono$ as a $F[V_1,\dots,V_n,\widetilde{V}_1,\dots,\widetilde{V}_n]$-module, we can see that $H(\cono)$ inherits a $F[U,\widetilde{U}]$-module structure. As the following remark points out, we can be more precise.
 \begin{remark}\label{rmk:somma_sarkar_map}
 	The map we are referring to as the sum of all the $n$ elementary homotopies:
 	\[ \sum_{i=1}^{n-1} h_{i,i+1} + h_{n1}, \]
 	is a known map in literature. In our setting, it coincides with the map $\Psi=\Psi_1$ in \cite{sarkar2015moving}. It follows that, for any $i=1,\dots,n$:
 	\[ \mathbf{V}_i + \widetilde{\mathbf{V}}_i = \begin{pmatrix}
 		0 & 0 \\ \rho \Psi & 0
 	\end{pmatrix}. \]
 	We can hence apply \cite[Lemma 4.4]{sarkar2015moving} and use that $\Psi^2$ is chain homotopic to the identity to say that, in the homology $H(\cono)$, the multiplication by $(U + \widetilde{U})^2$ is the zero map.
 	In other words, we endowed $H(\cono)$ with a $\F[U,\widetilde{U}]/(U+\widetilde{U})^2$-module structure.
 	This structure coincides with the previous $\F[U]$-module and $\F[\widetilde{U}]$-module structures if the map $\Psi$ is chain homotopic to zero. 
 \end{remark}
 \begin{defn}
 	Let $\G$ be a symmetric grid. We call the (\emph{unblocked}) \emph{cone complex} of $\G$ the bigraded $\F[V_1,\dots,V_n]$-chain complex $(\cono,\partial)$ endowed with the $\mathbf{V}_i$'s multiplications. The (\emph{unblocked}) \emph{cone homology} of $\G$, denoted $\hc(\G)$, the homology of $(\cono, \partial)$ as a bigraded $\F[U]$-module, where the $U$ action is induced by the action of any $\mathbf{V}_i$. 
 \end{defn}
 The grid complex $\gc$ depends directly on the grid $\G$. In contrast, Theorem~\ref{thm:gh} establishes that its homology $\gh$ depends only on the underlying knot, making it a knot invariant. 
 Similarly, the cone complex $\cono$ is not an invariant of the knot represented by the grid $\G$. However, the cone homology is a strongly invertible knot invariant. The proof of this statement is presented in the following section.
 
 Let $\mathcal{R}$ be a ring and consider two bigraded chain complexes over $\mathcal{R}$: $C$ and $D$. It is a classical result (see for example \cite[Lemma A.3.2]{ozsvath2015grid}) that the mapping cone of a $\mathcal{R}$-linear chain map $\varphi:C \rightarrow D$ fits in a short exact sequence:
 \[ 0 \rightarrow C \xrightarrow{i} \text{Cone}(\varphi) \xrightarrow{p} D \rightarrow 0, \]
 where $i$ is bigradded and $p$ is $(-1,0)$-homogeneous.
 The same lemma also shows that the connecting morphism of the induced homology long exact sequence is $\varphi_* : H(C) \rightarrow H(D)$.
 Clearly, therefore the same results hold for $\cono$ at the vector space level. A simple observation reveals that we can still make the same statement if we consider the complexes as module complexes.
 
 \begin{prop}\label{prop:exact_seq}
 	Consider the $\F[V_1,\dots,V_n]$-module structure on the chain complex $\cono$ given either by the $\mathbf{V}_i$- or by the $\widetilde{\mathbf{V}}_i$-multiplications. 
 	The following short exact sequence of bigraded complexes over $\F[V_1,\dots,V_n]$ is exact:
 	\[ 0 \rightarrow \gc \xrightarrow{i} \cono \xrightarrow{p} \gc \rightarrow 0, \]
 	where $i(c)=(0,c)$ and $p(c',c)= c'$.
 	\begin{proof}
 		The classical result (\cite[Lemma A.3.2]{ozsvath2015grid}) allows us to say that the only property to check is that $i$ and $p$ are module maps.
 		This is a straightforward computation, for any index $k=1,\dots,n$:
 		\[ i (V_kc) = (0,V_kc) = \begin{pmatrix}
 			V_k & 0 \\ \rho h_{k\rho(k)} & V_k
 		\end{pmatrix} (0,c) = \mathbf{V}_k i(c), \]
 		and:
 		\[ V_k p (c',c) = V_kc' = p(V_kc', \rho h_{k\rho(k)}c' + V_{k}c) = p(\mathbf{V}_k(c',c)).   \]
 		The computations in the $\widetilde{\mathbf{V}}_k$ case are just as trivial. 
 	\end{proof}
 \end{prop}
 \begin{remark}\label{rmk:long}
 	In Proposition~\ref{prop:exact_seq}, the map $i$ is bigraded, while $p$ is $(-1,0)$-homogeneous. We obtain a long exact sequence in homology:
 	\[  \xrightarrow{p*} GH^-_{d,s}(K) \xrightarrow{(Id+\rho)_*} GH^-_{d,s}(K) \xrightarrow{i_*} \hc_{d,s}(K) \xrightarrow{p_*} GH^-_{d-1,s}(K) \xrightarrow{(Id+\rho)_*} . \]
 \end{remark}
 
 \subsection{Other flavors}\label{sec:flavours}
 In the classical case, alternative versions of the grid complex are defined, such as $\widetilde{GC}(\G)$ and $\widehat{GC}(\G)$. This section constructs the equivariant analogues of these complexes. The simplest is the analogue of $\widetilde{GC}(\G)$, which forgets the module structure.
 
 \subsubsection{The tilde cone complex}
 Consider the fully blocked grid complex $\widetilde{GC}(\G)$, that is, the complex generated over $\F$ by $\St$ together with the differential defined as:
 \[ \widetilde{\partial}(\xx) = \sum_{\yy\in\St} {\#\{r\in \text{Rect}^\circ(\xx,\yy) \;|\; r\cap \X = r\cap \Oo = \emptyset \} } \cdot \yy, \]
 for each $\xx\in\St$. 
 The fully blocked grid homology is $\widetilde{GH}(\G) = H(\widetilde{GC}(\G),\widetilde{\partial})$. 
 When $\G$ is symmetric, we can define a linear involution $\rho : \widetilde{GC}(\G) \rightarrow \widetilde{GC}(\G)$, that sends every generator $\xx \in \St$ to its reflection along the NE to SW axis,
 \begin{defn}
 	Let $\G$ be a symmetric grid. The \emph{fully blocked cone complex} is:
 	\[ \widetilde{\text{Cone}}_\G(Id + \rho) = \text{Cone}(Id + \rho:\widetilde{GC}(\G) \rightarrow \widetilde{GC}(\G)). \]
 	The \emph{fully blocked cone homology}, denoted $\widetilde{\hc}(\G)$, is the $\F$-vector space obtained as the homology of $\widetilde{\text{Cone}}_\G(Id + \rho)$.
 \end{defn}
 Note how, also in the fully blocked case, we obtain an exact short sequence:
 \[ 0 \rightarrow \widetilde{GC}(\G) \rightarrow \widetilde{Cone}_\G(Id +\rho ) \rightarrow \widetilde{GC}(\G) \rightarrow 0. \]
 
 \begin{figure}
 	\centering
 	\includegraphics[width=0.35\linewidth]{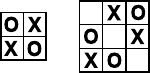}
 	\caption{Two symmetric grids representing the unknot: $\G$ and $\G'$.}
 	\label{img:gg'}
 \end{figure} 
 \begin{figure}
 	\centering
 	\begin{tikzpicture}
 		\draw (0,0) node[above right]{
 			\includegraphics[width=0.65\linewidth]{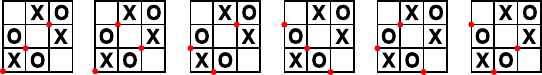}};
 		\draw (0.8,-0.10) node[] (a) {$\xx_1$};
 		\draw (2.22,-0.10) node[] (a) {$\xx_2$};
 		\draw (3.75,-0.10) node[] (a) {$\xx_3$};
 		\draw (5.16,-0.10) node[] (a) {$\xx_4$};
 		\draw (6.6,-0.10) node[] (a) {$\xx_5$};
 		\draw (8.07,-0.10) node[] (a) {$\xx_6$};
 	\end{tikzpicture}
 	\caption{States of the grid $\G'$.}
 	\label{img:x1_x6}
 \end{figure}
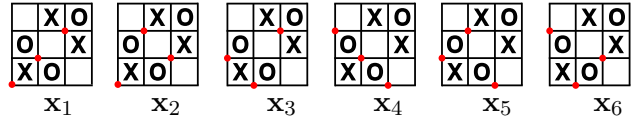
 
 \begin{remark}
 	As for the fully blocked grid homology, the fully blocked cone homology $\widetilde{\hc}(\G)$ does depend on the grid $\G$, so it is not a \si knot invariant. For example, consider the two symmetric grids representing the unknot in Figure~\ref{img:gg'}, respectively $\G$ and $\G'$. For the first grid, $\St$ consists of two symmetric states, hence fixed by $\rho$. It follows that $Id + \rho$ is the zero map, and $\widetilde{\text{Cone}}_\G(Id+\rho) \cong \widetilde{GC}(\G) \bigoplus \widetilde{GC}(\G) \cong \F^4$. For $\G'$, we need to be more careful. Referring to Figure~\ref{img:x1_x6}, there are six states. States $\xx_1,\dots,\xx_4$ are symmetric, and $\xx_5$ and $\xx_6$ are exchanged by $\rho$. Via a brief computation, we can see that the kernel of the differential is generated on the cone by:
 	\begin{align*} <&(\xx_2,0),(0,\xx_2),(\xx_3,0),(0,\xx_3),(\xx_4,0),(0,\xx_4), \\ &(\xx_5+\xx_6,0),(0,\xx_5+\xx_6),(0,\xx_5+\xx_1)>, 
 		\end{align*}
 	while a set of generators of the image of the differential is given by \[<(\xx_2,0),(0,\xx_2),(0,\xx_5+\xx_6)>.\] So the dimension of $\widetilde{\hc}(\G')$ is $6$.
 \end{remark}
 It is straightforward (\cite[Section 4.6]{ozsvath2015grid}) that there exist an isomorphism of $\F$-vector spaces:
 \[ \frac{\gc}{V_1=\dots=V_n=0} \cong \widetilde{GC}(\G).  \]
 It is interesting to observe that the same relation holds for the cone complex.
 
 \begin{prop}\label{prop:quozCon_Conquoz_tilde}
 	Let $\G$ be a symmetric grid. Then there exists an isomorphism of $\F$-vector spaces:
 	\[ \frac{\cono}{\mathbf{V}_1=\dots=\mathbf{V}_n=0} \cong \widetilde{Cone}_\G(Id+\rho). \]
 	\begin{proof}
 		By Proposition~\ref{prop:exact_seq}, there exists a short exact sequence (ignoring the shift):
 		\[ 0 \rightarrow \gc \xrightarrow{i} \cono \rightarrow \gc \rightarrow 0. \]
 		Consider the ideal $\text{I}=(\mathbf{V}_1,\dots,\mathbf{V}_n)$ of the base ring, generated by all the variable multiplications. The quotient by $\text{I}$ gives an exact sequence:
 		\[ \widetilde{GC}(\G) \xrightarrow{\iota} \cono/\text{I} \rightarrow\widetilde{GC}(\G) \rightarrow 0. \]
 		We aim to demonstrate that $\iota$ is indeed injective, thereby establishing a short exact sequence.
 		Since $i$ is injective, it suffices to show that $ \text{Im}(i) \cap \text{I}\cdot \cono = i(\text{I}\cdot\gc)$. The $\supseteq$ inclusion is obvious because $i$ is a module map; hence, we shall focus on the $\subseteq$ inclusion. 
 		An element $\text{x}\in \text{Im}(i) \cap \text{I}\cdot \cono$ is necessarily of the form:
 		\[ \text{x} = \left( 0 , \sum_{i=1}^n V_if_i\right), \quad f_i\in\gc \;\forall\,i. \]
 		It follows immediately that $\iota$ is injective. 
 		
 		Consider the following commutative diagram with exact rows:
 		\[
 		\begin{tikzcd}
 			0 \arrow{r}{} & \widetilde{GC}(\G) \arrow{r}{Id+\rho} \arrow[swap]{d}{Id} & 
 			\widetilde{\text{Cone}}_\G(Id+\rho) \arrow{r}{} \arrow{d}{\phi} & \widetilde{GC}(\G) \arrow{r}{} \arrow{d}{Id} & 0 \\ 
 			0 \arrow{r}{} & \widetilde{GC}(\G) \arrow{r}{Id+\rho} & 
 			\cono/\text{I} \arrow{r}{} & \widetilde{GC}(\G) \arrow{r}{} & 0,
 		\end{tikzcd}\]
 		where, $\phi(\text{x},\text{y}) = [\text{x},\text{y}]$ (we are using the fact that we can think $\widetilde{GC}(\G)$ as a subset of $\gc$). We conclude with the Five Lemma.
 	\end{proof}
 \end{prop}
 
 \subsubsection{The hat cone complex}
 The next flavor of the cone is constructed directly from $\cono$.
 The simply blocked grid complex is defined (see \cite{manolescuoz}) as the quotient of the unblocked grid complex by the multiplication by one variable $V_i$. The homology of the resulting complex does not depend, up to isomorphism, on the index $i$ (see \cite[Corollary 4.6.17]{ozsvath2015grid}). We proceed analogously. 
 \begin{defn}
 	Let $\G$ be a symmetric grid of size $n$ and fix any index $i=1,\dots,n$. The \emph{simply blocked cone complex} is the quotient complex \[\cono/\mathbf{V}_i,\] denoted $\hcono$. The \emph{simply blocked cone homology}, denoted $\widehat{\hc}(\G)$, is the $\F$-vector space obtained as the homology of $\hcono=(\cono/\mathbf{V}_i,\partial)$.
 \end{defn}
 Note that, thanks to Lemma~\ref{lem:multipl}, we are really defining only one object, as the choice of the index $i$ does not really make a difference. By writing this, we mean the following.
 \begin{prop}
 	Let $\G$ be a symmetric grid.
 	The simply blocked cone homology:
 	\[ \widehat{\hc}(\G) = H\left( \cono/\mathbf{V}_i \right) \]
 	does not depend on the choice of the index $i=1,\dots,n$.
 	\begin{proof}
 		The map $\textbf{V}_i:\cono \rightarrow \cono$ is injective. It follows (\cite[Lemma A.3.9]{ozsvath2015grid}) that the quotient $\cono/\textbf{V}_i$ is quasi-isomorphic to:
 		\[ \text{Cone}(\textbf{V}_i:\cono \rightarrow \cono). \]
 		Since by Proposition~\ref{lem:multipl} the $\textbf{V}_i$-multiplications are chain homotopic to each other, the isomorphism class of the latter complex does not depend on $i$ (\cite[Lemma A.3.7]{ozsvath2015grid}). This concludes the proof.
 	\end{proof}
 \end{prop}
  
 Observe that, for each $i=1,\dots,n$, there is an induced involution:
 \[ \rho: \frac{\gc}{V_i} \rightarrow \frac{\gc}{V_{\rho(i)}}. \]
 To define a map $Id+\rho$  between quotients of $\gc$, we need an index $i=1,\dots,n$ such that $\rho(i)=i$. 
 Up to stabilizations on the axis, one can always assume to have such a variable on a symmetric grid $\G$, paying attention to the fact that such a move changes the Legendrian type of the knot represented by $\G$.
 Call $\text{Cone}_\G(\overline{Id + \rho})$ the mapping cone of:
 \[ Id + \rho : \frac{\gc}{V_i} \rightarrow \frac{\gc}{V_i}. \]
 Given Proposition~\ref{prop:quozCon_Conquoz_tilde}, one could wonder if an analogue property holds also for the simply blocked complex. The non-diagonal structure of the multiplications on the cone complex implies that the answer would be negative.
 
 We consider $\hcono$ to be a more natural object in light of the following analogous of \cite[Proposition 4.6.18]{ozsvath2015grid}.
 \begin{prop}\label{prop:seq_esatta_meno-hat}
 	There is a long exact sequence relating $\hcono$ and $\cono$:
 	\[ \cdot\cdot \rightarrow \hc_{d+2,s+1}(\G) \xrightarrow{U} \hc_{d,s}(\G) \rightarrow \widehat{\hc}_{d,s}(\G) \rightarrow \hc_{d+1,s+1}(\G) \rightarrow \cdot\cdot . \] 
 	\begin{proof}
 		The wanted sequence is the long exact sequence associated with the short exact sequence:
 		\[  0 \rightarrow \cono \xrightarrow{\mathbf{V}_i} \cono \rightarrow \hcono \rightarrow 0.\]
 	\end{proof}
 \end{prop}
 Section~\ref{sec:prova_inv} is devoted to proving that $\hc(\G)$ only depends on the topological (unoriented) type of the underlying knot, and not on $\G$. Once one knows this, Proposition~\ref{prop:seq_esatta_meno-hat} shows that also $\widehat{\hc}(\G)$ is an invariant of \si knots.

 \subsection{Invariance of the cone homology}\label{sec:prova_inv}
 This section is devoted to the proof of the invariance of the $\hc(\G)$ module that we defined in the previous section.
 To establish this result, two preliminary lemmas are presented.
 Notice that in the following lemma we use the fact that the coefficients lie in the field with two elements.
 
 \begin{lemma}\label{lemma:commut_F_h}
 	Let $\G,\G'$ be two grid diagrams of size $n$. Let $f:\gc \rightarrow GC^-(\G')$ be the quasi-isomorphism induced by a commutation move on $\G$. Fix two indices $i,j=1,\dots,n$ such that the homotopy $h_{ij}$ between the multiplications $V_i$ and $V_j$ is an elementary homotopy.
 	Then $f\circ h_{ij}$ and $h_{ij} \circ f$ are chain homotopic.
 	\begin{proof}
 		We briefly recall the conventions in \cite[Section 5.1]{ozsvath2015grid} for the commutation induced map.
 		Assume that the commutation is a column commutation.
 		We view both $\G$ and $\G'$ simultaneously on the same torus.
 		Regard the $O$- and $X$-markings as fixed for both grids and draw two vertical circles (one for $\G$ and one for $\G'$) curved, as in Figure\ref{fig:simultanea}. The horizontal circles will be $\aalpha=\{ \alpha_1,\dots, \alpha_n \}$ for both $\G$ and $\G'$.
 		The set of vertical circles will be $\bbeta = \{\beta_1 ,\dots, \beta_n \}$ for $\G$ and $\ggamma= \{ \beta_1 ,\dots, \beta_{i-1}, \gamma_i ,\beta_{i+1} ,\dots, \beta_n \}$ for $\G'$. 
 		Draw $\beta_i$ and $\gamma_i$ so that they intersect transversally in two points $a$ and $b$.
 		Name the points so that of the two bigons delimited by $\beta_i$ and $\gamma_i$, the one with $\beta_i$ as western boundary has $a$ as southern tip. 
 		The map $f$ is defined in terms of specific domains called pentagons. 
 		Fix $\xx \in \St$ and $\yy \in \Stt$. A \emph{pentagon from $\xx$ to $\yy$} (\cite[Definition 5.1.1]{ozsvath2015grid}) is an embedded disk $p$ in the torus whose boundary is the union of five arcs, each of which lies on some $\alpha_j$, $\beta_j$, or on $\gamma_i$, that satisfies the following conditions:
 		\begin{itemize}
 	         \item Four of the corners of $p$ are in $\xx \cup \yy$.
 	         \item Note that each corner lies in the intersection between two curves from $\aalpha \cup \bbeta \cup \gamma_i$. A small circle centered in any corner is hence divided into four quadrants by those two curves. The pentagon $p$ contains exactly one of the four quadrants.  
 	         \item Calling $\partial_{\alpha}p = p \cap \aalpha$: $$ \partial( \partial_\alpha p ) = \yy- \xx. $$		
 		\end{itemize}
 		We call $\text{Pent}(\xx,\yy)$ the set of pentagons from $\xx$ to $\yy$. The conditions imply that the fifth corner is $a$. As usual, a pentagon $p \in \text{Pent}(\xx,\yy)$ such that
 		\[ p \cap (\xx \cup \yy) = \emptyset \]
 		is called empty, and we denote by $\text{Pent}^\circ(\xx,\yy)$ the set of such pentagons (see Figure~\ref{fig:pentagono}).
 		 
 		  \begin{figure}
 		 	\begin{subfigure}{\linewidth}
 		 		\centering
 		 		\begin{tikzpicture}
 		 			\draw (0,0) node[above right]{\includegraphics[width=0.33\linewidth]{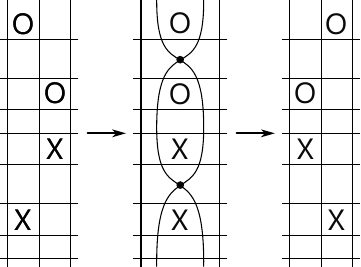}};
 		 			\draw (0.6,3.55) node[scale=1.3] (a) {$\G$};
 		 			\draw (4.08,3.55) node[scale=1.3] (a) {$\G'$};
 		 			\draw (2.5,1.1) node[scale=0.75] (a) {$a$};
 		 			\draw (2.53,2.57) node[scale=0.75] (a) {$b$};
 		 			\draw (2.48,3.45) node[scale=0.7] (a) {$\beta_i$};
 		 			\draw (2.07,3.45) node[scale=0.7] (a) {$\gamma_i$};
 		 		\end{tikzpicture}
 		 		\caption{The diagram in the center encodes simultaneously the commutation between $\G$ and $\G'$.}
 		 		\label{fig:simultanea}
 		 	\end{subfigure}
 		 	\begin{subfigure}{\linewidth}
 		 		\centering
 		 	    \begin{tikzpicture}
 		 			\draw (0,0) node[above right]{\includegraphics[width=0.33\linewidth]{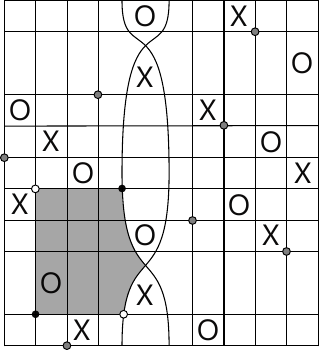}};
 		 			\draw (2.3,1.25) node[scale=0.75] (a) {$a$};
 		 			\draw (2.3,4.18) node[scale=0.75] (a) {$b$};
 		 			\draw (2.42,5) node[scale=0.7] (a) {$\beta_5$};
 		 			\draw (1.85,5) node[scale=0.7] (a) {$\gamma_5$};
 		 		\end{tikzpicture}
 		 		\caption{The dots represent points of the grid states $\xx$ and $\yy$. The shaded ones are shared, the black one belong to $\xx$ and the empty one belong to $\yy$. The shaded region is a pentagon $p \in \text{Pent}^\circ(\xx,\yy)$.}
 		 		\label{fig:pentagono}
 		 	\end{subfigure}
 		 	\caption{}
 		 	\label{fig:simultanei}
 		 \end{figure}
 		    
  		For any $\xx \in \St$:
 		\[ f(\xx) = \sum_{\yy\in\St}\sum_{\{p\in\text{Pent}^\circ(\xx,\yy)\;|\; p\cap\X = \emptyset \}} V_1^{O_i(p)}\dots V_{n}^{O_n(p)}\yy, \] and \[ h(\xx) = \sum_{\yy\in\St}\sum_{\{r\in\text{Rect}^\circ(\xx,\yy)\;|\; r\cap\X = X_i \}} V_1^{O_i(r)}\dots V_{n}^{O_n(r)}\yy. \]
 		So $h_{ij}$ counts empty rectangles that intersect the $X$-marking exactly in $X_i$, while $f$ counts empty pentagons with a corner at $a$, which do not intersect the $X$-markings (see Figure~\ref{fig:pentagono}). The boundary maps $\partial:\gc \rightarrow \gc$ and $\partial':GC^-(\G') \rightarrow GC^-(\G)$ both count empty rectangles that do not intersect the $X$-markings, and as the homotopy we will consider the map $H:\gc \rightarrow GC^-(\G')$ that counts empty pentagons that intersect the $X$-markings exactly  in $X_i$. More precisely, $H$ is defined on each generator $\xx \in \St$ as:
 		\[ H(\xx) = \sum_{\yy\in\St}\sum_{\{p\in\text{Pent}^\circ(\xx,\yy)\;|\; p\cap\X = X_i \}} V_1^{O_i(p)}\dots V_{n}^{O_n(p)}\yy. \]
 		Our claim is that, as maps of $\F[V_1,\dots,V_n]$-modules:
 		\[ \Phi := h_{ij} \circ f + f\circ h_{ij} + H\circ \partial + \partial'\circ H = 0. \]
 		The proof of this claim follows the same steps as the proof of \cite[Lemma 5.1.4]{ozsvath2015grid}, which shows that the commutation map $f$ is a chain map. We will now retrace those steps, adapting them to our claim.
 		
 		For a domain $\psi$, call $N(\psi)$ the number of decompositions it admits as:
 		\begin{itemize}
 			\item a rectangle containing $X_i$ (for $\G$) and an empty pentagon; or
 			\item an empty pentagon and a rectangle containing $X_i$ (for $\G'$); or
 			\item an empty rectangle (for $\G$) and a pentagon containing $X_i$; or
 			\item a pentagon containing $X_i$ and an empty rectangle (for $\G'$).
 		\end{itemize}
 		Then:
 		\[  \Phi(\xx) = \sum_{\zz\in\mathbf{S}(\G')} \sum_{\{\psi\in\pi(\xx,\zz) \;|\; \X\cap\psi = X_i\}} N(\psi) V_1^{O_1(\psi)}\dots V_n^{O_n(\psi)} \zz. \]
 		When $N(\psi) > 0$, the domain $\psi$ falls into one of the following cases:
 		\begin{enumerate}
 			\item $|\xx \setminus (\xx\cap \zz)|=4$. We have two cases, in both of them $N(\psi)=2$. In both cases, $\psi$ decomposes as $r*p$ and $p'*r'$, where $r$ and $r'$ (resp. $p$ and $p'$) share the underlying rectangle (resp. pentagon), that is, they differ only in the grid state they connect. The only difference in the two cases is that in the first case $r$ is empty and $p$ contains $X_i$, while in the second case $r$ contains $X_i$ and $p$ is empty.
 			\item $|\xx \setminus (\xx\cap \zz)|=3$. Assume the local multiplicity of $\psi$ is at most $1$ on the grid. Then $\psi$ has seven corners, where one is a $270^\circ$ corner. We decompose $\psi$ in two ways by cutting at the $270^\circ$ corner along two different directions. In both cases, we obtain a pentagon and a rectangle, but the precise order of the two and which one of the two contains $X_i$ depends on the geometry of $\psi$. If some local multiplicity is $2$, $\psi$ has a $270^\circ$ corner at $a$ and cutting there along the two directions gives two decompositions. 
 			\item $|\xx \setminus (\xx\cap \zz)|=1$. If so, $\psi$ wraps around the torus, forming a thin annulus. Here, by \emph{thin} we mean that the annulus' width is one, otherwise it would necessarily cross an $X$-marking different from $X_i$. There are two cases, in the first $\psi$ is a horizontal annulus minus a small triangle and in the second $\psi$ is the union of a small triangle and a vertical annulus.      
 			In the first case, $N(\psi)=2$. We find the two decompositions by cutting either along $\beta_i$ or along $\gamma_i$.
 			In the second case, the decomposition is unique. For any such vertical annulus there exists exactly one other vertical annulus that gives the same contribution to $\Phi(\xx)$. For more details, see \cite[Figure 5.6]{ozsvath2015grid}, as this case is the analogous of case \textbf{(P-3)(v)} in \cite[Lemma 5.1.4]{ozsvath2015grid}.
 		\end{enumerate}
 		Since we work on $\F_2$ coefficients, this concludes the proof.
 	\end{proof} 
 \end{lemma}
 \begin{lemma}\label{lemma:stab_F_h}
 	Let $\G,\G'$ be two grid diagrams. Let $f:\gc \rightarrow GC^-(\G')$ be the quasi-isomorphism induced by a (de)stabilization move on $\G$. Let $n$ be the size of $\G$, and fix two indices $i,j=1,\dots,n$ such that the homotopy $h_{ij}$ between the multiplications $V_i$ and $V_j$ is an elementary homotopy.
 	Depending on its type, the (de)stabilization introduces or removes one or two $X$-markings: $X_k$ and $X_h$ (eventually $h=k$).
 	Assume that $$\{h,k\} \cap \{i,j\} = \emptyset.$$
 	Then $f\circ h_{ij}$ and $h_{ij} \circ f$ are chain homotopic.
 	\begin{proof}
 		We claim that, unlike the previous lemma, the two commutations commute. We work in the case when $f:\gc \rightarrow GC^-(G')$ is an X:SW destabilization as in Figure~\ref{img:destab_xsw}. The computations in the remaining cases are analogous. Observe that, by hypothesis, $\{i,j\}$ is disjoint from $\{k-1,k\}$. Up to changing the enumeration on the markings, we assume $k=2$.
 		 \begin{figure}
 			\centering
 			\begin{tikzpicture}
 				\draw (0,0) node[above right]{
 					\includegraphics[width=0.48\linewidth]{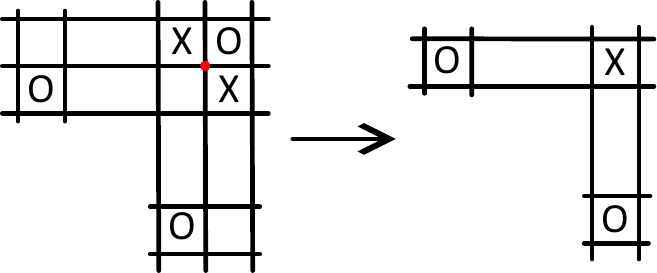}};
 				\draw (1.95,1.95) node[color=red,scale=1] (a) {$c$};
 			\end{tikzpicture}
 			\caption{X:SW-destabilization at $c$.}
 			\label{img:destab_xsw}
 		\end{figure}
 		
 		Split $\St = \text{I}(\G) \sqcup \text{N}(\G)$ into the states $\text{I}(\G)$ that contains the intersection point $c$ (see Figure~\ref{img:destab_xsw}) and the complementary $\text{N}(\G)$. Consequentially, consider the $\F[V_1,\dots,V_n]$-module splitting $\gc = \text{I} \oplus \text{N}$. Any rectangle from a state $\xx\in \text{N}$ to a state $\yy\in \text{I}$ must contain an $X$-marking, hence the submodule $\text{N}$ is a subcomplex and we can write the differential as:
 		\[ \partial^-_{\X} = \begin{pmatrix}
 			\partial^\text{I}_\text{I} & 0 \\ \partial^\text{N}_\text{I} & \partial^\text{N}_\text{N}
 		\end{pmatrix}. \]
 		It follows that $\gc$ is the mapping cone of the map:
 		$ \partial^\text{N}_\text{I}$ from $(\text{I},\partial^\text{I}_\text{I})$ to $(\text{N},\partial^\text{N}_\text{N})$.
 		
 		By \cite[Section 5.2]{ozsvath2015grid}, we know that $f$ is the composition of the following two quasi-isomorphisms:
 		\begin{align*}
 			\text{Cone}(\partial^\text{N}_\text{I}:\text{I}\rightarrow \text{N}) & \rightarrow \text{Cone}(V_2-V_1:GC^-(\G')[V_1]\rightarrow GC^-(\G')[V_1]) \\
 			(\xx,\yy) & \mapsto (e(\xx),e\circ\mathcal{H}_{X_2}(\yy))
 		\end{align*}
 		where $e:\textbf{I}\rightarrow GC^-(\G')$ is the natural identification of the states, and
 		\begin{align*}
 			\text{Cone}(V_2-V_1) &\rightarrow GC^-(\G')[V_1]/(V_1-V_2) \cong GC^-(\G') \\
 			(a,b)&\mapsto \overline{a}.
 		\end{align*} 
 		Consider the homotopy $h_{ij}:\gc \rightarrow  \gc$ from the multiplication by $V_i$ to the multiplication by $V_j$. We can write:
 		\[h_{ij}=\begin{pmatrix}
 			h_{ij}|_\text{I} & \mathcal{H}_{X_i}^\text{N} \\ \mathcal{H}_{X_i}^\text{I} & h_{ij}|_\text{N}
 		\end{pmatrix}: \text{Cone}(\partial^\text{N}_\text{I}) \rightarrow \text{Cone}(\partial^\text{N}_\text{I}).\]
 		Given a map $\phi$ taking values in $GC^-(\G')[V_1]$, call $\overline{\phi}$ its composition with the projection $GC^-(\G')[V_1] \rightarrow GC^-(\G')[V_1]/(V_1-V_2)= GC^-(\G')$. The claim coincides with the commutativity of the following diagram:
 		\[  \begin{tikzcd}
 			\text{Cone}(\partial_\X^-) \arrow{r}{h_{ij}} \arrow[swap]{d}{\overline{(e,0)}} & 
 			\text{Cone}(\partial_\X^-) \arrow{d}{\overline{(e,0)}} \\ GC^-(\G') \arrow{r}{h'_{ij}} & 
 			GC^-(\G'),
 		\end{tikzcd} \]
 		where $h'_{i,j}$ is the homotopy between multiplications $V_i$ and $V_j$ in the complex $\gcc$.
 		What we really want to do is to show that $h'_{ij}\circ e$ is chain homotopic, in the quotient, to $e\circ h_{ij|_\text{I}} + e\circ \mathcal{H}_{X_i}^\text{N}$.
 		
 		Consider the map $\mathcal{H}_{X_i}^\text{N}:\textbf{N}\rightarrow \textbf{I}$. For each $\xx\in\textbf{N}(\G)$ and each $\yy\in \mathbf{I}(\G)$, $\mathcal{H}_{X_i}^\text{N}$ counts empty rectangles from $\xx$ to $\yy$ that intersect the $X$-markings exactly in $X_i$. Note from Figure~\ref{img:destab_xsw} that, since the point $c$ belongs to $\yy$, any empty rectangle from $\xx$ to $\yy$ must intersect the $X$-markings either in $X_2$ or in $X_1$. It follows that $\mathcal{H}_{X_i}^\text{N}$ must be the zero map.
 		
 		We will now show that
 		$e\circ h_{ij}|_\text{I} = h'_{ij} \circ e.$
 		For any $\xx \in \mathbf{I}(\G)$, then:
 		\[ h_{ij}|_\text{I}(\xx) = \sum_{\yy\in\mathbf{I}(\G)} \sum_{\{ r\in \text{Rect}^\circ(\xx,\yy) \;|\; r\cap \X=X_i \}} V_1^{O_1(r)} \cdots V_n^{O_n(r)}\yy, \]
 		and:
 		\[ h'_{ij}(e(\xx)) = \sum_{e(\yy)\in\mathbf{S}(\G')} \sum_{\{ r'\in \text{Rect}^\circ(e(\xx),e(\yy)) \;|\; r\cap \X=X_i \}} V_2^{O_2(r)} \cdots V_n^{O_n(r)}e(\yy). \]
 		For each $\xx,\yy\in\mathbf{I}(\G)$, the bijection between $\mathbf{I}(\G)$ and $\mathbf{S}(\G')$ induces a bijection between $\text{Rect}^\circ(\xx,\yy)$ on $\G$ and $\text{Rect}^\circ(e(\xx),e(\yy))$ on $\G'$. We will also refer to this bijection as $e$. The claim follows showing that for any $i=1,\dots,n$ and for any $r\in \text{Rect}^\circ(\xx,\yy)$, we have that:
 		\[ O_i(r) = O_i(e(r)), \]
 		here we are implicitly setting $O_1(e(r))=0$ for every rectangle, since the grid $\G'$ lacks the index $1$ $X$-marking.
 		
 		For $i=2,\dots,n$, it is clear that the bijection $e$ preserves the multiplicity of $O_i$ in the rectangles. 
 		Hence, we end by observing that, for every $r\in \text{Rect}^\circ(\xx,\yy)$ for $\xx,\yy \in \mathbf{I}(\G)$, we have $O_1(r) = 0$. Suppose by contradiction that $r\in \text{Rect}^\circ(\xx,\yy)$ is such that $O_1(r) = 0$. Since one of the corners of the square containing $O_1$ is exactly the point $c$, it follows that $c$ must be one of the four corners of $r$ (see Figure~). This is absurd, in fact, by definition of $\mathbf{I}(\G)$, the point $c$ belongs to both $\xx$ and $\yy$, so $c$ cannot belong to $\partial r$.      
 	\end{proof} 
 \end{lemma} 
 
 The previous lemmas hold in the general, non-equivariant case. They are two technical results that we will apply to prove the invariance of cone homology.
 
 \begin{lemma}\label{lemma:F_h}
 	Let $\G,\G'$ be two symmetric grids. Let $F:\gc \rightarrow GC^-(\G')$ be an equivariant grid move and let $c_1,\dots,c_d$ be a sequence of non-equivariant grid moves (hence, row/column-commutations and (de)stabilizations) such that: $F= c_d \circ \dots \circ c_1$. Fix a marking point $O_k$ and consider the homotopy $h_{k\rho(k)}$ between the multiplication by $V_k$ and the multiplication by $V_{\rho(k)}$ on $\gc$. Write $h_{k\rho(k)}$ as sum of elementary homotopies: $h_{k\rho(k)} = \sum_{i=1}^m h_{i,i+1}$.
 	
 	Assume that, for every possible choice of indices $i=1,\dots,d$ and $j=1,\dots,m$, the pair $(c_i,h_{j,j+1})$, of a grid move and an elementary homotopy, either satisfies the hypothesis of Lemma~\ref{lemma:commut_F_h} or of Lemma~\ref{lemma:stab_F_h}. Then the homotopy $h'_{k\rho(k)}$, between the multiplications by $V_k$ and by $V_{\rho(k)}$ on $GC^-(\G')$, is well-defined, and the compositions $F\circ h_{k\rho(k)}$ and $h'_{k\rho(k)} \circ F$ are chain homotopic.
 	\begin{proof}
 		The proof follows directly from the preceding lemmas.
 	\end{proof}
 \end{lemma}
 The main result can now be stated and proved.
 \begin{thm}\label{thm:invariance}
 	Let $K$ be a \si knot and $\G$ be a grid diagram such that $L(\G) = K$. The homology of $\text{Cone}_{\G}(Id+\rho)$ only depends on the \si type of $K$.
 	\begin{proof}
 		It is enough to exhibit an isomorphism $\hc(\G) \rightarrow\hc(\G')$ when $\G$ and $\G'$ are two symmetric grids linked by an equivariant grid move, which induces a quasi-isomorphism $F:\gc \rightarrow GC^-(\G')$.
 		Consider the following square:
 		\[  \begin{tikzcd}
 			\gc \arrow{r}{Id+\rho} \arrow[swap]{d}{F} & 
 			\gc \arrow{d}{F} \\ GC^-(\G') \arrow{r}{Id+\rho'} & 
 			GC^-(\G'),
 		\end{tikzcd}. \]
 		By Theorem~\ref{prop:key}, we know that the compositions $\rho'\circ F$ and $F \circ \rho$ are chain homotopic, via a homotopy that we call $\phi$. Hence, there is an induced vector space quasi-isomorphism $\text{Cone}F : \cono \rightarrow \cono$, defined as:
 		\[ \text{Cone}F = \begin{pmatrix}
 			F & 0 \\ \phi_F & F
 		\end{pmatrix}. \]
 		Here we posed $\phi_F= \rho \circ \phi$.
 		To conclude, we must show that $\text{Cone}F_*: \hc(\G) \rightarrow \hc(\G')$ is a module map. To be more precise, we are going to show that it is a $\F[U,\overline{U}]$-module map.
 		
 		Our aim is now to show that, there exist indices $i,j\in\{1,\dots,n\}$ such that $\mathbf{V}_j\circ \text{Cone}F$ is chain homotopic to $\text{Cone} F \circ \mathbf{V}_i$ and indices $i',j'\in\{1,\dots,n\}$ such that $\overline{\mathbf{V}}_{j'} \circ \text{Cone}F$ is chain homotopic to $\text{Cone} F \circ \overline{\mathbf{V}}_{i'}$. By Lemma~\ref{lem:multipl}, this would conclude. In particular, we will find an index $i=1,\dots,n$ such that:
 		\[ \mathbf{V}_i\text{Cone} F  \sim \text{Cone} F\, \mathbf{V}_i \quad\text{ and }\quad \overline{\mathbf{V}}_{i}  \text{Cone}F \sim  \text{Cone}F \; \overline{\mathbf{V}}_{i}. \]
 		
 		Note that \cite[Lemma 4.3]{sarkar2015moving} states that $\Psi F \sim F \Psi$, hence for each $i=1,\dots,n$:
 		\[ (\mathbf{V}_i + \widetilde{\mathbf{V}}_i) \text{Cone}F \sim \text{Cone}F (\mathbf{V}_i + \widetilde{\mathbf{V}}_i). \]
 		It follows that it is enough to show the thesis either for $\mathbf{V}_i$ or for $\widetilde{\mathbf{V}}_i$.
 		
 		We start by writing $F$ as a composition of non-equivariant grid moves: $F= c_n\circ\dots\circ c_1$. 
 		
 		The \si grid move $F$ can be an equivariant row or column commutation, an equivariant (de)stabilization, or a move from Figure~\ref{img:silgm}. In each case, one can easily see that only up to 3 of the $c_i$'s are stabilizations, so if the grid is big enough, there exist an index $\ii$ such that either the pair $(F,h_{\ii\rho(\ii)})$ or the pair $(F,\widetilde{h}_{\ii\rho(\ii)})$ satisfies the hypothesis of Lemma~\ref{lemma:F_h}. The thesis for small grids is trivial and can be directly shown. 
 		
 		Assume that $Fh_{\ii\rho(\ii)} \sim Fh_{\ii\rho(\ii)}$, the other case is perfectly symmetric.
 		Observe that $\mathbf{V}_\ii\circ \text{Cone}F + \text{Cone} F \circ \mathbf{V}_\ii$ is equal to:
 		\[  \begin{pmatrix}
 			V_\ii F & 0 \\ \rho h_{\ii\rho(\ii)}F + V_\ii\phi_F & V_\ii F
 		\end{pmatrix} + \begin{pmatrix}
 			FV_\ii  & 0 \\ F\rho h_{\ii\rho(\ii)} + \phi_FV_\ii & FV_\ii
 		\end{pmatrix}. \]
 		Since $F$ is an $\F[V_1,\cdots,V_n]$-module map, the only non-zero component of the matrix is the bottom-left one. It reads:
 		\begin{equation}\label{eq:asd}
 			V_\ii\phi_F + \phi_FV_\ii + F\rho h_{\ii\rho(\ii)} + \rho h_{\ii\rho(\ii)}F.
 		\end{equation}  
 		Now we will compute the terms of the above sum to show that it is null-homotopic, thereby establishing the thesis. Since $\phi_F=\rho\circ\phi$, we have that $V_\ii\phi_F = \phi_F V_{\rho(\ii)}$. It follows that:
 		\[ V_\ii\phi_F + \phi_FV_\ii = \phi_F (V_{\rho(\ii)} - V_\ii) = \phi_F \partial h_{\ii\rho(\ii)} + \phi_F h_{\ii\rho(\ii)}\partial. \]

 		For the terms $F\rho h_{\ii\rho(\ii)} + \rho h_{\ii\rho(\ii)}F$, remember that $h_{\ii\rho(\ii)}F \sim Fh_{\ii\rho(\ii)}$.
 		We can hence state that:
 		\[ F\rho h_{\ii\rho(\ii)} + \rho h_{\ii\rho(\ii)}F = F\rho h_{\ii\rho(\ii)} + \rho F h_{\ii\rho(\ii)} = \phi_F \partial h_{\ii\rho(\ii)} + \partial \phi_F h_{\ii\rho(\ii)},  \]
 		where the second equality holds because $\phi_F$ is an homotopy between $F\rho$ and $\rho F$.
 		We found that Equation~\ref{eq:asd} equals to:
 		\[ \phi_F \partial h_{\ii\rho(\ii)} + \phi_F h_{\ii\rho(\ii)}\partial + \phi_F \partial h_{\ii\rho(\ii)} + \partial \phi_F h_{\ii\rho(\ii)}. \]
 		The two repeated terms on the right-hand side simplify each other, leaving us with a null-homotopic map. This implies the thesis for $\mathbf{V}_\ii$. As we already observed, this implies that the thesis also holds for $\widetilde{\mathbf{V}}_\ii$, so we conclude.
 	\end{proof}
 \end{thm}

 \subsection{Extraction of further invariants} \label{sec:altri_inv}
 The cone homology $\hc(K)$ of a strongly invertible knot constitutes a refined invariant. Simpler invariants can be extracted from it.
 
 \subsubsection{Tau invariant}
 Let us start by settling on some notation. Let $C$ be a module on the ring $\F[U]$, we call \emph{torsion} of $C$ the submodule given by
 \[ Tors = Tors (C) = \{ m\in C \;|\; \exists r\in \F[U]\setminus\{0\} \text{ such that } r\cdot m = 0 \}. \]
 If $C$ is bigraded, as are our modules, then the quotient $C/Tors$ inherits a bigrading. Since the modules with which we work are finitely generated, there exists an integer $r$ such that $C/Tors \cong \F[U]^r$. We call $r$ the \emph{rank} of $C$: $rk C = r$.  
 
 Let $C$ be a finitely generated, bigraded $\F[U]$ module of rank at least $1$. Then there exists a decomposition: 
 \[ C \cong \bigoplus_{k=1}^{rk C}F_i \oplus Tors, \] 
 where $F_k \cong \F[U]$ for all $k$.
 Since the multiplication by $U$ is homogeneous of degree $(-2,-1)$, the following definition makes sense.
 \begin{defn}
 	We call $F_k$ a \emph{tower of degree} $i$, or $i$-\emph{tower}, if it is supported in bigradings $(d,s)$ such that $d-2s = i$. 
 \end{defn}
 Observe that this definition does not depend on the choice of the decomposition.
 
 It is a well-known fact (see \cite[Proposition 6.1.4]{ozsvath2015grid}) that the grid homology of a knot has rank 1, given by a $0$-tower.
 This observation motivates the following definition.
 \begin{defn}\cite{ozsvath2015grid}
 	For any knot $K$, $\tau(K)$ is $-1$ times the maximal Alexander grading of a homogeneous non-torsion element in $GH^{-}(K)\setminus\{0\}$.  
 \end{defn}
 As the authors observe in \cite[Remark 6.1.6]{ozsvath2015grid}, the sign is chosen to make $\tau$ agree with its original definition in \cite{KFfour-ball}.
 Our first step in extracting an analog of $\tau$ in the equivariant setting is the following lemma.
 
 \begin{lemma}\label{lemma:2torri}
 	Let $K$ be a \si knot. Then $\hc(K)$ has rank 2, given by a $0$-tower and a $1$-tower.
 	
 	\begin{proof}
 		Consider the long exact sequence in homology from Remark~\ref{rmk:long}
 		By exactness, it is straightforward that $rk\hc(K)\leq 2$; let us explain in more detail the equality and the conditions on the bigradings.
 		
 		The first step is to prove that $Im((Id+\rho)_*)\subseteq Tors GH^-(K)$. The image of a torsion element is necessarily a torsion element; therefore, let $x\in GH^-(K)_{d,s}$ be a non-torsion element and suppose that $\rho_*x$ is a torsion element.
 		Then let $r\in \F[U]$ be a non-zero element such that $r\rho_*x=0$. Since $\rho_*$ is an isomorphism, therefore $rx=0$, hence $\rho_*x$ is a non-torsion element. It follows that $x,\rho_*x\in GH^-(K)_{d,s}$ are non-torsion elements with the same bigrading. Since $rkGH^-(K)=1$, we can consider a decomposition $GH^-(K) = H \oplus Tors$, where $H\cong\F\left[U\right]$. Call $h_{d,s}=H\cap GH^-(K)_{d,s}$. We obtain that there exist $t_1,t_2\in Tors_{d,s}$ such that $x=h_{d,s}+t_1$ and $\rho_*x=h_{d,s}+t_2$, hence $x+\rho_*x\in Tors$.
 		
 		We now prove that $GH^-(K)$ contains a $0$-tower. Let $x\in GH^-(K)_{d,s}$ be a non torsion element, hence $d-2s=0$. If there exist an element $r\F[U]\setminus\{0\}$ such that $ri_*x=0$ by exactness $rx\in Im((Id+\rho)_*)$, and that is absurd since $rx$ is a non torsion element.
 		
 		Similarly, assume without loss of generality that $K$ is in Legendrian position and consider the grid invariant $\lambda^-\in GH^-(K)_{d,s}$. Since it is a non-torsion element \cite[Proposition 6.4.8]{ozsvath2015grid}, it follows that $d-2s=0$. Recall that $\lambda^-=[\xx^-]$ and observe that, since $\G$ is symmetric, the same goes for $x^-$. Hence, $\rho_*\lambda = [\rho \xx^-] = [\xx^-] = \lambda^-$. We have that $(Id+\rho)_*\lambda = 0$, so by exactness there exist an element $y\in\hc(K)_{d+1,s}$ such that $p_*y=\lambda^-$, therefore $y$ must be a non torsion element. By the condition on $d$ and $s$, we obtain that $y$ lies in a $1$-tower.       
 	\end{proof}   
 \end{lemma} 
 By Lemma~\ref{lemma:2torri}, the following definition makes sense.
 
 \begin{defn}
 	Let $i\in\{0,1\}$. For any \si knot $K$, $\tau_i(K)$ is $-1$ times the maximal Alexander grading of a homogeneous non-torsion element lying in the $i$-tower of $\hc(K)\setminus\{0\}$. 
 \end{defn} 
 
 \begin{remark}\label{rmk:tauineq}
 	Observe that, for any \sil knot $K$: $\tau_0(K)\leq\tau(K)\leq\tau_1(K)$.
 	This follows from the proof of Lemma~\ref{lemma:2torri}. In fact, let $x\in GH^-(K)_{-2\tau,-\tau}$, we know that $i_*x$ is a non torsion element, hence $-\tau_0\geq-\tau$. The second inequality is analogous.
 \end{remark}
 
 \subsubsection{Grid invariants}
 Given a knot $K$ and a grid $\G$ representing $K$, it is possible to define the \emph{canonical grid states} $\xx^+(\G)$,$\xx^-(\G) \in \St$ (\cite[Definition 6.4.1]{ozsvath2015grid}). 
 As proved in \cite[Lemma 6.4.2]{ozsvath2015grid}, the canonical grid states are cycles in $\gc$, hence they induce well-defined classes in $GH^-(K)$.
 
 When we see $\G$ as a grid representing a Legendrian knot, the homology classes of the canonical states $\lambda^+(K),\lambda^-(K) \in GH^-(K)$ do not depend on the grid $\G$ up to Legendrian grid moves. Hence, $\lambda^+(K)$ and $\lambda^-(K)$ are two relevant invariants of Legendrian isotopy called \emph{Legendrian grid invariants} (\cite[Definition 12.3.1]{ozsvath2015grid}, originally in \cite{legendriangridinvaria}). 
 
 One of the most interesting properties of Legendrian grid invariants is given by the following theorem.
 \begin{thm}\label{thm:u_mult}{\cite[Proposition 12.3.4]{ozsvath2015grid}}
 	Let $\G$, $\G^+$ and $\G^-$ be grid diagrams whose associated \si Legendrian knots are $K$ and its stabilization $K^+$ and $K^-$ respectively. Then, there are bigraded isomorphism:
 	\[ \phi^-:\gh \rightarrow GH^-(\G^-) \qquad\qquad \phi^+:\gh \rightarrow GH^-(\G^+) \]
 	satisfying:
 	\[ \phi^-(\lambda^+(\G))=\lambda^+(\G^-), \qquad\qquad\qquad U \cdot \phi^+(\lambda^+(\G))=\lambda^+(\G^+), \]\[
 	U \cdot\phi^-(\lambda^-(\G))=\lambda^-(\G^-), \qquad\qquad\qquad  \phi^+(\lambda^-(\G))=\lambda^-(\G^+).\]
 \end{thm}
 Theorem~\ref{thm:u_mult} provides an obstruction to the stabilization of a grid.
 The objective is to define invariants of equivariant Legendrian isotopy that exhibit an equivariant analogue of this obstruction.
 For clarity, the following definition is introduced.
 \begin{defn}
 	Let $\G$ be a symmetric grid. An \emph{equivariant stabilization} on $\G$ is one of the following:
 	\begin{itemize}
 		\item a NE or SW-type stabilization along the axis;
 		\item a pair of NE or SW-type stabilizations, symmetric with respect to the axis. 
 	\end{itemize}
 	After applying an equivariant stabilization to $\G$, we call \emph{equivariantly stabilized} the resulting symmetric grid.
 \end{defn}
 The focus is on constructing invariants that are sensitive to equivariant stabilizations. 
 Let $\G$ be a symmetric grid. Since the canonical grid states are cycles, we can consider two well-defined classes:
 \[ [(0,\xx^+)],[(0,\xx^-)]\in \hc(\G). \]
 Consider the long exact sequence in Remark~\ref{rmk:long}, we have that:
 \[ i_*(\lambda^\pm(\G))=[(0,\xx^\pm)]\in\hc(\G). \]
 Let $F:\gc\rightarrow\gcc$ be the map induced by a strongly invertible Legendrian grid move. We obtain the commutative square:
 \[  \begin{tikzcd}
 	\gc \arrow{r}{i} \arrow[swap]{d}{F} & 
 	\cono \arrow{d}{\text{Cone}\,F} \\ GC^-(\G') \arrow{r}{i'} & 
 	\text{Cone}_{\G'}(Id+\rho'),
 \end{tikzcd}. \]
 Since $F(\lambda^\pm(\G))=\lambda^\pm(\G)$, we get that the two classes $[(0,\xx^+)],[(0,\xx^-)]\in\hc(\G)$ are invariants of the underlying strongly invertible Legendrian knot.
 
 \begin{defn}
 	Let $K$ be a \sil knot. Referring to the notations in Remark~\ref{rmk:long}, we call \emph{\sil grid invariants} the classes $i_*(\lambda^+(K))$ and $i_*(\lambda^-(K))$ in $\hc(K)$.
 \end{defn} 
 \begin{remark}\label{rmk:brutte_notizie}
 	By $U$-equivariance of the map $i_*: \gh \rightarrow \hc(\G)$, the \sil grid invariants are divisible by $U$ any time the grid $\G$ is stabilized.
 	It follows that $i_*(\lambda^+(K))$ and $i_*(\lambda^-(K))$ are not sensitive to the difference between a grid that has been equivariantly stabilized and one non-equivariantly stabilized. 
 \end{remark}
 
 The following construction is inspired by \cite{sano2024involutive}.
 Given a symmetric grid $\G$, consider the two elements:
 \[ (\xx^\pm,0)\in\cono. \]
 Since the two canonical grid states $\xx^\pm$ are symmetric with respect to the SW to NE axis, $(Id+\rho)(\xx^\pm)=0$. This, together with the fact that the canonical grid states are cycles, implies that $(0,\xx^+)$ and $(0,\xx^-)$ are also cycles, and we get two well-defined classes $[(0,\xx^\pm)]\in\hc(\G)$.
 
 However, consider the map $\text{Cone}\,F:\cono \rightarrow \text{Cone}_{\G'}(Id+\rho')$ induced by a \sil grid move. We have:
 \[ \text{Cone}\,F((\xx^\pm(\G),0)) = (\xx^\pm(\G'),\phi_F(\xx^\pm)). \]
 We do know that $\phi_F(\xx^\pm)\in\gcc$ is a cycle, but not if it is a border. This implies that the classes $[(\xx^\pm,0)]$ could not be preserved by \sil grid moves. A natural approach is to extract information from these classes.
 \begin{defn}\label{def:d}
 	Let $\G$ be a symmetric grid. We call \emph{\sil correction terms} the integers:
 	\[d_e^+(\G) = \max  \left\{ n\in \Z_{\geq0} \;|\; \exists\,y\in\hc(\G)\;\text{s.t.}\;[(\xx^+,0)] + U^ny \in \text{Tors}(\hc(\G)) \right\}\]
 	and
 	\[d_e^-(\G) = \max  \left\{ n\in \Z_{\geq0} \;|\; \exists\,y\in\hc(\G)\;\text{s.t.}\;[(\xx^-,0)] + U^ny \in \text{Tors}(\hc(\G)) \right\}\]
 \end{defn}
 \begin{remark}\label{rmk:d}
 	Let $\G$ be a symmetric grid. For a given homogeneous element $y\in\hc(\G)$, call $\text{A}(y)$ the Alexander grading of $y$. 
 	Note that both $[(\xx^+,0)]$ and $[(\xx^-,0)]$ lie in the $1$-tower of $\hc(\G)$. It follows that:
 	\[ d^\pm_e(\G) = \tau_1(\G) - \text{A}([(\xx^\pm,0)]). \]
 \end{remark}
 The next two propositions demonstrate that the \sil correction terms are invariants of \sil knots and are sensitive to equivariant stabilizations.
 \begin{prop}\label{prop:d_inv}
 	Let $\G,\G'$ be two symmetric grids and  let $F:\gc\rightarrow\gcc$ be the quasi-isomorphism induced by a \sil grid move.
 	Then $d_e^+(\G)=d_e^+(\G')$ and $d_e^-(\G) = d_e^-(\G')$.
 	\begin{proof}
 		Start by noting that $F$ is bigraded, namely it is $(0,0)$-homogeneous. This implies that $ (\xx^\pm(\G'),0) $ and $ F((\xx^\pm(\G'),0))=(\xx^\pm(\G'),\phi_F(\xx^\pm))$ are two non-torsion elements lying in the same bidegree. 
 		
 		Lemma~\ref{lemma:2torri} tells us that $\hc(\G)$ contains a $0$-tower and a $1$-tower. It follows that, for a fixed degree $(d,s)$, $\hc(\G)$ contains only one non-torsion element, up to torsion elements. So, up to torsion elements, $\text{Cone}\,F$ sends $(\xx^\pm(\G),0)$ to $(\xx^\pm(\G'),0)$. The claim follows from Remark~\ref{rmk:d} because $\tau_1(\G)=\tau_1(\G')$.  
 	\end{proof}
 \end{prop}
 
 \begin{prop}\label{prop:d_U_molt}
 	Let $\G$, $\G^+$ and $\G^-$ be grid diagrams whose associated \si Legendrian knots are $K$ and its equivariant stabilization $K^+$ and $K^-$. Assume that both stabilizations are NE- or SW-type along the axis. Then:
 	\[ d_e^+(\G^-)=d_e^+(\G) \quad\quad d_e^-(\G^-)=d_e^-(\G)+1, \]
 	and
 	\[ d_e^+(\G^+)=d_e^+(\G)+1 \quad\quad d_e^-(\G^+)=d_e^-(\G). \]      
 	\begin{proof}
 		The procedure is very close to the proof of the previous proposition.
 		As in \cite[Theorem 12.3.4]{ozsvath2015grid}, we provide the proof in the case of an X:SW negative stabilization; the proofs for the other cases follow the same steps. 
 		Let $F:GC^-(\G^-) \rightarrow \gc$ be the bigraded destabilization map. Note that, in this setting (see \cite[Lemma 6.4.6 Case (S-2)]{ozsvath2015grid}):
 		\[ \text{Cone}\,F((\xx^+(\G^-),0)) = (\xx^+(\G),\phi_F(\xx^+(\G^-))),
 		\]and\[\text{Cone}\,F((\xx^-(\G^-),0)) = (V_i\cdot\xx^+(\G),\phi_F(\xx^+(\G^-))),\]
 		where the index $i$ depends on the index of the $X$-marking lying SE to the point at which we destabilize.
 		
 		Hence, the difference in homology between the class of $\text{Cone}\,F((\xx^+(\G^-),0))$ and the class of $(\xx^+(\G),0)$ is given by $[(0,\phi_F(\xx^+(\G^-)))]$, where $\phi_F(\xx^+(\G^-))$ is a torsion element. We conclude as in Proposition~\ref{prop:d_inv} that the correction terms coincide.
 		The difference in homology between the class of $\text{Cone}\,F((\xx^-(\G^-),0))$ and the class of \[\mathbf{V}_i\cdot(\xx^-(\G),0) = (V_i\cdot\xx^-(\G),\rho\circ h_{i\rho(i)}(\xx^-(\G))) \] is given by $[(0,\phi_F(\xx^+(\G^-))+\rho\circ h_{i\rho(i)}(\xx^-(\G^-)))]$. Also in this case, the difference is a torsion element. In fact, since we have a representative where the first coordinate is zero, $[(0,\phi_F(\xx^+(\G^-))+\rho\circ h_{i\rho(i)}(\xx^-(\G^-)))]$ could only be a non-torsion element lying in the $0$-tower, but the bidegree represents an obstruction: the difference between the Maslov grading an twice the Alexander grading is $1$.
 		We conclude, as for the previous cases, that $[(\xx^-(\G^-),0)]$ and $U\cdot[(\xx^-(\G),0)]$ are non-torsion elements lying in the respective $0$-towers and with the same bidegree. Again, Remark~\ref{rmk:d} and the equality $\tau_1(\G)=\tau_1(\G')$ conclude the proof. 
 	\end{proof}
 \end{prop}
 In the case of a non-equivariant stabilization we do not obtain an induced $\text{Cone}\,F$ map. This precludes determination of the behavior of the \sil correction terms. 
 
 \section{Cobordisms of strongly invertible knots}\label{sec:cobbound}
 In analogy with \cite[Chapter 8]{ozsvath2015grid}, in this section, we recover the well-known genus-bounding property of $\tau$.
 At first, we define the cobordisms of our interest, following \cite{sano2024involutive}. 
 \begin{defn}
 	Let $L$ and $L'$ be two \sil links that share the same involution $\rho$ on $S^3$. A \emph{simple equivariant cobordism} between $L$ and $L'$ is an oriented cobordism $S\subseteq S^3\times[0,1]$ such that $(\rho\times id)S = S$. We call $S$ a \emph{simple isotopy-equivariant cobordism} if $(\rho\times id)S$ is isotopic to $S$. 
 \end{defn}
 In the literature, simple equivariant cobordisms are also called \emph{standard equivariant cobordisms}, see for example \cite{manolescu2025rasmussen}.
 \begin{remark}
 	Fix the following notation. Given a \sil knot $K$ we call $\widetilde{g}_3$ the equivariant genus of $K$, $g_4(K)$ the slice genus of $K$, $\widetilde{g}_4(K)$ the equivariant slice genus of $K$, $\widetilde{sig}_4(K)$ the simple isotopy-equivariant genus and $\widetilde{sg}_4(K)$ the simple equivariant genus of $K$.  
 	There are some obvious inequalities:
 	
 	\centering
 	\begin{tikzpicture}
 		\node (A) at (0.2, 0) {$g_4$};
 		\node (B) at (1.5, 0.75) {$\widetilde{g}_4$};
 		\node (D) at (1.5, -0.75) {$\widetilde{sig}_4$};
 		\node (C) at (3, 0) {$\widetilde{sg}_4$};
 		\node (E) at (4.5,0) {$\widetilde{g}_3$};
 		\node[rotate=30, scale=0.85] at (0.75, 0.375) {${}\leq$}; 
 		\node[rotate=-30, scale=0.85] at (2.25, 0.375) {${}\leq$}; 
 		\node[rotate=-30, scale=0.85] at (0.75, -0.375) {${}\leq$}; 
 		\node[rotate=30, scale=0.85] at (2.25, -0.375) {${}\leq$}; 
 		\node[ scale=0.85] at (3.75, 0) {${}\leq$}; 
 	\end{tikzpicture}
 \end{remark}
 Throughout this section, when referring to a cobordism between two \sil links $L$ and $L'$, it will be taken for granted that $L$ and $L'$ share the same involution $\rho$.
 
 In Section~\ref{sec:proof_genus_bound}, we prove the following theorem.
 
 \begin{thm}\label{thm:genus_bound}
 	Let $K$ and $K'$ be two \si knots connected by a genus $g$ simple equivariant cobordism $S\subseteq [0,1]\times S^3$. Then the equivariant tau invariants $\tau_0$ and $\tau_1$ both provide a lower bound to the genus of the cobordism:
 	\[ \max \left\{|\tau_0(K)-\tau_0(K')|\,,\, |\tau_1(K)-\tau_1(K')| \right\} \leq g. \]
 \end{thm}
 \begin{cor}\label{cor:genus_bound}
 	For a \si knot $(K,\rho)$ and for $i=0,1$, $|\tau_i(K)| \leq \widetilde{sg}_4(K)$. In particular, if $K$ is simple equivariantly slice, then $\tau_i(K)=0$.
 	\begin{proof}
 		As we already assumed, we can extend $\rho$, which acts on $S^3=\partial B^4$, to an involution on the whole four-ball $\rho:B^4\rightarrow B^4$.
 		Let $F$ be a $\rho$-invariant surface in $B^4$, such that $\partial F = K$ and such to minimize the simple equivariant genus of $K$.
 		If we remove a small four-ball from $B^4$, centered at a point in $F$ and symmetric with respect to $\rho$, we obtain a simple equivariant cobordism from $K$ to the unknot $\mathcal{U}$. The claim follows from Theorem~\ref{thm:genus_bound}, because $\tau_0(\mathcal{U})=\tau_1(\mathcal{U})=0$.
 	\end{proof}
 \end{cor}

 \subsection{Normal form of simple equivariant cobordism}
 The focus is on simple equivariant cobordism. A proof of the existence of a normal form for cobordisms in this context is provided. For convenience, the ambient manifold of the cobordism is modified.
 \begin{lemma}\label{lemma:morse_self}
 	Let $W\subseteq S^3\times[-1,3]$ be a simple isotopy-equivariant cobordism between two \sil links $L$ and $L'$. Up to equivariant isotopy, we can assume that the restriction of the projection on the $[-1,3]$ factor: \[ f_W=\pi_{[-1,3]}|_{W}:W \rightarrow [-1,3] \] is a self-indexing Morse function.
 	\begin{proof}
 		We can assume that all the critical points map into the interval $[0,2]$.
 		We have two kinds of critical points: equivariant pairs of critical points and critical points lying on the fixed annuli of $\rho\times id$. 
 		
 		The case of index $0$ critical points is simple. For an equivariant pair of index $0$ critical points, consider an equivariant pair of arcs connecting them to the $0$ level, such that the arcs intersect no other critical point. Such a pair of arcs exists by general position arguments. Now, a neighbourhood of this pair of arcs supports an isotopy of $W$ bringing the pair of critical points to the $0$ level. For an index $0$ critical point, we consider an arc lying on the fixed annulus and proceed with the same steps. This procedure applies symmetrically to the index $2$ critical points. 
 		
 		For the index $1$ points, we must be more careful. For the equivariant pairs, we must perform the steps in \cite[Lemma B.5.3]{ozsvath2015grid} equivariantly for the set of critical points on both "sides" of the fixed annulus. We must slightly modify the same procedure to move the set of points lying on the fixed annulus.
 		
 		Consider $t_1<t_2$ two levels with an index $1$ critical point each, such that $f_W^{-1}((t_1,t_2))$ contains no critical points. We can replace the two critical points with two embedded \emph{saddle bands} (see Figure~\ref{img:bands}) via an equivariant isotopy on $W$. The whole band lies on one only level, hence after this modification, the projection is no longer Morse. Note how, either the core or the co-core of the band lies on the fixed $S^1$ at level $t_i$, $i=1,2$.
 		
 		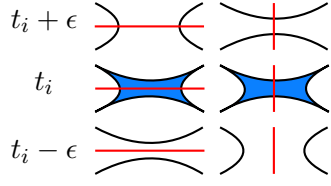
\begin{figure}[t]
 			\centering
 			\begin{tikzpicture}
 				\draw (0,0) node[above right]{\includegraphics[width=0.25\linewidth]{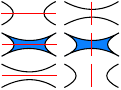}};
 				\draw (-0.5,2.2) node (a) [scale=1] {$t_i+\epsilon$};
 				\draw (-0.5,1.38) node (a) [scale=1] {$t_i$};
 				\draw (-0.5,0.47) node (a) [scale=1] {$t_i-\epsilon$};
 			\end{tikzpicture}
 			\caption{The two possible ways a band may appear. The red axes are portions of the fixed $S^1$, that is intersections of the fixed annulus with a height level.}
 			\label{img:bands}
 		\end{figure}
 				\begin{figure}[t]
 			\centering
 			\includegraphics[width=0.55\textwidth]{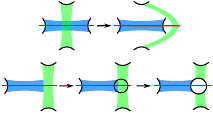}
 			\caption{Arranging bands in the non-disjoint case.}
 			\label{img:arr_bands}  
 		\end{figure}
 		
 		Thanks to the assumption that $f_W^{-1}((t_1,t_2))$ contains no critical points, the cobordism between the levels $t_1$ and $t_2$ is now trivial, that is we can assume it is the product cobordism. We move the band at level $t_1$ to the level $t_2$. If the two bands at level $t_2$ are disjoint we transform them back to index $1$ critical points, that now lie on the same level. Otherwise, the intersection is either on a whole square or at an endpoint of the core of one band. In the first case, we move one of the bands as in the top lane of Figure~\ref{img:arr_bands}. In the second case, we see the union of the two bands as the union of three bands and a circle (see the bottom lane of Figure~\ref{img:arr_bands}).
 		The disk's centre will be an index $0$ critical point, which we move to height $0$ as in the previous steps, ending up with just its boundary $S^1$ at level $t_2$. Transforming the three bands back to critical points, we obtain a critical point on the fixed annulus and a pair of equivariant critical points. We obtained one additional index-$1$ critical point, which compensates for the new index-$0$ critical point. 
 		In both cases, the index $1$ critical points all lie on the same level, iterating this procedure a finite number of times, we conclude.
 	\end{proof}
 \end{lemma}
 The existence of a normal form can now be established.
 Let $K$ be a knot. Call $\mathcal{U}_b(K)$ the disjoint union of $K$ with $b\geq 0$ unknotted components. If $(K,\rho)$ is a \si knot, it is assumed that $\rho$ sends $\mathcal{U}_b(K)$ to itself, switching the orientation on every component. This implies that the unknotted components are either strongly invertible components, which meet the fixed-point axis of $\rho$ at two points, or pairs of components exchanged by $\rho$. 
 \begin{thm}\label{thm:normal_form}
 	Let $K_1$ and $K_2$ be two \sil knots connected by a simple equivariant genus $g$ cobordism $W\subseteq S^3\times [0,2]$. Then there exist \sil knots $K_1'$ and $K_2'$, integers $b,d \geq 0$, and links $\mathcal{U}_b(K_1)$ and $\mathcal{U}_d(K_1)$ such that:
 	\begin{enumerate}
 		\item $\mathcal{U}_b(K_1)$ can be obtained from $K_1'$ via $b$ simultaneous oriented saddle moves. 
 		\item $K_1'$ and $K_2'$ can be connected by a sequence of $2g$ oriented saddle moves.
 		\item $\mathcal{U}_d(K_2)$ can be obtained from $K_2'$ via $d$ simultaneous oriented saddle moves.
 	\end{enumerate}
 	\begin{proof}
 		The proof is the same as the one of \cite[Proposition B.5.1]{ozsvath2015grid}. Again, to simplify the notation, suppose that $W\subseteq S^3\times[-1,3]$. By Lemma~\ref{lemma:morse_self}, we can assume that $f_W$ is a self-indexing Morse function. Then, for some integers $b$ and $d$, we can identify $f_W^{-1}(0.1)$ with $\mathcal{U}_b(K_1)$ and $f_W^{-1}(1.9)$ with $\mathcal{U}_b(K_2)$. Note that $b$ and $d$ are respectively the number of index $0$ and index $2$ critical points of $f_W$.
 		
 		Now we can choose $b$ saddle points such that the corresponding handles turn $\mathcal{U}_b(K_1)$ into a knot, that we call $K_1'$. We move these points to level $0.2$. Similarly, we can choose $d$ saddle points such that the corresponding handles turn $\mathcal{U}_d(K_2)$ into a knot, which we call $K_2'$. We move these points to level $1.8$.
 		We can rename the resulting Morse function $g_W$, so $g_W^{-1}(0.5)$ and $g_W^{-1}(1.5)$ will provide respectively $K_1'$ and $K_2'$.
 		A straightforward computation on the Euler characteristic shows that the remaining saddle points between $K_1'$ and $K_2'$ must be $2g$. 
 	\end{proof}
 \end{thm}
 
 \subsubsection{Collapsed grid homology for links}
  In \cite[Section 8.2]{ozsvath2015grid}, the authors present a version of grid homology for links called \emph{collapsed grid homology}. We briefly recall this definition. 
  Fixed a grid $\G$, the construction of the grid complex $\gc$ does not depend on the number $l$ of components of the link represented by $\G$. Lemma~\ref{eU} only grants the chain homotopy of the multiplication by variables in the case of multiplication by two variables $V_i$ and $V_j$ associated to $O$-markings $O_i$ and $O_j$ belonging to the same component. In order to obtain an $\F[U]$-module in homology, one chooses a subset $\{ O_{i_1},\dots,O_{i_l} \} \subseteq \Oo$ of $O$-markings, one per component. The collapsed grid complex is the quotient:
  \[ c\gc = \frac{\gc}{V_{i_1} = \dots = V_{i_l}}. \]   
  
  One checks that nothing changes in the case of the cone complex, once we prove that the quotient is equivariant.
  Let $\G$ be a symmetric grid representing an $l$-component link $L$.
  Fix an $O$-marking on $\G$ for every component of $L$ in the following way:
  \begin{itemize}
  	\item If the $j$-th component intersects the axis, we can assume up to stabilization to have at least an $O$-marking along the SW to NE axis. Call $O_{i_j}$ such $O$-marking. 
  	\item If the $j$-th component does not intersect the axis, fix any $O$-marking on it as $O_{i_j}$ and its symmetric $\rho(O_{i_j})$ as $O_{i_{\rho(j)}}$.
  \end{itemize}
  If at least one of the fixed $O$-markings lies on the axis, assume that it is $O_{i_1}$. In such case, we obtain $c\gc$ quotienting $\gc$ by the submodule $$(V_{i_1}-V_{i_2},\dots,V_{i_1}-V_{i_l})\gc,$$ which is $\rho$-equivariant.
  Otherwise, we can write the set of fixed $O$-markings as:
  \[ \{ O_1,\dots,O_{l/2},O_{\rho(1)},\dots,O_{\rho(l/2)}. \} \]
  In this case, as a $\rho$-equivariant submodule, consider:
  \[ (V_{1}-V_{\rho(1)},V_{1}-V_{2},\dots,V_{1}-V_{l/2}, V_{\rho(1)}-V_{\rho(2)},\dots,V_{\rho(1)}-V_{\rho(l/2)})\gc. \] 
  In both cases, we then have an induced morphism:
  \[ \rho: c\gc \rightarrow c\gc, \]
  And, consequently, we can define $c\cono$ as:
  \[ c\cono = \text{Cone}\left(  Id+\rho : c\gc \rightarrow c\gc \right). \]
  As in the connected case, the homology $c\hc(\G)$ of the collapsed cone complex only depends on the underlying link $L$ and not on the fixed grid $\G$. 
  The only distinction is that collapsed grid homology for links is sensitive to changes in the orientation of the components. We will assume to work with oriented links. 
  
  Proposition~\ref{prop:exact_seq} and Remark~\ref{rmk:long} still hold.
  In light of \cite[Proposition 8.3.2]{ozsvath2015grid}, it follows that, by exactness, the rank of $c\cono/\text{Tors}$ is $2^l$ and we can define the following set of invariants:
  \begin{defn}
  	Let $\G$ be a symmetric grid representing an $l$-component link $L$.
  	Choose a generating set of $2^l$ elements for $c\hc(L)/\text{Tors}$, with the property that each element is homogeneous with respect to both the Alexander and the Maslov grading.
  	The unordered set obtained by multiplying the Alexander gradings of these generators by $-1$ is the \emph{equivariant $\tau$-set} of $L$.
  \end{defn}
  To reflect the two invariants $\tau_0$ and $\tau_1$ in the case of a knot, we can refine the equivariant $\tau$-set, splitting it into two sets. 
  Since localizing is exact, the cone homology long exact sequence from Remark~\ref{rmk:long} remains exact if we invert $U$. Moreover, since the connecting morphism $(Id+\rho)_*$ has only image in the torsion part of $c\gh$, we obtain a short exact sequence:
  \[ 0 \rightarrow c\gh \otimes \F[U^{\pm1}] \xrightarrow{i} c\hc(L) \otimes \F[U^{\pm1}] \rightarrow c\gh \otimes \F[U^{\pm1}] \rightarrow 0. \]
  It follows that, out of the $2^l$ towers in $c\hc(L)\otimes \F[U^{\pm1}]$, $2^{l-1}$ intersect non trivially $i(c\gh\otimes \F[U^{\pm1}])$.
  We call any such tower a \emph{$0$-tower}.
  Clearly, for any tower in $c\hc(L)\otimes \F[U^{\pm1}]$ we get a corresponding generator in $c\hc(L)/Tors$, and hence a corresponding value in the equivariant $\tau$-set. 
  The $2^{l-1}$ values in the equivariant $\tau$-set corresponding to the $2^{l-1}$ $0$-towers form the $\tau_0$-set of $L$, the remaining $2^{l-1}$ are the $\tau_1$-set of $L$.
  
  For both sets, we are primarily interested in two values. We can characterize them as follows: $\tau_0^{min}$ (resp. $\tau_1^{min}$) is $(-1)$ times the maximal Alexander grading of any homogeneous, non-torsion element in $c\hc(L)$ such that, when tensorizing by $\F[U^{\pm1}]$, it belongs to a $0$-tower (resp. it does not belong to a $0$-tower); $\tau_0^{max}$ (resp. $\tau_1^{max}$) works in the same way, but we look for the minimal Alexander grading of a homogeneous element that is not contained in $U\cdot (c\hc(L)/Tors)$.
 
\subsection{The genus bound}\label{sec:genusbound}
  This section follows the constructions presented in \cite[Chapter 8]{ozsvath2015grid}. 
  As previously established, cobordisms can be described by decomposing them into simpler pieces: isotopies, birth/death of an unknotted unlinked component, and saddle moves.
  For a simple equivariant cobordism, isotopies must be equivariant with respect to the action of $\rho$. The following definitions indicates that this requirement also applies to birth, deaths, and saddle moves. 
  \begin{defn}
  	We call \emph{equivariant saddle move} one of the following:
  	\begin{itemize}
  		\item A saddle move performed along the SW to NE axis.
  		\item A pair of saddle moves that are symmetric to each other with respect to the SW to NE axis.
  	\end{itemize}
  \end{defn}
  \begin{defn}
  	We call \emph{equivariant birth} (respectively \emph{death}) either the addition (resp. removal) of an unknotted unlinked component setwise fixed by $\rho$ intersecting the SW to NE axis in two points, or the addition (resp. removal) of two unknotted unlinked components interchanged by $\rho$.
  \end{defn}
  \begin{remark}\label{rmk:eq_saddle}
  	Elementary equivariant cobordisms connect two links that are symmetric with respect to the fixed-point circle, which we will call \emph{symmetric links} or \emph{involutive links}.
  	Such links are not necessarily \si, since equivariant saddle moves may alter both the number of link components and the number of intersections with the symmetry axis so that the \si condition may no longer be satisfied.   	 
  \end{remark}
  The goal is to demonstrate that both $\tau_0$ and $\tau_1$ provide a lower bound for the genus of a cobordism, analogous to the classical $\tau$ invariant. This will be achieved by utilizing the maps and constructions from \cite[Sections 8.3 and 8.4]{ozsvath2015grid}, decomposing the cobordism into elementary pieces.
  Let $L$ and $L'$ be two involutive links. If $L$ and $L'$ are connected by an equivariant isotopy, then, in analogy with Theorem~\ref{thm:invariance}, we can find a bigraded isomorphism of $\F[U]$-modules between $c\hc(L)$ and $c\hc(L')$, induced by \si grid move.  
  The remaining cases are addressed in the following paragraphs.
  
  \subsubsection{Saddles}
  The analysis begins with the maps associated with saddle moves.
  The first case considered is that of a pair of saddle moves exchanged by $\rho$.
  Let $L'$ be an involutive link obtained from $L$ by a pair of $\rho$-equivariant saddle moves, and let $\G$ and $\G'$ be symmetric grids respectively representing $L$ and $L'$.  
  All the possible configurations occur: the two saddles may form a pair of split moves, a pair of merge moves, or consist of one split move and one merge move.
  We begin with some auxiliary lemma. Let $W= \F_{(0,0)}\oplus\F_{(-1,-1)}$. 
  \begin{lemma}\label{lem:hom_quoz_tens_W}
  	If $O_1$ and $O_{\rho(1)}$ are on the same component of the link specified by $\G$, then there is an isomorphism of bigraded $\F[U]$-modules
  	\[ H\left( \frac{c\cono}{(\mathbf{V}_1-\mathbf{V}_{\rho(1)})} \right) = c\hc(\G) \otimes W. \] 
  	If $O_1$ and $O_2$ are on the same component of the link specified by $\G$, distinct from the one shared by $O_{\rho(1)}$ and $O_{\rho(2)}$, then there is an isomorphism of bigraded $\F[U]$-modules
  	\[ H\left( \frac{c\cono}{(\mathbf{V}_1-\mathbf{V}_2,\mathbf{V}_{\rho(1)}-\mathbf{V}_{\rho(2)})} \right) = c\hc(\G) \otimes W^{\otimes 2}. \] 
  	\begin{proof}
  		Focus on the first case. Multiplication by $\mathbf{V}_1$ and by $\mathbf{V}_{\rho(1)}$ are chain homotopic as endomorphisms of $c\cono$. We have that $\mathbf{V}_{1}-\mathbf{V}_{\rho(1)}$ is an injective morphism, and it is chain homotopic to the $0$ map. By classical results (see for instance \cite[Appendix A.3]{ozsvath2015grid}) these properties give quasi-isomorphisms:
  		\[ c\hc(\G) \otimes W \rightarrow \text{Cone}_\G\left(\mathbf{V}_1-\mathbf{V}_{\rho(1)} \right) \rightarrow \frac{c\cono(\G)}{(\mathbf{V}_1-\mathbf{V}_{\rho(1)})}. \]
  		The second case is obtained by iterating the procedure.
  	\end{proof}
  \end{lemma}
  Let $i=1,\dots,n$ be an index. The submodule $(V_i-V_{\rho(i)})\gc$ is $\rho$-equivariant, hence we obtain a well-defined map:
  \[ \rho:\frac{\gc}{V_i-V_{\rho(i)}} \rightarrow \frac{\gc}{V_i-V_{\rho(i)}}. \]
  As in the non-quotiented case, $\rho$ is an $\F$-linear map, and we can consider the vector space:
  \[ \text{Cone}_\G(\overline{Id+\rho}) := \text{Cone}_\G\left( Id + \rho:\frac{\gc}{V_i-V_{\rho(i)}} \rightarrow \frac{\gc}{V_i-V_{\rho(i)}} \right), \]
  which we endow with the multiplications $\{\mathbf{V}_k \;|\; k\in\{1,\dots,n\}\setminus\{i,\rho(i)\}\}$ that we already know.
  \begin{lemma}\label{lem:conoquoz_quozcono}
  		Let $\G$ be a size $n$ symmetric grid an assume that $\rho(1)=2$.
  		Then there is an isomorphism of $\F[V_3,\dots,V_n]$-modules:
  		\[ \frac{\cono}{\mathbf{V}_1-\mathbf{V}_2} \xrightarrow{f} \text{Cone}_\G(\overline{Id+\rho}).  \]
  		\begin{proof}
  			The projection $\gc \xrightarrow{\pi}\gc/(V_1-V_2)$ is $\rho$-equivariant, so we can induce a surjective projection:
  			\[ \cono \xrightarrow{\mathfrak{p} :=\begin{pmatrix}
  					\pi & 0 \\ 0 & \pi
  			\end{pmatrix}} \text{Cone}_\G(\overline{Id+\rho}). \]
  			As usual, this map is only $\F$-linear. To check that it is a module map is a quick computation. 
  			
  			Note that, by the hypotesis that $\rho(1)=2$, the lower-left coefficients of the multiplications by $\mathbf{V}_1$ and $\mathbf{V}_2$ agree. Hence, working in modulo two coefficients, for any $(x,y)\in\cono$ we get:
  			\[ (\mathbf{V}_1-\mathbf{V}_2)(z,w) = ((V_1-V_2)z,(V_1-V_2)w) \]
  			It is now straightforward to note that $(\mathbf{V}_1-\mathbf{V}_2)\cono$ is contained in $\text{Ker}\,\mathfrak{p}$. Assume that $(x,y)\in \text{Ker}\,\mathfrak{p}$, i.e. $(\pi(x),\pi(y))=(0,0)$. Then there exists $z,w\in\gc$ sucht that $x=(V_1-V_2)z$ and $y=(V_1-V_2)w$, and:
  			\begin{align*}
  				(x,y) &= ((V_1-V_2)z,(V_1-V_2)w) \\
  				&= ((V_1-V_2)z,\rho h_{12}z+\rho h_{12}z +(V_1-V_2)w) \\
  				&= (\mathbf{V}_1-\mathbf{V}_2)(z,w).
  			\end{align*}
  			This implies the thesis.   
  		\end{proof}
  	\end{lemma}
  	\begin{remark}\label{rmk:vero_conoquoz_quozcono}
  		Lemma~\ref{lem:conoquoz_quozcono} can be adapted to the case in which $V_{\rho(1)}= V_1$ and we quotient by the ideal $(V_1-V_2,V_1-V_{\rho(2)})$. It will be enough to apply Lemma~\ref{lem:conoquoz_quozcono} quotienting by $V_2-V_{\rho(2)}$ and then iterate quotienting by $V_1-V_2$. We can invoke the same result, since $V_1-V_2$ is equivariant under the quotient.
  	\end{remark}
  	The following is the final auxiliary lemma. In the previous section, the collapsed grid complex was defined. For a component crossing the axis, the choice of variables to quotient by is rigid. The next lemma provides a more flexible approach.
  \begin{lemma}\label{lem:quella_roba_su_quozientare_diversamente}
  	Let $L$ be an involutive link. Assume that there exists a component $K\subseteq L$ that meets the axis of symmetry. Consider a symmetric grid $\G$ representing $L$. We can assume that:
  	\[ c\gc = \frac{\gc}{(V_{i_1}-V_{i_2},\dots,V_{i_1}-V_{i_l})\gc},  \]
  	Where $O_{i_1}$ belongs to the component representing $K$.
  	Assume that $\rho(2)=3$ and consider the quotient module:
  	\[ c'\gc = \frac{\gc}{(V_{i_1}-V_{a},V_{i_1}-V_{\rho(a)},V_{i_1}-V_{i_4},\dots,V_{i_1}-V_{i_l})\gc},  \]
  	where $O_a$ and its symmetric $O_{\rho(a)}$ both belong to the component associated with $K$. Note that we preserved the $\rho$-equivariance.
  	Then $c\hc$ is isomorphic as $F[U]$-module to:
  	\[ c'\hc = c'H\cono = H\text{Cone}(Id+\rho: c'\gc \rightarrow c'\gc ). \]
  	In particular, $c\hc$ and $c'\hc$ share the $\tau_0$-set and the $\tau_1$-set.
  	\begin{proof}
  		Observe that:
  		\[ \frac{c\gc}{(V_{i_1}-V_a,V_{i_1}-V_{\rho(a)})} = \frac{c'\gc}{(V_{i_1}-V_a,V_{i_1}-V_{\rho(a)})} \]
  		as quotients of $\gc$.
  		By Remark~\ref{rmk:vero_conoquoz_quozcono} we know that
  		\[ \text{Cone}\left(Id+\rho: \frac{c\gc}{(V_{i_1}-V_a,V_{i_1}-V_{\rho(a)})} \rightarrow \frac{c\gc}{(V_{i_1}-V_a,V_{i_1}-V_{\rho(a)})} \right) \]
  		is isomorphic to
  		\[ \frac{c\cono}{\mathbf{V}_{i_1}-\mathbf{V}_{a},\mathbf{V}_{i_1}-\mathbf{V}_{\rho(a)}}. \]
  		By Lemma~\ref{lem:hom_quoz_tens_W}, there exists an $\F[U]$-module isomorphism:
  		\[ H\left( \frac{c\cono}{\mathbf{V}_{i_1}-\mathbf{V}_{a},\mathbf{V}_{i_1}-\mathbf{V}_{\rho(a)}} \right) = c\hc \otimes W^{\otimes 2}. \]
  		Doing the same with $c'\gc$ and remembering about the firs equivalence on this proof, it follows that there is an isomorphism of $\F[U]$-modules:
  		\[ c\hc \otimes W^{\otimes 2} = c'\hc \otimes W^{\otimes 2}.  \]
  		Since we work with free finitely generated complexes of modules over $\F[V_1,\dots,V_n]$, this implies the thesis.
  	\end{proof}
  \end{lemma}
  The next proposition addresses the cases of a pair of $\rho$-equivariant splits or merge saddles.
  
  \begin{prop}\label{prop:2saddle_buono}
  	If $L'$ is obtained from $L$ by a $\rho$-equivariant pair of split moves, then there are $\F[U]$-module maps:
    \begin{align*}
    	\sigma &: c\hc(L) \otimes W^{\otimes 2} \rightarrow c\hc(L') \\
    	\mu &:  c\hc(L') \rightarrow c\hc(L) \otimes W^{\otimes 2}
    \end{align*}
    with the following properties:
    \begin{itemize}
    	\item $\sigma$ is homogeneous of degree $(-2,0)$,
    	\item $\mu$ is homogeneous of degree $(-2,-2)$.
    	\item $\mu \circ \sigma$ and $\sigma \circ \mu$ both are the multiplication by $U^2$.
   \end{itemize}
  	\begin{proof}
  		Let $\G$ and $\G'$ be two symmetric grids representing $L$ and $L'$. Up to stabilizations, and up to planar isotopies, we can assume that the saddles are as specified by Figure~\ref{fig:sad_segm}. 
  			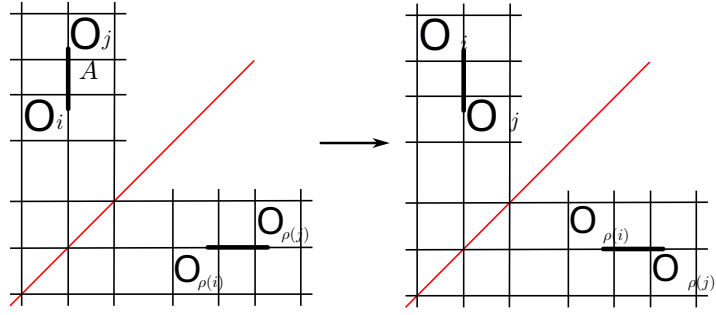
\begin{figure}
  			\centering
  			\begin{tikzpicture}
  				\draw (0,0) node[above right]{\includegraphics[width=0.73\linewidth]{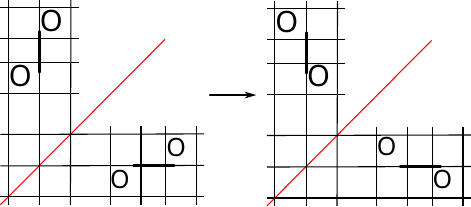}};
  				\draw (1.18,3.28) node[scale=0.9] (a) {$A$};
  				\draw (0.8,2.6) node[scale=0.75] (a) {$i$};
  				\draw (1.43,3.68) node[scale=0.75] (a) {$j$};
  				\draw (2.8,0.46) node[scale=0.55] (a) {$\rho(i)$};
  				\draw (3.92,1.1) node[scale=0.55] (a) {$\rho(j)$};
  				\draw (6.15,3.64) node[scale=0.75] (a) {$i$};
  				\draw (6.82,2.6) node[scale=0.75] (a) {$j$};
  				\draw (8.16,1.08) node[scale=0.55] (a) {$\rho(i)$};
  				\draw (9.29,0.46) node[scale=0.55] (a) {$\rho(j)$};
  			\end{tikzpicture}
  			\caption{The case of a pair of $\rho$-equivariant saddles. The thickened segments indicate $A$ and its symmetric $\rho(A)$.}
  			\label{fig:sad_segm}
  		\end{figure}
  		Identify the states of $\G$ and $\G'$. Define the two sets of states $\mathcal{A}$ and $\overline{\mathcal{A}}=\rho(\mathcal{A})$ as the states which have a component on the segment $A$ and $\rho(A)$ respectively. 
  		Observe that $O_i$ and $O_j$ are on different components of $\G'$ (and hence also $O_{\rho(i)}$ and $O_{\rho(j)}$). By Lemma~\ref{lem:quella_roba_su_quozientare_diversamente}, we can assume that, when defining $c\gcc$, we molded out by $V_i-V_j$ and by $V_{\rho(i)}-V_{\rho(j)}$. This implies that in both $c\gcc$ and in $c\gc/(V_i-V_j,V_{\rho(i)}-V_{\rho(j)})$ ve have that:
  		\begin{itemize}
  			\item $V_i$ and $V_j$ induce the same multiplication, which we call multiplication by $V$.
  			\item $V_{\rho(i)}$ and $V_{\rho(j)}$ induce the same multiplication, that we call multiplication by $\overline{V}$.
  		\end{itemize}
  		Now, \cite[Proposition 8.3.1]{ozsvath2015grid} provide the saddle maps $\sigma$ (for the split) and $\mu$ (for the merge) in the non-equivariant case:
  		 \[ \sigma: \frac{c\gc}{(V_i-V_j)} \rightarrow c\gcc \quad\text{ and }\quad \mu: c\gcc \rightarrow \frac{c\gc}{(V_i-V_j)}\]
  		 defined on states as:
  		 \begin{equation}\label{eq:sigma_mu} \sigma(\xx)=\begin{cases}
  		 	V\cdot\xx  &\xx\in \mathcal{A} \\
  		 	\xx &\xx\notin\mathcal{A}.
  		 \end{cases} \quad\text{ and }\quad
  		 \mu(\xx)=\begin{cases}
  		 	\xx  &\xx\in \mathcal{A}\\
  		 	V \cdot\xx &\xx\notin\mathcal{A}.
  		 \end{cases} 
  		 \end{equation}
  		We compose the two split maps $\sigma$ and the two merge maps $\mu$ from \cite[Proposition 8.3.1]{ozsvath2015grid} (one set of $\{\sigma,\mu\}$ maps for the saddle along $A$ and one for the saddle along $\rho(A)$) obtaining two maps defined on states as:
  		
  		 \[ \sigma: \frac{c\gc}{(V_i-V_j,V_{\rho(i)}-V_{\rho(j)})} \rightarrow c\gcc, \; \sigma(\xx)=\begin{cases}
  		 	V\cdot \overline{V} \cdot\xx  &\xx\in \mathcal{A}\cap\overline{\mathcal{A}}\\
  		 	V\cdot \xx &\xx\in\mathcal{A}\cap\overline{\mathcal{A}}^c\\
  		 	\overline{V}\cdot \xx &\xx\in\mathcal{A}^c\cap\overline{\mathcal{A}}\\
  		 	\xx &\xx\in\mathcal{A}^c\cap\overline{\mathcal{A}}^c.
  		 \end{cases} \]
  		 and:
  		 \[ \mu: c\gcc \rightarrow \frac{c\gc}{(V_i-V_j,V_{\rho(i)}-V_{\rho(j)})},
  		 \;
  		 \mu(\xx)=\begin{cases}
  		 	\xx  &\xx\in \mathcal{A}\cap\overline{\mathcal{A}}\\
  		 	\overline{V}\cdot \xx &\xx\in\mathcal{A}\cap\overline{\mathcal{A}}^c\\
  		 	V\cdot \xx &\xx\in\mathcal{A}^c\cap\overline{\mathcal{A}}\\
  		 	V \cdot \overline{V}\cdot \xx &\xx\in\mathcal{A}^c\cap\overline{\mathcal{A}}^c.
  		 \end{cases}
  		 \]
  		The conditions on the compositions are trivially satisfied.
  		From \cite[Proposition 8.3.1]{ozsvath2015grid} we know that $\sigma$ and $\mu$ are homogeneous chain maps with the correct bigradings. 
  		We must check the $\rho$-equivariance to induce maps on the cone complex. 
  		It is enough to prove it on states. On the $\rho$-invarivariant sets $\mathcal{A} \cap \overline{\mathcal{A}}$ and $\mathcal{A}^c \cap \overline{\mathcal{A}}^c$ we have that $\rho \circ \sigma = \sigma \circ \rho$. If, for example, $\xx \in \mathcal{A} \cap \overline{\mathcal{A}}^c$, we still achieve commutativity thanks to the action of $\rho$ on the multiplications. In fact: 
  		$$\rho \circ \sigma (\xx) = \rho (V\cdot \xx) = \overline{V} \cdot \xx$$ 
  		and
  		$$ \sigma \circ \rho (\xx) = \sigma(\rho(\xx)) = \overline{V} \cdot \xx. $$ 
  		The same computations work with $\mu$. It follows that we can induce maps $\sigma$ and $\mu$ between $c\ccono$ and
  		\[ \text{Cone}_\G\left(Id+\rho: \frac{c\gc}{(V_i-V_j,V_{\rho(i)}-V_{\rho(j)})} \rightarrow \frac{c\gc}{(V_i-V_j,V_{\rho(i)}-V_{\rho(j)})}\right). \]
  		By Lemma~\ref{lem:hom_quoz_tens_W} and Remark~\ref{rmk:vero_conoquoz_quozcono}, we obtain the thesis.
  	\end{proof}
  \end{prop}
  The case of a $\rho$-equivariant pair of saddles consisting of one merge move and one split move is analogous.
  \begin{prop}\label{prop:2saddle_cattivo}
  		If $L'$ is obtained from $L$ by a $\rho$-equivariant pair of saddle moves, made of one merge and one split, then there are $\F[U]$-module maps:
  	\begin{align*}
  		\varphi &: c\hc(L) \rightarrow c\hc(L') \\
  		\psi &:  c\hc(L') \rightarrow c\hc(L)
  	\end{align*}
  	with the following properties:
  	\begin{itemize}
  		\item $\varphi$ is homogeneous of degree $(-2,-1)$,
  		\item $\psi$ is homogeneous of degree $(-2,-1)$.
  		\item $\psi \circ \varphi$ and $\varphi \circ \psi$ both are the multiplication by $U^2$.
  	\end{itemize}
  	\begin{proof}
  		The setting is still the one in Figure~\ref{fig:sad_segm}. This time, $V_i$ and $V_j$ (and hence $V_{\rho(i)}$ and $V_{\rho(j)}$) belongs to different components of both $\G$ and $\G'$. Again, we assume to mold out by $V_i-V_j$ and by $V_{\rho(i)}-V_{\rho(j)}$ in the definition of both $c\gc$ and $c\gcc$.
  		In the previous case, we composed two $\sigma$ maps and two $\mu$ maps. In this case, for both $\varphi$ and $\psi$, we will compose one split map and one merge map to obtain the map we are looking for. 
  		
  	    In this point, we must pay attention. Consider the relative positions of $O_i$ and $O_j$ in Figure~\ref{fig:sad_segm}. We are assuming that the map going from the grid on the left to the grid on the right is a split map, and the merge map goes the other way. We can assume this up to a rotation of the saddle trough grid moves.
  	    If we assume the opposite, namely that we have a split saddle going from right to left, a straightforward computation shows that, referring to Equation~\ref{eq:sigma_mu} in the proof of Proposition~\ref{prop:2saddle_buono}, it is enough to switch the definitions of $\sigma$ and $\mu$.
  	    In the current setting, by symmetry, we have that the relative positions of $O_i$ with respect to $O_j$ and of $O_{\rho(i)}$ with respect to $O_{\rho(j)}$ are the same, but one map is a split and one map is a merge. It follows that, sticking to the definitions of $\sigma$ and $\mu$ that we provided in the proof of Proposition~\ref{prop:2saddle_buono}, we are again composing two $\sigma$ and two $\mu$.
  	    It follows, that the definition of the maps is analogous to the case that we already showed. We then reach the conclusion through the same steps.
  	\end{proof}
  \end{prop}
  For a saddle along the symmetry axis, the procedure is analogous. The computations are simpler, and the steps follow those in \cite[Proposition 8.3.1]{ozsvath2015grid}. The $\rho$-equivariance is achieved because the segment $A$, which defines the subset of states $\mathcal{A}$, collapses to a single point $c$. After appropriate stabilizations, the configuration is shown in Figure~\ref{fig:sad_segm_2}. In fact, switching the $O$-markings introduces two crossings, which can be simplified using a pair of kinks (see Figure~\ref{fig:sad_segm_3}).
  \begin{figure}
  \begin{subfigure}{\linewidth}
  	\centering
  	\begin{tikzpicture}
  		\draw (0,0) node[above right]{\includegraphics[width=0.35\linewidth]{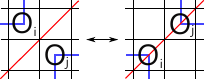}};
  		\draw (-0.15,1) node[scale=1.3] (a) {$\G$};
  		\draw (5.07,1) node[scale=1.3] (a) {$\G'$};
  		\draw (0.86,0.84) node[scale=1.1] (a) {$c$};
  		\draw (4.04,0.82) node[scale=1.1] (a) {$c'$};
  		\draw (1.025,1.01) node[scale=0.65] (a) {$\bullet$};
  		\draw (3.81,1.01) node[scale=0.65] (a) {$\bullet$};
  	\end{tikzpicture}
  	\caption{The case of a split/merge move along the symmetry axis. The blue strands represent the knot portion.}
  	\label{fig:sad_segm_2}
  \end{subfigure}
  \begin{subfigure}{\linewidth}
  	\centering
    \includegraphics[width=0.8\linewidth]{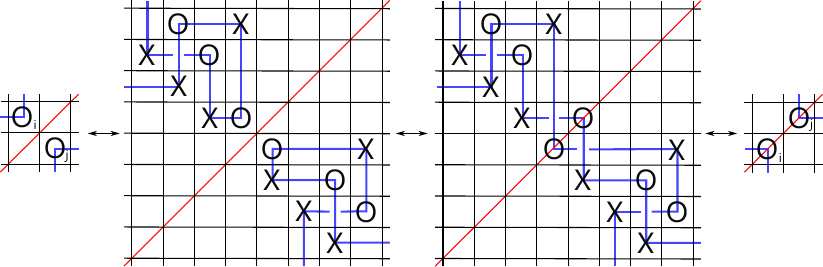}
  	\caption{Left to right. First step: equivariant stabilizations (Reidemeister R1). Second step: saddle. Third step: equivariant commutations and stabilizations (R2).}
  	\label{fig:sad_segm_3}
  \end{subfigure}
  \caption{}
  \end{figure}
  The following proposition is stated without proof.
  \begin{prop}
  		If $L'$ is obtained from $L$ by a split move along the axis, then there are $\F[U]$-module maps:
  	\begin{align*}
  		\sigma &: c\hc(L) \otimes W \rightarrow c\hc(L') \\
  		\mu &:  c\hc(L') \rightarrow c\hc(L) \otimes W
  	\end{align*}
  	with the following properties:
  	\begin{itemize}
  		\item $\sigma$ is homogeneous of degree $(-1,0)$,
  		\item $\mu$ is homogeneous of degree $(-1,-1)$.
  		\item $\mu \circ \sigma$ and $\sigma \circ \mu$ both are the multiplication by $U$.
  	\end{itemize}
  \end{prop}
  Maps $\sigma$ and $\mu$ are now used to analyze how a saddle move affects the elements of the equivariant $\tau$-sets of interest.
  The result is analogous to \cite[Theorem 8.3.6]{ozsvath2015grid}.
  \begin{thm}\label{thm:inequalities_tau}
  	Let $L$ and $L'$ be two symmetric links.
  	If $L'$ is obtained from $L$ by a pair of $\rho$-equivariant split moves, then:
  	\begin{align}
  		\tau_0^{min}(L)-2 &\leq \tau_0^{min}(L') \leq \tau_0^{min}(L) \label{ineq_doppia} \\
  		\tau_0^{max}(L) &\leq \tau_0^{max}(L') \leq \tau_0^{max}(L) +2.\label{ineq_doppia2}
  	\end{align}
  	If $L'$ is obtained from $L$ by a pair of $\rho$-equivariant saddle moves, namely one merge and one split:
  	\begin{align}
  		\tau_0^{min}(L) - 1 &\leq \tau_0^{min}(L') \leq \tau_0^{min}(L) +1 \\
  		\tau_0^{max}(L) - 1 &\leq \tau_0^{max}(L') \leq \tau_0^{max}(L) + 1.
  	\end{align}
  	If $L'$ is obtained from $L$ by a split move along the axis, then:
  	\begin{align}
  		\tau_0^{min}(L)- 1
  		 &\leq \tau_0^{min}(L') \leq \tau_0^{min}(L) \\
  		\tau_0^{max}(L) &\leq \tau_0^{max}(L') \leq \tau_0^{max}(L) + 1.
  	\end{align}
  	In each of the three cases, every stated inequality holds for $\tau_1^{min}$ and $\tau_1^{max}$ as well.
  	\begin{proof}
  		In all three cases, the proof follows exactly the steps of \cite[Theorem 8.3.6]{ozsvath2015grid}. 
  	\end{proof}
  \end{thm} 
    
 \subsubsection{Births and deaths}
  This section demonstrates that adding or removing unlinked, unknotted components to a link does not alter the $\tau$-sets.
  As in the previous section, once a map is induced between the cone complexes, the procedure is entirely analogous to that in \cite{ozsvath2015grid}. Therefore, the focus here is on the $\rho$-equivariance of the induced maps on the grid chain complex.
  Recall that, if $K$ is a knot, we call $\mathcal{U}_d(K)$ the disjoint union of $K$ with $d>0$ unknotted components.
  The following result is established.
  \begin{prop}\label{prop:tau_con_unknots}
  	If $L$ is a symmetric link of the form $\mathcal{U}_d(K)$ for some \si knot $K$, then, for any orientation on $L$:
  	\[\tau_0^{min}(L) = \tau_0^{max}(L) = \tau_0(K) \quad \text{and} \quad \tau_1^{min}(L) = \tau_1^{max}(L) = \tau_1(K).\]
  \end{prop}
  This result is a direct consequence of the following lemma, which describes how the collapsed cone homology changes when an unlinked unknotted component is added.
  \begin{lemma}\label{lem:hc_con_unknot}
  	Let $L$ be a symmetric link. Let $L'=\mathcal{U}_1(L)$ be a symmetric link obtained by adding a $\rho$-equivariant unlinked unknotted component along the axis, and let $L''=\mathcal{U}_2(L)$ be a symmetric link obtained by adding a $\rho$-equivariant pair of unlinked unknotted components. Then, there are isomorphism of bigraded $\F[U]$-modules:
  	\[c\hc(L') = c\hc(L) \oplus c\hc(L)\llbracket0,1\rrbracket,\]
  	and
  	\[c\hc(L'') = c\hc(L) \oplus c\hc(L)\llbracket0,1\rrbracket \oplus c\hc(L)\llbracket0,1\rrbracket \oplus c\hc(L)\llbracket0,2\rrbracket.\]
  \end{lemma}
  Lemma~\ref{lem:hc_con_unknot} is analogous to \cite[Lemma 8.4.2]{ozsvath2015grid}. As in the section on saddles, the proof consists of establishing the $\rho$-equivariance of the quasi-isomorphism $\Phi$ provided in \cite[Section 8.4]{ozsvath2015grid} at the level of grid chain complexes.
  After this step, the induced the map between cone complexes is a quasi-isomorphism, since $\Phi$ possesses this property.
  
  We now work with the extended grid diagrams that we mentioned in Remark~\ref{rmk:extended}. In particular, the unknotted components we add or remove will be represented by a single square marked with both an $O$- and an $X$-marking. 
  The grid complex computed on an extended grid is related by a quasi-isomorphism $\psi$ to the complex computed on a grid containing a $2\times 2$ subgrid representing an unknot. This is a modification of the destabilization induced morphism, and the isomorphism at the homology level preserves the $\tau$-sets (see \cite[Lemma 8.4.6]{ozsvath2015grid}). Since the unknots are $\rho$-equivariant, $\psi$ is also $\rho$-equivariant.
  The latter statement is proved using a procedure analogous to that described below for the map induced by the equivariant removal of unknotted components.
  
  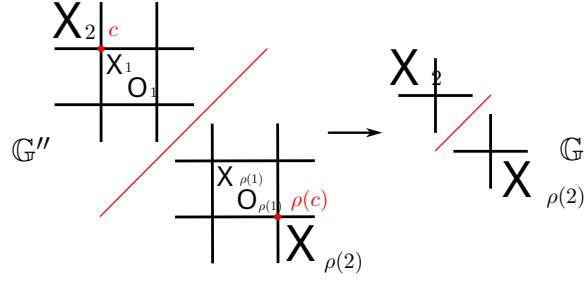
\begin{figure}
  	\centering
  	\begin{tikzpicture}
  		\draw (0,0) node[above right]{\includegraphics[width=0.5\linewidth]{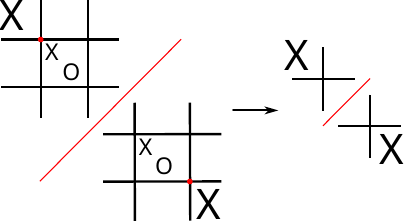}};
  		\draw (-0.15,1.65) node (a) [scale=1.1] {$\G''$};
        \draw (7,1.65) node (a) [scale=1.1] {$\G$};
        \draw (0.93,3.2) node (a) [scale=0.75] {\textcolor{red}{$c$}};
		\draw (3.53,0.97) node (a) [scale=0.75] {\textcolor{red}{$\rho(c)$}};
		\draw (1.14,2.73) node (a) [scale=0.5] {$1$};
		\draw (1.47,2.4) node (a) [scale=0.5] {$1$};
		\draw (2.76,1.24) node (a) [scale=0.45] {$\rho(1)$};
		\draw (3.01,0.88) node (a) [scale=0.45] {$\rho(1)$};
		\draw (0.63,3.2) node (a) [scale=0.75] {$2$};
		\draw (3.98,0.15) node (a) [scale=0.75] {$\rho(2)$};
		\draw (5.21,2.59) node (a) [scale=0.75] {$2$};
		\draw (6.92,1) node (a) [scale=0.75] {$\rho(2)$};
  	\end{tikzpicture}
  	\caption{Removal of a $\rho$-equivariant pair of unknotted components.}
  	\label{fig:birth/death}
  \end{figure}
  Let $\G$ and $\G''$ be two symmetric grids representing $L$ and $L''$. Up to stabilizations, we can assume that the $\rho$-equivariant pair of unknotted components are located close to an $X$-marking each, as in Figure~\ref{fig:birth/death}.
  Call $c$ the intersection point lying on the horizontal and vertical circles that separate one of the doubly marked squares from the $X$-marking that we fixed next to it. We call $\phi_c$ the quasi-isomorphism described in the proof of \cite[Lemma 8.4.7]{ozsvath2015grid}. 
  Recall that we can split $\mathbf{S}(\G'')$ as $\mathbf{I}(\G'') \sqcup \mathbf{N}(\G'')$, where the former is the subset of states containing the intersection point $c$, and the latter is the complementary. 
  Then we can consider submodules $GC^-(\G'')= \mathbf{I} \sqcup \mathbf{N}$, the latter being a subcomplex. It follows that $GC^-(\G'')= \text{Cone}(\partial_\mathbf{I}^\mathbf{N}:\mathbf{I} \rightarrow \mathbf{N})$, where $\partial_\mathbf{I}^\mathbf{N}$ is the component of the differential going from $\mathbf{I}$ to $\mathbf{N}$.
  Call $\G_1$ the grid found by removing from $\G''$ the component relative to $c$.
  Then the map
  \[ \phi_c : \text{Cone}(\partial_\mathbf{I}^\mathbf{N}) \rightarrow GC^-(\G_1)[V_1] \oplus GC^-(\G_1)[V_1] \llbracket1,0\rrbracket  \]
  is defined as:
  \[ \phi_c = \begin{pmatrix}
  	e_c & 0 \\ 0 & e\circ\mathcal{H}_1
  \end{pmatrix}. \]
  The map $e_c$ is an isomorphism induced by the identification of the states in $\mathbf{I}$ with the ones in $\St$, by cancellation of $c$.
  The chain map $\mathcal{H}_1: \mathbf{N} \rightarrow \mathbf{I} $ is defined as:
  \[ \mathcal{H}_1(\xx) = \sum_{y \in \mathbf{I}(\G'')} \sum_{\{ r \in \text{Rect}^\circ(\xx,\yy) \;|\; O_1 \in r,\, r\cap \X = \{X_1\} \}} V_2^{O_2(r)}, \dots, V_n^{O_n(r)} \yy. \]
  
  The map induced by the $\rho$-equivariant removal of unknotted components will then be $\phi = \phi_{\rho(c)}\circ \phi_c$.
  Note that: $$\rho \circ \phi_{\rho(c)}\circ \phi_c \circ \rho = \phi_c \circ \phi_{\rho(c)}.$$
  Then, showing the $\rho$-equivariance of $\phi$ is equivalent to showing that $\phi_{\rho(c)}\circ \phi_c$ is chain homotopic to $\phi_{\rho(c)}\circ \phi_c$.
   
  We can write $\mathbf{S}(\G'')$ as a new disjoint union:
  \[ \mathbf{S}(\G'') = \mathbf{I}(\G'') \sqcup \mathbf{I}_c(\G'') \sqcup \mathbf{I}_{\rho(c)}(\G'') \sqcup \mathbf{N}(\G''), \]
  where:
  \begin{itemize}
  	\item $\mathbf{I}(\G'') \subseteq \mathbf{S}(\G'')$ is the subset of the states that contain both the intersection points $c$ and $\rho(c)$;
  	\item $\mathbf{I}_c(\G'') \subseteq \mathbf{S}(\G'')$ is the subset of the states that contain $c$ but not $\rho(c)$;
  	\item $\mathbf{I}_{\rho(c)}(\G'') \subseteq \mathbf{S}(\G'')$ is the subset of the states that contain $\rho(c)$ but not $c$;
  	\item $\mathbf{N}(\G'') \subseteq \mathbf{S}(\G'')$ is the subset of the states that do not contain $c$ nor $\rho(c)$.
  \end{itemize}
  It is enough to prove the thesis for the states. 
  If $\xx \in \mathcal{I}(\G'')$ then:
  \[ \phi_{\rho(c)}\circ \phi_c (\xx) = \phi_c\circ \phi_{\rho(c)} (\xx) = e_{\rho(c)}e_c\xx, \]
  that is, the state of $\G$ that we find removing $c$ and $\rho(c)$ from $\xx$.
  When no confusion arises, the identifications $e_c$ and $e_{\rho(c)}$ are omitted.
  If $\xx \in \mathcal{I}_c(\G'')$, $\phi_c(\xx)=\mathcal{H}_{\rho(1)}(\xx)$ and $\phi_{\rho(c)}(\xx)= e_c \circ \mathcal{H}_{\rho(1)}(\xx)$. We can say that these two coincide, if every state $\yy$ that appears as an addend in $\mathcal{H}_{\rho(1)}(\xx)$ contains the intersection point $c$. 
  If not, it would follow that $c$ is a vertex of a rectangle going from $\xx$ to $\yy$. This is absurd; in fact, such a rectangle is not admitted by $\mathcal{H}_{\rho(1)}$, as it either contains $X_1$ or $X_2$.
  The case $\xx \in \mathcal{I}_{\rho(c)}(\G'')$ is symmetric to the previous one.
  The case $\xx \in \mathcal{N}(\G'')$ is the most delicate. First, observe that:
  \[ \phi_{\rho(c)}\circ \phi_c (\xx) = \mathcal{H}_{\rho(c)}\circ \mathcal{H}_c (\xx) \quad\text{and}\quad \phi_c \circ \phi_{\rho(c)} (\xx) = \mathcal{H}_c \circ \mathcal{H}_{\rho(c)} (\xx). \]
  In fact, one can prove that $\mathcal{H}_c(\xx)$ does not contain $\rho(c)$ by absurd, analogously to what we previously did. 
  Every addend $\zz \in \gc$ in $\mathcal{H}_{\rho(c)}\circ \mathcal{H}_c(\xx)$ or $\mathcal{H}_c \circ \mathcal{H}_{\rho(c)} (\xx)$ either differs from $\xx$ in $4$ or in $3$ intersection points.
  If $|\xx \cap (\xx -\zz)| = 4$, $\mathcal{H}_{\rho(c)}\circ \mathcal{H}_c$ and $\mathcal{H}_c \circ \mathcal{H}_{\rho(c)}$ count the same two disjoint rectangles in switched order, hence the maps commute.
    \begin{figure}
  	\centering
  	\begin{tikzpicture}
  		\draw (0,0) node[above right]{\includegraphics[width=0.35\linewidth]{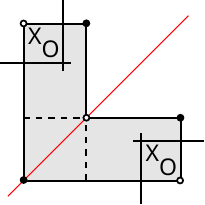}};
  		\draw (0.4,4.18) node (a) [scale=1] {$c$};
  		\draw (4.68,0.72) node (a) [scale=1] {$\rho(c)$};
  	\end{tikzpicture}
  	\caption{Example of the case: $|\xx \cap (\xx -\zz)| = 3$. The points of $\xx$ are represented by black dots, while we use white dots for $\zz$. The dashed lines show the two possible cuts.}
  	\label{fig:frame_birth_death}
  \end{figure}
  When $|\xx \cap (\xx -\zz)| = 3$, the two rectangles merge into one only domain. There are two ways to cut this domain to obtain two squares; see the example in Figure~\ref{fig:frame_birth_death}. This two different cuts provide one decomposition as an empty rectangle containing $X_1$ joined to an empty rectangle containing $X_{\rho(1)}$ (this is counted by $\mathcal{H}_{\rho(c)}\circ \mathcal{H}_c$) and one decomposition as an empty rectangle containing $X_{\rho(1)}$ joined to an empty rectangle containing $X_1$ (this is counted by $\mathcal{H}_c\circ \mathcal{H}_{\rho(c)}$). 
  We proved that the two maps are chain-homotopic. Hence, we can induce a quasi-isomorphism between the cone complexes associated with $\G''$ and $\G$. 
    
  The case of a symmetric unknotted component added or removed along the axis is analogous. Alternatively, it is not necessary to address this case, since any birth or death along the symmetry axis can be replaced by a saddle lying on the axis and a $\rho$-equivariant pair of births or deaths.
  
  \subsubsection{Proof of the bound}\label{sec:proof_genus_bound}
   The proof follows exactly as in \cite[Section 8.5]{ozsvath2015grid}. These final steps are included here for completeness.  
   \begin{prop}\label{prop:8.5.1}
   	Let $i=0,1$. Suppose that $K_1$ and $K_2$ are two \si knots so that $K_2$ is obtained from $K_1$ by a sequence of $2g$ equivariant saddle moves. Then $|\tau_i(K_1) - \tau_i(K_2)| \leq g$.
   	\begin{proof}
   		Let $L_1$ and $L_2$ be two symmetric links connected by a sequence of $m$ merge moves and $s$ split moves. Separating these moves in the three cases of Theorem~\ref{thm:inequalities_tau} and iterating the inequalities, we find that $\tau_i^{max}(L_1) \leq \tau_i^{max}(L_2) + m $ and $\tau_i^{min}(L_2) - m \leq \tau_i^{min}(L_1) $.
   		We are starting and ending with knots, so $ s = g = m $ and $\tau_i^{max}(K_j)=\tau_i^{min}(K_j)=\tau_i(K_j)$ for $j=1,2$; hence the above inequalities imply the thesis.
   	\end{proof}
   \end{prop}
   Recall from \cite{boyle2024equivariant} the definitions of equivariant unknotting numbers. Adopting the notations given by the authors Boyle and Chen, we get the following.
   \begin{thm}\label{rmk:unknot_bound}
   	Let $K$ be a \si knot.
   	Assume to unknot $K$ performing only \emph{Type A} crossing changes, then $|\tau_i(K)| \leq 2\widetilde{u}_A(K)$, $i=0,1$. Performing only \emph{Type B} crossing changes, one gets $|\tau_i(K)| \leq \widetilde{u}_B(K)$, $i=0,1$. 
   	In general: $|\tau_i(K)| \leq 2\widetilde{u}(K)$, $i=0,1$.
   \begin{proof}
   	 Since a crossing change can be decomposed into two saddle moves, the thesis is a straightforward consequence of Proposition~\ref{prop:8.5.1}.
   \end{proof}
   \end{thm}
   \begin{prop}\label{prop:8.5.3}
   	 Let $i=0,1$. Suppose that $K_1$ and $K_2$ are two \si knots, and $K_2$ is obtained by a \si $\mathcal{U}_d(K)$ by exactly $d$ $\rho$-equivariant saddle moves. 
   	 Then $\tau_i(K_1)=\tau_i(K_2)$  
   	 \begin{proof}
   	  	Now we can prove the actual statement.
   	 	Taking the saddle moves one at a time (here we mean that one saddle move is either a saddle along the axis or a $\rho$-equivariant pair of saddles), we obtain a sequence of symmetric links $L_1=\mathcal{U}_d(K_1), L_2, \dots, L_d,L_{d+1}=K_2$.
   	 	We prove that $\tau_i^{max}(L_k)=\tau_i^{min}(L_k)=\tau_i(K_1)$ by induction on $k$, where the case $k=d+1$ is the thesis. Proposition~\ref{prop:tau_con_unknots} provides the case $k=1$; we now prove the inductive step.
   	 	
   	 	Suppose that $L'$ is a symmetric link with the property that $\tau_i^{max}(L')=\tau_i^{min}(L')$; and suppose that $L$ is obtained by $L'$ by a merge move along the axis, or by a $\rho$-equivariant pair os merge moves. Then $\tau_i^{max}(L) = \tau_i^{min}(L) = \tau_i^{max}(L')$.
   	 	The case of a merge move along the axis is exactly \cite[Lemma 8.5.4]{ozsvath2015grid}. We adapt the proof of the same lemma to the case of a $\rho$-equivariant pair of moves.
   	 	Tautologically, $\tau_i^{max}(L) \geq \tau_i^{min}(L)$. Combine this with Inequalities~\ref{ineq_doppia} and~\ref{ineq_doppia2} from Theorem~\ref{thm:inequalities_tau} to get $\tau_i^{min}(L') \leq \tau_i^{min}(L) \leq \tau_i^{max}(L) \leq \tau_i^{max}(L')$.
   	\end{proof}
   \end{prop}
   Now everything is settled for the proof of the bound.
   \begin{proof}[Proof of Theorem~\ref{thm:genus_bound}.]
   	Let $i=0,1$. Fix a genus $g$ simple equivariant cobordism from $K_1$ to $K_2$. Let $K'_1$ and $K'_2$ be as in the statement of the equivariant normal form from Theorem~\ref{thm:normal_form}. By Proposition~\ref{prop:8.5.3}, $\tau_i(K_1)=\tau_i(K'_1)$ and $\tau_i(K_2)=\tau_i(K'_2)$. By Proposition~\ref{prop:8.5.1} $|\tau_i(K'_1)- \tau_i(k'_2)|\leq g$, this concludes.
   \end{proof}
    
 \subsection{The induced morphism}\label{sec:finale}
  This final section is dedicated to the discussion of the following conjecture.
  \begin{conj}\label{congettura}
  	Let $K_1$ and $K_2$ be two \si knots connected by a genus $g$ simple equivariant cobordism.
  	Then there exists an $\F[U]$-linear $(-2g,-g)$-homogeneous homomorphism:
  	\[ \phi: \hc(K_1) \rightarrow \hc(K_2). \]
  	Moreover, $\phi$ is induced by a chain map and its image is not contained in the torsion submodule.
  \end{conj}
  In addition to its intrinsic interest, the existence of such a map would directly imply Theorem~\ref{thm:genus_bound} as a corollary.

  In many knot homology theories, maps induced by cobordisms have been constructed by defining explicit chain-level maps for elementary moves. In Khovanov homology, Bar-Natan \cite{BarNatan2005} associated chain maps to births, deaths, and saddles, yielding functoriality up to sign. In knot Floer homology, cobordism maps were developed in the sutured and decorated settings by Juhász \cite{Juhasz2008} and Juhász–Marengon \cite{JuhaszMarengon2018}, and were subsequently placed in a fully functorial framework by Zemke \cite{Zemke2019}.
  
  In grid homology, explicit chain maps corresponding to births, deaths, and saddles in the $\widehat{GH}$ setting are introduced by Sarkar \cite{Sarkar2011}, who models cobordisms via sequences of diagrammatic grid moves and defines the associated chain maps. Later constructions of cobordism maps in the grid setting build on this diagrammatic framework; for instance, Graham \cite{GrahamThesis} used the Sarkar construction to partially define the maps in the $GH^-$ setting. More recently, Baldwin–Lidman–Wong \cite{baldwin2022lagrangian} provided the cobordism induced maps in the Legendrian and Lagrangian setting for the $\widetilde{GH}$ version. 
 
  The following discussion adopts the constructions in \cite[Section 3]{baldwin2022lagrangian} to attempt the definition of suitable maps at the level of $GC^-$. 
  To satisfy the conjecture, it is necessary to construct, for any elementary equivariant cobordism, a chain map $\phi: \gc \rightarrow \gc$ that is $\rho$-equivariant and homogeneous of the prescribed bidegree. Furthermore, it is essential to control the image of a cycle $\xx \in \gc$ such that the homology class of its image $[\phi(\xx)]$ is of non-torsion. More specifically, it is required to track such a cycle through the composition of all elementary cobordism maps. This approach ensures that the image of the final map is not contained in the torsion part.
  
  Note that in \cite{baldwin2022lagrangian} the authors define contravariant morphisms, whereas we consider covariant maps. Accordingly, births correspond to deaths and vice versa.
  Results are proposed for the removal of a \si unknotted component, the removal of a $\rho$-equivariant pair of unknotted components, and the case of a pair of $\rho$-equivariant saddle moves. After Lemma~\ref{label:1} and Lemma~\ref{lemma:due_selle}, the remaining cases will be discussed. 
  
  \begin{lemma}\label{label:1}
 	Let $L$ and $L'$ be two symmetric links and assume that $L'$ is obtained from $L$ by an equivariant death move. For any pair of grids $\G$ and $\G'$ respectively representing $L$ and $L'$, there exists a homomorphism:
 	\[ \phi: \gcc \rightarrow \gc \]
 	that sends $\lambda^{\pm}(\G')$ to $\lambda^{\pm}(\G)$. Furthermore, this map is $\rho$-equivariant and has Maslov-Alexander bigrading either $(1,0)$ or $(2,0)$. The bidegrees are, respectively, those of a death move on the axis and of a pair of symmetric death moves.  
 	\begin{proof}
 		We must prove the thesis in the case of a single unknotted circle appearing on the axis and in the case of a pair of equivariant unknots.
 		In both cases, the map will be analogous, for the unblocked case, to the map defined in \cite[Lemma 3.9]{baldwin2022lagrangian}.
 		Since isotopies induce homotopic equivalence between grid complexes (see \cite{manolescuoz} or \cite[Chapter 5]{ozsvath2015grid}), it is enough to prove the thesis for any pair $\G$ and $\G'$ of symmetric grids representing $L$ and $L'$.
 		There are two cases:
 		\begin{enumerate}
 			\item death of an unknotted component along the axis;
 			\item death of a pair of unknotted components swapped by $\rho$.
 		\end{enumerate} 
 		\begin{figure}
 			\centering
 			\includegraphics[width=0.3\linewidth]{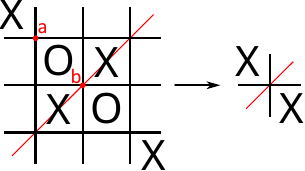}
 			\caption{Death of an equivariant component along the symmetry axis.}
 			\label{fig:axis_birth}
 		\end{figure}
 		\emph{Case i.}  Up to $S$-stabilization, we can state that $\G$ has a pair of $\rho$-symmetric $X$-markings close the axis: $X_1$ and $X_{\rho(1)}$. Assume to obtain $\G'$ adding, between $X_1$ and $X_{\rho(1)}$, a $2\times 2$ square as in Figure~\ref{fig:axis_birth}.
 		Call $a$ and $b$ the intersection points (refer to Figure~\ref{fig:axis_birth}) obtained from the new vertical and horizontal circles. Observe that we can split $\mathbf{S}(\G')$ as a disjoint union:
 		\[ \mathbf{S}(\G') = AB\sqcup NB\sqcup AN\sqcup NN, \]
 		where:
 		\begin{itemize}
 			\item $ AB= \{\xx \in \mathbf{S}(\G') \;|\; \{ a,b \}\subseteq \xx \};$
 			\item $ AN= \{\xx \in \mathbf{S}(\G') \;|\; \{ a,b \}\cap \xx = \{a \}\};$
 			\item $ NB= \{\xx \in \mathbf{S}(\G') \;|\; \{ a,b \}\cap \xx = \{b\}\};$
 			\item $ NN= \{\xx \in \mathbf{S}(\G') \;|\; \{ a,b \}\cap \xx = \emptyset \}.$
 		\end{itemize} 
 		Consequently, there is an induced decomposition:
 		\[ \gcc = AB^-(\G') \oplus AN^-(\G') \oplus NB^-(\G') \oplus NN^-(\G'), \]
 		the summmands being the $\F[V_1,\dots,V_n]$-module freely generated on the corresponding set of states.
 		Note that $\gcc$ is a $\F[V_1,\dots,V_n]$-module, while $\gc$ is a $\F[V_3,\dots,V_n]$-module. Hence we quotient $\gcc$ by the submodule $(V_1,V_2)\gcc$. This quotient preserves the splitting, which we rename for convenience:
 		\[ \frac{\gcc}{(V_1,V_2)\gcc} = AB(\G') \oplus AN(\G') \oplus NB(\G') \oplus NN(\G'). \]
 		There is a sequence $NB(\G')\subseteq AB(\G')\oplus NB(\G') \subseteq \gcc$ of subcomplexes. This follows from the fact that any rectangle that starts in $b$ or ends in $a$ must contain an $X$-marking, either $X_1$ or one of the two new markings. Crossing an $X$-marking is an obstruction to appearing in the differential sum. 
 		Let $AB(\G')$ be the quotient of $ AB(\G')\oplus NB(\G')$ by $NB(\G')$. There is an isomorphism:
 		\[ e: AB(\G') \rightarrow \gc, \]
 		that is, the $\F[V_1,\dots,V_n]$-linear extension of the map defined on states that sends $\xx\in AB$ in $\xx\setminus\{a,b\} \in \St$.
 		The extension is an isomorphism because, for all $\xx,\yy \in AB$ there is a natural bijection:
 		\[ \text{Rect}^{\circ}_{\G'}(\xx,\yy) \rightarrow \text{Rect}^{\circ}_\G(\xx\setminus\{a,b\},\yy\setminus\{a,b\}), \]
 		which preserves the multiplicity of the $O$-markings and the condition to avoid the $X$-markings. Moreover, a quick computation shows that $e$ is bigraded, namely $(0,0)$- homogeneous.
 		
 		Given two states $\xx\in NB$ and $\yy \in AB$, call
 		\[ \text{Rect}_{AB}(\xx,\yy)\subset \text{Rect}_{\G'}(\xx,\yy) \]
 		the set of rectangles $p$ such that:
 		\begin{itemize}
 			\item $p\cap(\Oo\cup\{O_1,O_2\})\supseteq\{O_1,O_2\},$
 			\item $p\cap(\X\cup\{X_1,X_2\})=\{X_1,X_2\},$
 			\item $\text{Int}(p)\cap \xx=\text{Int}(p)\cap \yy = \{b\}.$
 		\end{itemize}
 		Define $\psi:NB(\G')\rightarrow AB(\G')$ as the linear extension of the map defined on the states as:
 		\[ \psi(\xx)= \sum_{\yy\in AB} \sum_{p\in \text{Rect}_{AB}(\xx,\yy)}V_3^{O_3(p)}\cdots V_n^{O_n(p)} \yy. \]
 		Define
 		\[ \phi: \gcc \rightarrow \gc \]
 		to be the $\F[V_1,\dots,V_n]$-linear map given as the composition
 		\[ \phi = e\circ \psi\circ \pi_{NB}, \]
 		where $\pi_{NB}:\gcc \rightarrow NB(\G')$ is the projection onto the summand $NB^-(\G')$ composed with the projection to the quotient $NB(\G')$.
 		
 		To begin with, we must prove that $\phi$ is a chain map. The map $e$ is clearly a chain map, so we must prove:
 		\[ \partial^-_{AB} \circ (\psi\circ\pi_{NB}) = (\psi\circ\pi_{NB}) \circ \partial^-_{\G'}.  \]
 		The left-hand side vanishes on any generator $\xx \notin NB(\G')$. The right-hand side, in particular $\pi_{NB}\circ \partial^-_{\G'}$, vanishes for generators $\xx \notin AB(\G')\cup NB(\G')$. In fact, any generator $\xx \notin AB(\G')\cup NB(\G')$ whose image $\partial^-_{\G'}(\xx)$ has non trivial component in $NB^-(\G')$, is sent to zero going to the quotient $NB(\G')$, as each coefficient contains either $V_1$ or $V_2$. The proof proceeds as in \cite[Lemma 3.10]{baldwin2022lagrangian} (Note that it refers to the computations for the vanishing of the square of the differential in \cite{manolescuoz}, or \cite[Chapter 4]{ozsvath2015grid}). 
 		
 		The conditions on the canonical cycles and the bidegree of the map are proved exactly as in \cite[Lemma 3.11, Lemma 3.12]{baldwin2022lagrangian}. It remains to show $\rho$-equivariance, namely that $\phi\circ\rho$ is chain homotopic to $\rho\circ\phi$. Here, $\rho$ acts on $\gcc$ in the first composition, and on $\gc$ in the latter.
 		Equivalently, we are going to show that $\phi$ is chain homotopic to $\rho\circ\phi\circ\rho$.
 		As in the proof of Proposition~\ref{prop:rhochain}, a computation shows that, by symmetry of the set we do the sum on, we have the commutation of the two compositions.
 		More precisely, $\rho(a)=a'$ and $\rho(b)=b$. We obtain a subset $A'B\subseteq \Stt$ of the states containing $\{a',b\}$, such that $\rho(AB)=A'B$. The same goes defining a subset $N'B\subseteq \Stt$
 		
 		Let us write explicitly the two compositions, on a state $\xx \in \Stt$:
 		\[\phi(\xx)=\sum_{\yy\in AB}\sum_{r\in \text{Rect}_{AB}(\xx,\yy)} V_{\rho(3)}^{O_3(r)}\cdots V_{\rho(n)}^{O_n(r)}\cdot (\yy\setminus \{a,b\})\] 
 		when $\xx\in NB$, and vanishes elsewhere, while
 		\[\rho\circ\phi\circ\rho(\xx)=\sum_{\yy\in A'B}\sum_{r\in \text{Rect}_{A'B}(\xx,\yy)} V_{\rho(3)}^{O_3(r)}\cdots V_{\rho(n)}^{O_n(r)}\cdot (\yy\setminus \{a',b\})
 		\]
 		when $\xx\in N'B$, and vanishes elsewhere.
 		It follows that, for $\xx \in \Stt \setminus (NB \cup N'B)$, the homomorphisms coincide, as they are both trivial. Note that $\text{Rect}_{AB}(\xx,\yy) = \emptyset$ if $\xx$ contains the intersection point $a'$. This is due to the condition that a rectangle in $\text{Rect}_{AB}$ must contain the $2 \times 2$ square representing the unknot. So $\phi$ vanishes on $A'B$. The same goes for $\rho\circ \phi \circ\rho$, which vanishes on $AB$.
 		
 		It follows that the only non-trivial case is $\xx \in NB \cup N'B$.
 		Here, the assumption of the $X$-markings located NW and SE of the $2 \times 2$ square imply that, for $\text{Rect}_{AB}(\xx,\yy)$ to be non-empty for some $\yy$, $\xx$ must intersect at least one between the SW and the NE vertex of the $2 \times 2$ square.
 		In this condition, for any $\yy \in AB$ such that $\text{Rect}_{AB}(\xx,\yy) \neq \emptyset$, we have that the rectangle is unique and it intersects no $O$-markings (aside from $O_1$ and $O_2$). Furthermore, for any such state $\yy$ we have that $\text{Rect}_{A'B}(\xx,\rho(\yy))$ consists of one only rectangle $r$ such that $r \cap \Oo = \{O_1,O_2\}$ and $\yy \setminus \{a,b\}$ coincides with $\rho(\yy) \setminus \{a',b\}$ as a element of $\St$. This concludes the $\rho$-equivariance of the map $\phi$.
 		
 		\begin{figure}
 			\centering
 			\includegraphics[width=0.45\linewidth]{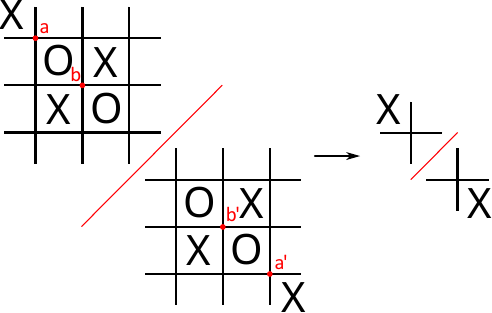}
 			\caption{Death of a pair of $\rho$-equivariant unknotted components.}
 			\label{fig:birth_equiv}
 		\end{figure}
 		\emph{Case ii.} In this case we can assume to perform a move as in Figure~\ref{fig:birth_equiv}. Here we are assuming that in $\G$ we can find a pair of symmetric $X$-markings, that is a valid assumption up to equivariant stabilization. 
 		The constructions are analogous to the ones in the previous case. The distinct intersection points here are four: $a$, $b$, $a'=\rho(a)$ and $b'=\rho(b)$.
 		We split the again the states as $\Stt = AB\sqcup NB \sqcup AN \sqcup NN$, but the definitions are slightly different:
 		\begin{itemize}
 			\item $ AB= \{\xx \in \mathbf{S}(\G') \;|\; \{ a,b,a',b' \}\subseteq \xx \};$
 			\item $ AN= \{\xx \in \mathbf{S}(\G') \;|\; \{ a,a' \}\subseteq \xx ,\; |\xx\cap\{b,b'\}|<2\};$
 			\item $ NB= \{\xx \in \mathbf{S}(\G') \;|\; \{ b,b' \}\subseteq \xx ,\; |\xx\cap\{a,a'\}|<2\};$
 			\item $ NN= \Stt \setminus \{ AB\sqcup NB \sqcup AN \}.$
 		\end{itemize} 
 		The sequence of subcomplexes is the same as before, as the observations on the isomorphism $e$.
 		Of course, this time we have two subsets of $\text{Rect}_{\G'}(\xx,\yy)$: $\text{Rect}_{AB}(\xx,\yy)$ and $\text{Rect}_{A'B'}(\xx,\yy)$. These two subsets are in bijection through $\rho$, and the bijection is such that $O_i(\rho(r))=O_{\rho(i)}(r)$ for any $r\in \text{Rect}_{AB}(\xx,\yy)$. 
 		Referring to this two subset, there are natural definitions of $\psi$ and $\psi'$ as in the previous case. 
 		We define the map $\phi:\gcc \rightarrow \gc$  as:
 		\[ \phi := e \circ \psi' \circ \psi \circ \pi_{NB}. \]
 		Checking that this map is a $(2,0)$-homogeneous chain map that sends the canonical cicles of $\G'$ into the ones of $\G$ is analogous to the previous case. The same goes for the $\rho$ equivariance of $e$ and $\pi_{NB}$. To show that $\rho\circ\psi\circ\psi' \sim \psi\circ\psi'\circ\rho$, it is enough to observe that:
 		\[ \rho\circ\psi\circ\rho = \psi' \]
 		through the same steps of the previous case.
 	\end{proof}
 \end{lemma}

 \begin{lemma}\label{lemma:due_selle}
 	Let $L$ and $L'$ be two symmetric links and assume that $L'$ is obtained from $L$ by a pair of saddle moves, symmetric with respect to the axis. For any pair of symmetric grids $\G$ and $\G'$ respectively representing $L$ and $L'$, there exists a homomorphism:
 	\[ \phi: \gcc \rightarrow \gc \]
 	that sends $\lambda^{\pm}(\G')$ to $\lambda^{\pm}(\G)$. Furthermore, this map is $\rho$-equivariant and has Maslov-Alexander bigrading:
 	\[\begin{cases}
 		(-2,0)  &\text{if }|L'|=|L|+2\\
 		(-2,-2) &\text{if }|L'|=|L|-2\\
 		(-2,-1) &\text{if }|L'|=|L|,
 	\end{cases}\]
 	where $|L|$ is the number of components of the link $L$.
 	\begin{proof}
 	As for the previous lemmas, it is enough to show the thesis for a specific pair of symmetric grids $\G$ and $\G'$, respectively representing $L$ and $L'$.
 	    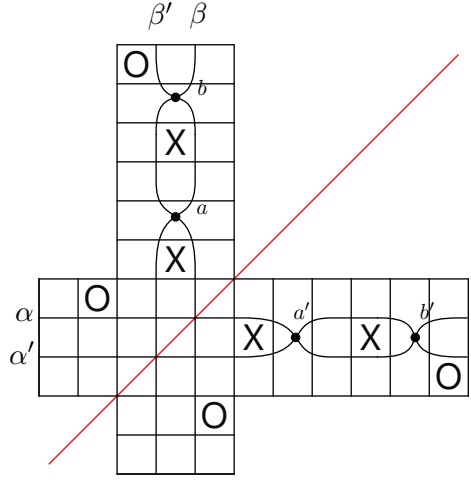
\begin{figure}
 			\centering
 			\begin{tikzpicture}
 			    \draw (0,0) node[above right]{\includegraphics[width=0.45\linewidth]{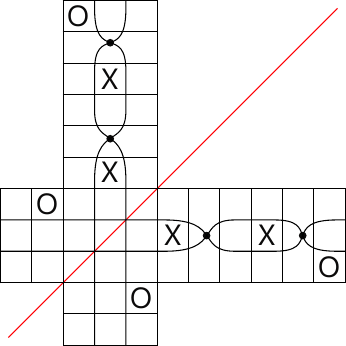}};
 				\draw (2.3,3.65) node (a) [scale=0.75] {$a$};
 				\draw (2.3,5.26) node (a) [scale=0.75] {$b$};
 				\draw (3.62,2.3) node (a) [scale=0.75] {$a'$};
 				\draw (5.28,2.3) node (a) [scale=0.75] {$b'$};
 				\draw (2.25,6.2) node (a) [scale=1] {$\beta$};
 				\draw (1.75,6.2) node (a) [scale=1] {$\beta'$};
 				\draw (-0.05,2.25) node (a) [scale=1] {$\alpha$};
 				\draw (-0.08,1.75) node (a) [scale=1] {$\alpha'$};
 			\end{tikzpicture}
 			\caption{Equivariant saddle on a grid diagram representing simultaneously $\G$ and $\G'$.}
 			\label{fig:croce_saddle}
 		\end{figure}
 		As shown in \cite[Proposition 3.1]{baldwin2022lagrangian}, we can assume that $\G$ and $\G'$ differ only on a $2$-rows $\times$ $2$-columns cross, as in Figure~\ref{fig:croce_saddle}. Observe how the exchanged markings could be $O$'s instead. In \cite{baldwin2022lagrangian}, this case is addressed separately; however, this is due to the Legendrian grid moves constraints. Up to equivariant grid stabilization, we can assume that the saddle always switches $X$-markings. 
 		We can depict both $\G$ and $\G'$ on the same grid diagram at the same time; see Figure~\ref{fig:simultanea}. We call $\alpha$ and $\beta$ the horizontal and vertical circles belonging to $\G$ but not to $\G'$, and we use $\alpha'$, $\beta'$ for the viceversa. We call $a$, $b$ the intersection points of $\alpha$ with $\alpha'$, and $a'=\rho(a)$, $b'=\rho(b)$ the intersection points of $\beta'$ with $\beta'$, as shown in the figure.
 		Again, the relative positions of the $X$-markings we switch are not restrictive, up to \si grid moves.
 		
 		Consider the pentagon counting map $\phi$ defined in \cite[Section 3.1.1]{baldwin2022lagrangian}. To be precise, since we work with the unblocked flavor, we should refer to \cite[Equation 5.2]{ozsvath2015grid}. We define the map $\phi:\gcc \rightarrow \gc$ as the composition of the map counting pentagons in  $a$ with the map counting pentagons in $a'$: $\phi:= \phi_{a'}\circ \phi_{a}$. 
 		Observe that \cite[Lemma 5.1.4]{ozsvath2015grid} implies that also $\phi$ is a chain map, while the statement on the bidegree is a straightforward computation, as in \cite[Lemma 3.4]{baldwin2022lagrangian}. The proof that the canonical states of $\G'$ are sent to the canonical states of $\G$ follows from the observation in \cite[Lemma 3.3]{baldwin2022lagrangian}. We only have to show that the map $\phi$ is $\rho$-equivariant.
 		
 		Start by noticing that the conjugated $\rho\phi\rho$ coincides with $\phi_{a}\phi_{a'}$, so the thesis is equivalent to showing that $\phi_{a'}\phi_{a}+\phi_{a}\phi_{a'}$ is null homotopic. To do so, we must consider a hexagon counting map $H$ on the combined diagram, in analogy with the one in \cite[Equation 5.8]{ozsvath2015grid}. The main difference will be the shape of the hexagons we want to count.
 		Fix grid states $\xx \in \Stt$ and $\yy \in \St$. We call a \emph{hexagon} from $\xx$ to $\yy$ an embedded disk $h$ in the combined torus, whose boundary is in the union of the horizontal and vertical circles and such to satisfy the following conditions:
 		\begin{itemize}
 			\item At any of the six corner points $c$ of $h$, the hexagon contains exactly one of the four quadrants determined by the two intersecting curves at $c$.
 			\item Four of the corner points are in $\xx \cup \yy$, and the remaining two are $a$ and $a'$.
 			\item $\partial(\partial_\alpha h) = y - x$.
 		\end{itemize}    
 		We denote the set of hexagons from $\xx$ to $\yy$ by $\text{Hex}(\xx,\yy)$. A hexagon $h\in \text{Hex}(\xx,\yy)$ is \emph{empty} if $\text{Int}(h)\cap \xx = \text{Int}(h)\cap \yy = \emptyset$. We denote by $\text{Hex}^\circ(\xx,\yy)$ such hexagons.
 		Consider the $\F[V_1,\dot,V_n]$-module homomorphism $H:\gcc \rightarrow \gc$ defined on the states $\xx\in\Stt$ as:
 		\[ H(\xx) = \sum_{\yy\in\St} \sum_{\{ h \in \text{Hex}^\circ(\xx,\yy) \;|\; h\cap\X = \emptyset \}} V_1^{O_1(h)}\dots V_n^{O_n(h)} \cdot \yy. \]
 		We claim that $H$ provides a chain homotopy from $\phi_{a'}\phi_{a}$ to $\phi_{a}\phi_{a'}$.
 		
 		For a domain $\psi \in \pi(\xx,\yy)$ call $N(\psi)$ the number of ways $\psi$ decomposes as either:
 		\begin{itemize}
 			\item $r*h$, where $r$ is an empty rectangle and $h$ is an empty hexagon;
 			\item $h*r$, where $h$ is an empty hexagon and $r$ is an empty rectangle;
 			\item $p*p'$, where $p$ is an empty pentagon with a corner in $a$ and $p'$ is an empty pentagon with a corner in $a'$;
 			\item $p'*p$, where $p'$ is an empty pentagon with a corner in $a'$ and $p$ is an empty pentagon with a corner in $a$.
 		\end{itemize}    
 		Clearly, $\phi_{a'}\phi_{a}+\phi_{a}\phi_{a'} + \partial H + H \partial$ applied on a state $\xx\in\Stt$ coincides with:
 		\[ \sum_{\yy\in\St} \sum_{\psi\in \pi(\xx,\yy)} N(\psi) V_1^{O_1(\psi)}\dots V_n^{O_n(\psi)} \cdot \yy. \]
 		The proof ends showing that the sum above is zero. The strategies and cases are analogous to those in the proof of \cite[Lemma 5.1.6]{ozsvath2015grid}.
 		\end{proof}
 \end{lemma}
 The difficulty in proving Conjecture~\ref{congettura} lies in the remaining cases, which include the birth of an unknotted component along the axis, the birth of a $\rho$-equivariant pair of unknotted components, and a symmetric saddle move on the axis.
 It is not necessary to construct a map for every remaining case. 
 For instance, constructing a homogeneous birth-induced map of bidegree $(1,1)$ or $(2,2)$, depending on the number of components added, is sufficient. In fact, a saddle along the axis can always be substituted with a $\rho$-equivariant pair of saddles and a birth or death along the axis.
 By a similar argument, it suffices to construct the morphism in the case of a saddle on the axis and of a $\rho$-equivariant pair of births.
 
 As regards birth moves, no analogue is found in \cite{baldwin2022lagrangian}. In fact, the birth maps defined there correspond, in our setting, to death maps. However, such maps are not considered, since the authors work with Lagrangian cobordisms. We were unable to identify a suitable candidate map, so this case remains open.
 
 Let $\G'$ be a symmetric grid obtained from a symmetric grid $\G$ through a saddle move on the axis, as in Figure~\ref{fig:sad_segm_2}. Starting from the pentagon (or triangle) counting chain map defined in \cite{baldwin2022lagrangian}, one constructs the map in the minus case:
 \[\phi: \gc \rightarrow \gcc.\]
 The only remaining property to verify is $\rho$-equivariance. 
 In defining the map $\phi$, it is assumed that $\G$ and $\G'$ differ by a horizontal circle. Pentagons, or triangles, are counted on a grid $\overline{\G}$ that simultaneously represents $\G$ and $\G'$, and thus contains two horizontal circles intersecting in two points. Then, the map $\rho \circ \phi \circ \rho$ is the pentagon, or triangle, counting map on the grid $\rho\overline{\G}$, which still represents $\G$ and $\G'$ simultaneously, but now regarded as differing by a vertical circle. This is possible because the saddle lies on the symmetry axis. In this setting, the computations are less straightforward and more delicate, and to date we have not been able to complete them.


\bibliographystyle{amsalpha}
\bibliography{Bibliography.bib}

@article{manolescu2025rasmussen,
	title={A {R}asmussen invariant for links in $\mathbb{RP}^3$},
	author={Manolescu, Ciprian and Willis, Michael},
	journal={Transactions of the American Mathematical Society, Series B},
	volume={12},
	number={21},
	pages={789--830},
	year={2025}
}

@article{Zemke2019,
	author  = {Zemke, Ian},
	title   = {Link cobordisms and functoriality in link {F}loer homology},
	journal = {Journal of Topology},
	volume  = {12},
	number  = {1},
	year    = {2019},
	pages   = {94--220},
}

@article{hendricks2017involutive,
	title={Involutive {H}eegaard {F}loer homology},
	author={Hendricks, Kristen and Manolescu, Ciprian},
	journal={arXiv preprint arXiv:1507.00383},
	year={2017}
}

@article{xiao2026real,
	title={{Real link Floer homology}},
	author={Xiao, Yonghan},
	journal={arXiv preprint arXiv:2604.21240},
	year={2026}
}

@article{parikh2024localization,
	title={{Localization and the Floer homology of strongly invertible knots}},
	author={Parikh, Aakash},
	journal={arXiv preprint arXiv:2408.13892},
	year={2024}
}

@article{hendricks2025note,
	title={{A note on real Heegaard Floer homology and localization}},
	author={Hendricks, Kristen},
	journal={arXiv preprint arXiv:2508.03897},
	year={2025}
}

@article{dioguz,
	title={Equivariant Q-sliceness of strongly invertible knots},
	author={Di Prisa, Alessio and {\c{S}}avk, O{\u{g}}uz},
	journal={Trans. Amer. Math. Soc},
	pages={8},
	year={2025}
}

@article{jua_thu_zem,
	title = "Naturality and Mapping Class Groups in {H}eegard {F}loer Homology",
	author = "Andr{\'a}s Juh{\'a}sz and Thurston, {Dylan P.} and Ian Zemke",
	note = "Publisher Copyright: {\textcopyright} 2021 American Mathematical Society.",
	year = "2021",
	month = sep,
	volume = "273",
	pages = "1--185",
	journal = "Memoirs of the American Mathematical Society",
	publisher = "American Mathematical Society",
	number = "1338",
}

@article{sarkar2015moving,
	title={Moving basepoints and the induced automorphisms of link {F}loer homology},
	author={Sarkar, Sucharit},
	journal={Algebraic \& Geometric Topology},
	volume={15},
	number={5},
	pages={2479--2515},
	year={2015},
	publisher={Mathematical Sciences Publishers}
}

@article{dipriframba1,
	title={Solvability of concordance groups and {Milnor} invariants},
	author={Di Prisa, Alessio and Framba, Giovanni},
	journal={Revista Matem\'atica Iberoamericana},
	year={2025},
	volume = {41},
	number = {5},
	pages = {1957-1972}
}

@article{dipriframba2,
	title={A new invariant of equivariant concordance and results on 2-bridge knots},
	author={Di Prisa, Alessio and Framba, Giovanni},
	journal={Algebraic \& Geometric Topology},
	year={2025},
	volume = {25},
	number = {2},
	pages = {1117-1132}
}

@article{alfieri2021strongly,
	title={Strongly invertible knots, invariant surfaces, and the {A}tiyah--{S}inger signature theorem},
	author={Alfieri, Antonio and Boyle, Keegan},
	journal={Michigan Mathematical Journal},
	volume={74},
	number={4},
	pages={845--861},
	year={2024},
	publisher={University of Michigan, Department of Mathematics}
}

@article{snape,
	title={Homological invariants of strongly invertible knots}, 
	author={Snape, Michael},
	year={2018},
	journal={PhD thesis, University of Glasgow},
}

@article{boyle2021equivariant,
	author = {Boyle, Keegan and Issa, Ahmad},
	title = {Equivariant 4-genera of strongly invertible and periodic knots},
	journal = {Journal of Topology},
	volume = {15},
	number = {3},
	pages = {1635-1674},
	year = {2022}
}

@article{sakuma,
	author = {Sakuma, Makoto},
	year = {1986},
	month = {01},
	pages = {},
	title = {On strongly invertible knots},
	journal = {Algebraic and topological theories (Kinosaki, 1984)}
}

@Book{smith,
	author = { Bass, Hyman and Morgan, John W.},
	title = { The {S}mith conjecture},
	publisher = { Academic Press Orlando [Fla.] },
	pages = { xv, 243 p. : },
	year = { 1984 },
}

@article{collarilisca,
	title={{Strongly Invertible Legendrian Links}},
	author={Collari, Carlo and Lisca, Paolo},
	journal={arXiv preprint arXiv:2311.07974},
	year={2023}
}

@article{manolescuoz,
	title={On combinatorial link {F}loer homology},
	author={Manolescu, Ciprian and Ozsv{\'a}th, Peter and Szab{\'o}, Zolt{\'a}n and Thurston, Dylan P},
	journal={Geometry \& Topology},
	volume={11},
	number={4},
	pages={2339--2412},
	year={2007},
	publisher={Mathematical Sciences Publishers}
}

@article{sano2024involutive,
	author    = {Taketo Sano},
	title     = {Involutive {K}hovanov homology and equivariant knots},
	journal   = {Algebraic \& Geometric Topology},
	volume    = {25},
	number    = {8},
	pages     = {5059--5111},
	year      = {2025},
}

@article{miller2023strongly,
	title={Strongly invertible knots, equivariant slice genera, and an equivariant algebraic concordance group},
	author={Miller, Allison N. and Powell, Mark},
	journal={Journal of the London Mathematical Society},
	volume={107},
	number={6},
	pages={2025--2053},
	year={2023},
	publisher={Wiley Online Library}
}

@phdthesis{GrahamThesis,
	author = {Graham, David},
	title  = {Studying Surfaces in 4-Dimensional Space Using
	Combinatorial Knot {F}loer Homology},
	school = {Brandeis University},
	year   = {2012}
}

@article{baldwin2022lagrangian,
	title={Lagrangian cobordisms and Legendrian invariants in knot {F}loer homology},
	author={Baldwin, John A and Lidman, Tye and Wong, C-M Michael},
	journal={Michigan Mathematical Journal},
	volume={71},
	number={1},
	pages={145--175},
	year={2022},
	publisher={University of Michigan, Department of Mathematics}
}

@article{hirasawa2022invariant,
	title={Invariant {S}eifert surfaces for strongly invertible knots},
	author={Hirasawa, Mikami and Hiura, Ryota and Sakuma, Makoto},
	journal={Essays in geometry--dedicated to Norbert A’Campo},
	pages={325--349},
	year={2023}
}

@article{lobb2021refinement,
	title={A refinement of {K}hovanov homology},
	author={Lobb, Andrew and Watson, Liam},
	journal={Geometry \& Topology},
	volume={25},
	number={4},
	pages={1861--1917},
	year={2021},
	publisher={Mathematical Sciences Publishers}
}

@article{lamm2022symmetric,
	title={Symmetric diagrams for all strongly invertible knots up to 10 crossings},
	author={Lamm, Christoph},
	journal={arXiv preprint arXiv:2210.13198},
	year={2022}
}

@article{boyle2024equivariant,
	title={Equivariant unknotting numbers of strongly invertible knots},
	author={Boyle, Keegan and Chen, Wenzhao},
	journal={arXiv preprint arXiv:2412.09797},
	year={2024}
}

@article{BarNatan2005,
	author  = {Bar-Natan, Dror},
	title   = {Khovanov's homology for tangles and cobordisms},
	journal = {Geometry \& Topology},
	volume  = {9},
	year    = {2005},
	pages   = {1443--1499}
}

@article{KFfour-ball,
	author = {Peter Ozsv{\'a}th and Zolt{\'a}n Szab{\'o}},
	title = {{Knot {F}loer homology and the four-ball genus}},
	volume = {7},
	journal = {Geometry \& Topology},
	number = {2},
	publisher = {MSP},
	pages = {615 -- 639},
	year = {2003},
}

@book{ozsvath2015grid,
	title={Grid homology for knots and links},
	author={Ozsv{\'a}th, Peter S and Stipsicz, Andr{\'a}s I and Szab{\'o}, Zolt{\'a}n},
	volume={208},
	year={2015},
	publisher={American Mathematical Soc.}
}

@article{CROMWELL199537,
	title = {Embedding knots and links in an open book I: {B}asic properties},
	journal = {Topology and its Applications},
	volume = {64},
	pages = {37-58},
	year = {1995},
	author = {Peter R. Cromwell},}

@article{tait,
	title = {On knots I, II, III},
	journal = {Scientific papers},
	volume = {1},
	pages={273-347},
	year = {1898},
	publisher={Cambridge},
	author = {Peter G. Tait},}

@article{legendriangridinvaria,
	title={Legendrian knots, transverse knots and combinatorial {F}loer homology},
	author={Ozsv{\'a}th, Peter and Szab{\'o}, Zolt{\'a}n and Thurston, Dylan P},
	journal={Geometry \& Topology},
	volume={12},
	number={2},
	pages={941--980},
	year={2008},
	publisher={Mathematical Sciences Publishers}
}

@article{Juhasz2008,
	author  = {Juhász, András},
	title   = {Cobordisms of sutured manifolds},
	journal = {Advances in Mathematics},
	volume  = {219},
	year    = {2008},
	pages   = {1361--1430}
}

@article{khovanov2025symmetries,
	author = {Mikhail Khovanov and Taketo Sano},
	title = {Symmetries of equivariant Khovanov homology},
	journal = {arXiv:2509.03785},
	year = {2025}
}

@incollection{lipshitz2022khovanov,
	author       = {Robert Lipshitz and Sucharit Sarkar},
	title        = {Khovanov homology of strongly invertible knots and their quotients},
	booktitle    = {Frontiers in Geometry and Topology},
	series       = {Proceedings of Symposia in Pure Mathematics},
	volume       = {109},
	publisher    = {American Mathematical Society},
	year         = {2024},
	pages        = {157--182},
}

@article{di2023equivariant,
	title={Equivariant algebraic concordance of strongly invertible knots},
	author={Di Prisa, Alessio},
	journal={Journal of Topology},
	volume={17},
	number={4},
	pages={article no. e70006},
	year={2024},
}

@article{boyle2023classification,
	title={A classification of symmetries of knots},
	author={Boyle, Keegan and Rouse, Nicholas and Williams, Ben},
	journal={arXiv preprint arXiv:2306.04812},
	year={2023}
}

@article{Sarkar2011,
	author  = {Sarkar, Sucharit},
	title   = {Grid diagrams and the {Ozsváth--Szabó} tau-invariant},
	journal = {Mathematical Research Letters},
	volume  = {18},
	year    = {2011},
	pages   = {1239--1257}
}

@article{Mallick2022,
	author  = {Abhishek Mallick},
	title   = {Knot {F}loer homology and surgery on equivariant knots},
	journal = {J. Topol.},
	year    = {2024}
}

@article{dai_mallick_stoffregen,
	title={Equivariant knots and knot {Floer} homology},
	author={Dai, Irving and Mallick, Abhishek and Stoffregen, Matthew},
	journal={Journal of Topology},
	volume={16},
	number={3},
	pages={1167--1236},
	year={2023},
	publisher={Wiley Online Library}
}

@article{JuhaszMarengon2018,
	author  = {Juhász, András and Marengon, Marco},
	title   = {Cobordisms of decorated links},
	journal = {Geometry \& Topology},
	volume  = {22},
	year    = {2018},
	pages   = {2915--3000}
}
\end{document}